\documentclass[a4paper, reqno, english,11pt]{smfart}

\usepackage{amsfonts, amsthm, amsmath, amssymb}

\RequirePackage{mathrsfs} \let\mathcal\mathscr

\numberwithin{equation}{section}

\renewcommand{\d}{\mathrm{d}}
\renewcommand{\phi}{\varphi}

\newcommand{\PP}{\mathbb{P}}
\renewcommand{\AA}{\mathbb{A}}

\newcommand{\ZZ}{\mathbb{Z}}

\newcommand{\NN}{\mathbb{N}}
\newcommand{\QQ}{\mathbb{Q}}
\newcommand{\RR}{\mathbb{R}}

\newcommand{\XX}{\boldsymbol{X}}
\newcommand{\YY}{\boldsymbol{Y}}
\newcommand{\UU}{\boldsymbol{U}}

\newtheorem{thm}{Theorem}[section] 
\newtheorem*{thm*}{Theorem}
\newtheorem{lemma}{Lemma}[section] 

\renewcommand{\leq}{\leqslant}
\renewcommand{\le}{\leqslant}
\renewcommand{\geq}{\geqslant}
\renewcommand{\ge}{\geqslant}

\newcommand{\al}{\alpha}

\newcommand{\e}{\ensuremath{\mathrm e}}
\newcommand{\bet}{\boldsymbol{\eta}}

\newcommand{\bal}{\boldsymbol{\alpha}}
\newcommand{\x}{\mathbf{x}}
\newcommand{\ve}{\varepsilon}
\DeclareMathOperator{\rank}{rank}

\DeclareMathOperator{\rad}{rad}
\DeclareMathOperator{\Pic}{Pic}
\DeclareMathOperator{\Area}{meas}
\DeclareMathOperator{\meas}{meas}
\DeclareMathOperator{\sign}{sign}

\theoremstyle{definition}
\newtheorem*{ack}{Acknowledgements}

\newcommand{\dif}{\mathrm{d}}

\begin{document}

\title[Inhomogeneous cubic congruences]{Inhomogeneous cubic
  congruences and 
rational points on del Pezzo surfaces}

\author{S.\ Baier}
\address{School of Mathematics\\
University of Bristol\\ Bristol\\ BS8 1TW\\ United Kingdom}
\email{stephan.baier@bristol.ac.uk}
\author{T.D.\ Browning}
\address{School of Mathematics\\
University of Bristol\\ Bristol\\ BS8 1TW\\ United Kingdom}
\email{t.d.browning@bristol.ac.uk}

\date{\today}

\begin{abstract}
For given non-zero integers $a,b,q$ we investigate
the density of solutions $(x,y)\in \ZZ^2$ to the binary cubic congruence 
$
ax^2+by^3\equiv 0 \bmod{q},
$
and use it to establish the Manin conjecture for a singular del
Pezzo surface of degree $2$ defined over $\QQ$.
\end{abstract}


\maketitle
\tableofcontents

\section{Introduction}\label{s:intro}

The quantitative arithmetic of low degree del Pezzo surfaces has
received a great deal of attention in recent years. The aim of the
present investigation is to provide a new tool in the analysis of such
questions and to show how it can be used to 
estimate the number of $\QQ$-rational 
points of bounded height on a del Pezzo surface of degree $2$
defined over $\QQ$.  Such surfaces arise as subvarieties $X\subset \PP(2,1,1,1)$
of weighted projective space and are given by equations
of the shape
$$
x_0^2+F(x_1,x_2,x_3)=0,
$$ 
where $F\in \QQ[x_1,x_2,x_3]$ is a quartic form. 
In the classification of del Pezzo surfaces it is those of degree $1$ and
$2$ whose arithmetic remains the most elusive.  

If $X(\QQ)\neq \emptyset$ and 
$H:X(\QQ)\rightarrow \RR_{\geq 0}$ denotes the 
anticanonical height function 
then it is natural to study the counting function
$$
N_{U}(B):=\#\{x \in U(\QQ): H(x) \leq B\},
$$
as $B\rightarrow \infty$, for a Zariski open subset $U\subseteq X$
obtained by deleting the accumulating subvarieties.
These are the exceptional divisors arising from the bitangents of the 
plane quartic curve $F(x_1,x_2,x_3)=0$. There are $28$ of these when
$F$ is non-singular, producing $56$ exceptional curves on $X$.
A well-known conjecture of Manin 
\cite{f-m-t} predicts the existence of constants $c_{X}>0$ and
$\rho_X\in \NN$ such that
\begin{equation}\label{manin}
N_{U}(B) = c_{X} B (\log B)^{\rho_X-1}\big(1+o(1)\big),
\end{equation}
as $B \rightarrow \infty$. Moreover, if 
$\widetilde{X}$ denotes the minimal desingularisation of $X$, with
$\widetilde{X}=X$ if it is non-singular, then 
it is expected that $\rho_X=\rank  (\Pic \widetilde{X})$, 
where $\Pic \widetilde{X}$ is the Picard group of $\widetilde{X}$.
There is a prediction of Peyre \cite{p} concerning the value of
the constant $c_{X}$. These refined conjectures have received a great deal of
attention in the context of del Pezzo surfaces of degree at least $3$,  
an account of which can be found in Browning's treatise
\cite{ferran}. Our success in degree $2$ has been rather
limited however.  

When $X$ is non-singular it follows from work of Broberg
\cite{broberg} that $N_U(B)=O_{\ve, X}(B^{9/4+\ve})$ for any $\ve>0$. 
This argument uses Siegel's lemma to cover the rational points of
height at most $B$ on $X$ with $O(B^{3/2})$ plane sections $H$ defined
over $\QQ$. For each
of these one is left with counting points of bounded height on the
curves $C_H\subset \PP(2,1,1)$ given by 
$$
y_0^2+G(y_1,y_2)=0,
$$
with $G\in \QQ[y_1,y_2]$ a binary form of degree $4$. This one needs
to do uniformly with respect to the coefficients of $G$ and is achieved
via a modification of Heath-Brown's
determinant method \cite{annal}. 
For generic $H$ the curve $C_H$ has genus $1$ and so one expects it
to have very few rational points. Nevertheless it is difficult to demonstrate this
with the requisite degree of uniformity.

In this paper we will consider split singular $X$, by which
we mean that the quartic form $F$ is singular and the 
singularities and exceptional curves of $X$ are all defined over
$\QQ$. According to the classification of Alexeev and Nikulin
\cite{a-n} the minimal desingularisation
$\widetilde{X}$ can then be realised as the blow-up of $\PP^2$  along $7$
$\QQ$-points in ``almost general position''.  As such one verifies
that $\Pic \widetilde{X}\cong \ZZ^8$, so that 
$\rho_X=8$ in the Manin conjecture. 
We will impose the condition that $F$ is reducible and we will assume 
that $X\subset
\mathbb{P}(2,1,1,1 )$ takes the shape 
\begin{equation}
  \label{eq:X}
  x_0^2+x_1x_2^3+x_1^3x_3=0.
\end{equation}
This defines a del Pezzo surface 
of degree $2$ 
with a unique exceptional curve $x_0=x_1=0$ and a unique singularity $[0,0,0,1]$. 
According to the classification of Arnol'd \cite{arnold} a hypersurface has 
a simple $\mathbf{E}_7$ singularity if it can be put in the local
normal form
$
z_1^3+z_1z_2^3+z_3^2+\cdots + z_k^2,
$
where $k-1$ is the dimension of the hypersurface. 
Dividing through by $x_2^4$ and making a
change of variables one is directly led to this normal form,
so that $X$ has a simple $\mathbf{E}_7$ singularity. 
It follows from work of Ye \cite[Lemma 4.6]{Ye} that there is only one such
surface up to isomorphism. We let $U$ be the Zariski open subset
formed by deleting the curve $x_0=x_1=0$ from $X$. 

When one returns to the above argument involving plane sections, the
situation is more favourable since  the resulting curves
$C_H$ generically define elliptic curves with rational
$2$-torsion. In this way one can show that
$N_U(B)=O_{\ve}(B^{3/2+\ve})$ for any $\ve>0$.  Our goal is to show
how the full Manin conjecture \eqref{manin} can be 
established for this particular surface using analytic
number theory. 
We will establish the following result.

\begin{thm}\label{main}
We have
$$
N_{U}(B)=c B(\log B)^7(1+o(1)),
$$
as $B\rightarrow \infty$,  
where $c>0$ is the constant predicted by Peyre.
\end{thm}

The investigations of Derenthal and Loughran \cite{cox, dl}
show that $X$ is neither toric nor an equivariant compactification of
$\mathbb{G}_a^2$. Thus Theorem \ref{main} is not a special case of
\cite{toric} or \cite{CLT}. 
The proof of our theorem relies on a passage to the 
universal torsor above the minimal desingularisation $\widetilde{X}$ of $X$,
which in this setting is a subset of the affine hypersurface 
$T\subset \AA^{11}$, given by the equation
\begin{equation} \label{count}
\eta_1^2\eta_2\alpha_1^3+\eta_7\alpha_2^2+
\eta_4\eta_5^2\eta_6^3\eta_8^4\alpha_3=0.
\end{equation}
The idea is to establish a bijection between
$U(\QQ)$ and a suitable subset of $T(\ZZ)$.
This step underpins many proofs of the Manin conjecture, 
such as that found in work of la Bret\`eche, Browning and Derenthal
\cite{e6} 
dealing with the split cubic surface
of  singularity type $\mathbf{E}_6$.  Here, as there, it 
is useful to view the torsor equation as a congruence 
$\eta_1^2\eta_2\alpha_1^3+\eta_7\alpha_2^2\equiv 0 \bmod{
\eta_4\eta_5^2\eta_6^3\eta_8^4}$, with $\alpha_1,\alpha_2$ being thought of 
as the main variables.
Regrettably one finds  that the arguments developed
in \cite{e6} no longer bear fruit
in the present setting.  Rather one is forced to consider in
general terms the counting function
\begin{equation}
  \label{eq:Mq}
M(B,\XX,\YY;a,b;q):=
\#\left\{
(x,y)\in\ZZ^2: 
\begin{array}{l}
0<x\leq \XX, ~|y|\leq \YY, ~(xy,q)=1, \\
ax^2+by^3\equiv 0 \bmod{q}, ~|ax^2+by^3|\leq qB
\end{array}
\right\},
\end{equation}
for $B\geq 2$, $\boldsymbol{X},\boldsymbol{Y}\geq 1$ and non-zero integers
$a,b,q$ such that $q>0$.  
In our case of interest we have 
$a=\eta_7$, $x=\alpha_2$, $b=\eta_1^2\eta_2$, $y=\alpha_1$ and
$q=\eta_4\eta_5^2\eta_6^3\eta_8^4$.
One seeks an asymptotic formula for 
$M(B,\boldsymbol{X},\boldsymbol{Y};a,b;q)$ 
which is completely uniform in the relevant parameters. This will 
ultimately be achieved in \S \ref{s:final}, where it is recorded as 
Theorem~\ref{t:final}.   A substantially easier problem is to produce a
good upper bound for the counting function in which the
condition 
$|ax^2+by^3|\leq qB$ is dropped. Let us denote 
by $M(\XX,\YY;a,b;q)$ this counting function. During the course of our
argument we will be led to the following result in \S
\ref{s:proof-upper}, 
which we record here for ease of use.

\begin{thm}\label{cor:0}
Let $\ve>0$, let $\XX,\YY\geq 1$ 
and assume that $(ab,q)=1$. Then we have 
\begin{align*}
M(\XX,\YY;a,b;q)
\ll~&
(q\XX\YY)^\ve
 \Bigg\{\frac{\XX \YY}{q}
+ 
\frac{a^{1/2}b^{-1/4}\XX+
a^{1/4} \XX^{1/2}\YY^{3/4}}{q^{1/2}}
\\
& +
(s(q) s_1(q) q)^{1/2} \left(
  \frac{b^{1/2}\YY}{\XX} +
\frac{a^{1/2}}{\YY^{1/2}}\right)\Bigg\},
\end{align*}
where
\begin{equation}\label{eq:s-s1}
s(q):=\prod_{p\| q}p, \quad 
s_1(q):=\prod_{\substack{p^\nu \| q \\ 2\nmid \nu}}p.
\end{equation}
\end{thm}

The implied constant in this estimate is allowed to depend on the
choice of small parameter $\ve>0$, a convention that we adhere to 
in all of our estimates. Our estimate 
for $M(\XX,\YY;a,b;q)$ is  sharpest when $q$ has small square-free
kernel or when $\YY$ is small compared with $\XX$.  It is the cornerstone
of our entire investigation. It should be noted that when $q$ is
square-free 
and $a=b=1$ Theorem~\ref{cor:0} does not give anything sharper
than what is available through  the work of Pierce
\cite{pierce}.

We proceed to give a simplified description of the method
behind Theorem \ref{cor:0}.  The 
$x,y$ appearing in \eqref{eq:Mq} are constrained to lie in a certain
region and our first task will be to 
cover this region by small boxes and to 
approximate their characteristic
functions by smooth weights. This facilitates an application of the Poisson
summation formula, which we invoke after breaking the sums over $x$ and $y$
into residue classes modulo $q$. This transformation leads to
expressions involving exponential sums of the form
$$
E(m,n;q)=
\sum\limits_{\substack{x=1\\ (x,q)=1}}^q
\e\left(\frac{mx^3-nx^2}{q}\right),
$$
for $m,n\in \ZZ$, where $e(\cdot):=\exp(2\pi i \cdot)$.
If one were now to estimate these sums directly 
one would retrieve an estimate of the sort obtained by 
Pierce \cite{pierce}.   In the setting of Theorem \ref{main}, however, 
this 
would only yield a  final upper bound of the form $N_U(B)=O_\ve(B^{3/2+\varepsilon})$.
A key point in the method is to evaluate the sums 
$E(m,n;q)$ explicitly for power-full
moduli $q$, a situation that explicitly arises in the application to Theorem
\ref{main} since then $q=\eta_4\eta_5^2\eta_6^3\eta_8^4$ contains high
powers. This is carried out in \S \ref{s:expsum}, where an easy 
multiplicativity property renders it sufficient to study 
$E(m,n;p^t)$ for primes $p$ and suitable $t\geq 2$. 
It turns out that considerable labour is required to deal with the
primes $p=2$ and $p=3$ and the reader may be inclined to take the results of
this section on faith at a first reading. 

After explicitly evaluating the exponential sums $E(m,n;q)$ at power-full moduli 
we encounter
terms of the shape
$$
\e\left(\frac{\overline{m}^2n^3}{q}\right),
$$
for $m,n\in \ZZ$ such that $(m,q)=1$ and where $\overline{m}$ is the
multiplicative inverse of $m$ modulo $q$. 
We need to sum these up non-trivially. Our second key innovation is to flip
the numerator and denominator, 
using the familiar identity 
\begin{equation} \label{genflip}
\e\left(\frac{\overline{l}}{k}\right) \e\left(\frac{\overline{k}}{l}\right)= \e\left(\frac{1}{kl}\right),
\end{equation}
for any non-zero integers $k,l,\overline{k},\overline{l}$ such that 
$k\overline{k}+l\overline{l}=1$.
This has the desired effect of reducing the size of the denominator
drastically. Next, we break up the summation over $n$ into residue
classes modulo the new denominator and use Poisson summation in $n$
again. This time  we encounter new cubic exponential sums and cubic
exponential integrals which we need to estimate non-trivially. 
This leads to
an asymptotic estimate for the number of solutions of the congruence
$ax^2+by^3\equiv 0 \bmod{q}$ in small boxes. 
The final step is to sum up all these
contributions. In fact this summation is also  carried out
non-trivially,  with parts of the  averaging process 
replaced by an integration, in order to take advantage of 
extra cancellations.  While this leads to substantial extra work it is 
nonetheless essential for obtaining the asymptotic formula in Theorem
\ref{main}.

Aside from its intrinsic utility in the proof of Theorem~\ref{main} we
can use Theorem~\ref{cor:0} and its refinement Theorem \ref{t:final} 
to tackle other questions in Diophantine
geometry.   Firstly, an inspection of the various universal torsors
calculated by Derenthal \cite{cox}
arising  in the theory of split singular del Pezzo surfaces, shows
that the underlying counting function precisely 
matches \eqref{eq:Mq} whenever its singularity type is maximal. Although we will not present details
here, it is possible to provide independent proofs of the Manin
conjecture for such cases, albeit 
with weaker error terms than are already available. This
includes  the $\mathbf{E}_6$ cubic surface 
considered by la Bret\`eche, Browning and Derenthal \cite{e6}, 
the $\mathbf{D}_5$ degree $4$ del Pezzo surface considered by la Bret\`eche and
Browning \cite{dp4-d5} and the  $\mathbf{A}_4$ degree
$5$ del Pezzo surface. The latter two surfaces
are equivariant compactifications of $\mathbb{G}_a^2$ by 
\cite{dl} and so are covered by the investigation  of 
Chambert-Loir 
and Tschinkel \cite{CLT}.

A second and rather different application of Theorem \ref{cor:0} lies
in the theory of elliptic curves over $\QQ$. Such curves 
may be brought into
Weierstrass form
$$
E_{A,B}: \quad y^2=x^3+Ax+B,
$$
for $A,B\in \ZZ$ with non-zero discriminant
$-16(4A^3+27B^2)=-16\Delta_{A,B}$, say. 
It is presently unknown whether there are infinitely many $E_{A,B}$
for which $\Delta_{A,B}$ is prime. An easier questions concerns the
square-freeness of $\Delta_{A,B}$. 
Let $\max\{|A|^{1/4}, |B|^{1/6}\}$ be the exponential height of
$E_{A,B}$ and let $\mu$ be the M\"obius function. 
We maintain  the conventions $\mu(0)=0$ and
$\mu(-n)=\mu(n)$. A measure of the density of
elliptic curves with square-free discriminant  is achieved by studying
the quantity
$$
S(X):= \sum_{\substack{(A,B)\in \ZZ^2\\ 
|A|\leq X^4, |B|\leq X^6}} \mu^2(\Delta_{A,B}).
$$
We will use Theorem~\ref{cor:0} to establish
the following result in \S \ref{s:ec}.

\begin{thm}\label{t:ec}
For $q\in \NN$ let
$\varrho(q):= \#\{\alpha, \beta \bmod{q}: \Delta_{\alpha,\beta}\equiv 0\bmod{q}\}$. 
Then for any $\ve>0$ we have 
$$
S(X)=4X^{10}\prod_{p} \left( 1-\frac{\varrho(p^{2})}{p^{4}} \right) +O(X^{7+\ve}).
$$
\end{thm}

For comparison Wong \cite[Proposition 6]{wong} has shown a similar
asymptotic formula but with only a logarithmic saving in the error
term. In fact, as we shall see in \S \ref{s:ec}, 
an  adaptation of an argument due to
Estermann \cite{estermann} would permit a power saving but only leads
to an error  term of order $X^{8+\ve}$.  
Through M\"obius inversion the problem is 
to count solutions to the congruence 
$\Delta_{A,B}=4A^3+27B^2\equiv 0\bmod{k^2}$ for square-free integers $k$. 
Estermann's approach helps us to deal with the
contribution from both small and large $k$. Theorem~\ref{cor:0} is
the key ingredient in the treatment of medium $k$.

\begin{ack}
Some of this work was done while the authors were visiting the 
{\em Institute for Advanced Study} in Princeton, 
the hospitality and financial
support of which is gratefully acknowledged. 
While working on this paper the  authors were 
supported by EPSRC grant number \texttt{EP/E053262/1}. The authors are
grateful to Ulrich Derenthal for drawing their attention to the
$\mathbf{E}_7$ del Pezzo surface of degree $2$ and  to both Pierre Le Boudec 
and the referee for useful comments on an earlier version. 
\end{ack}

\section{Elliptic curves with square-free discriminant}\label{s:ec}

In this section we show how Theorem~\ref{t:ec} follows from Theorem
\ref{cor:0}.  Using the M\"obius function to detect the 
square-freeness condition we may write
$$
S(X)= \sum_{k=1}^\infty\mu(k)
\sum_{\substack{(A,B)\in \mathcal{A}(X)\\ k^2\mid 4A^3+27B^2}}1,
$$
where $\mathcal{A}(X):=\{
(A,B)\in \ZZ^2: |A|\leq X^4, |B|\leq X^6\} $.
We will estimate this inner sum differently according to the size of
$k$. For parameters $\xi_1<\xi_2$ let us write
$S_i(X)$ for the overall contribution to $S(X)$ from $k$
in the interval $I_i$, with
$$
I_1=(0,\xi_1]\cap\NN, \quad 
I_2=(\xi_1, \xi_2]\cap \NN, \quad 
I_3=(\xi_2,\infty)\cap \NN.
$$
Finally for $q\in \NN$ we recall the definition
of $\varrho(q)$ from Theorem~\ref{t:ec}. 
This is a multiplicative function of $q$ and it is easy to see
that 
$\varrho(p^2)=O(p^{2})$ for any prime $p$.   
Hence we have $\varrho(k^2)=O(k^{2+\ve})$ for any
square-free $k\in \NN$. 

Beginning with the contribution from small $k$, the idea is to break the sum over
$A,B$ into congruence classes modulo $k^2$. 
Beginning with the sum over $B$ one sees that 
\begin{align*}
S_1(X)
&= \sum_{k\leq \xi_1}
\mu(k)
\sum_{|A|\leq X^4}
\sum_{\substack{\beta \bmod{k^2}\\\Delta_{A,\beta}\equiv 0 \bmod{k^2}}}
\left(
\frac{2X^6}{k^2}+O(1)\right),
\end{align*}
where we recall that $\Delta_{A,B}=4A^3+27B^2$. For $n \in \ZZ$ and
square-free $k\in \NN$ let $\gamma_{k^2}(n)$ 
denote the number of incongruent
solutions $\beta$ modulo $k^2$ of $n\equiv 27\beta^2\bmod{k^2}$. 
It is trivial to see that $\gamma_{k^2}(n)\ll k^\ve(k,n)$ 
for
any $\ve>0$.  We may now write
\begin{align*}
S_1(X)
&= 2X^6\sum_{k\leq \xi_1}
\frac{\mu(k)}{k^2}
\sum_{|A|\leq X^4}\gamma_{k^2}(-4A^3) 
+O(\xi_1^{2+\ve}+X^4\xi_1^{1+\ve}),
\end{align*}
where the error term $\xi_1^{2+\ve}$ comes from $A=0$.
Next we claim that 
\begin{equation}
  \label{eq:onion}
  \sum_{|A|\leq z}\sigma_{k^2}(-4A^3)= \frac{2z\rho(k^2)}{k^2} +O(k^{1+\ve}),
\end{equation}
for any square-free $k \in \NN$. Once armed with this it is then
straightforward to see that 
\begin{equation}
  \label{eq:S1}
S_1(X)  
=4
X^{10}
\prod_{p} \left( 1-\frac{\varrho(p^{2})}{p^{4}} \right)
+O\left(\frac{X^{10}}{\xi_1^{1-\ve}}+
\xi_1^{2+\ve}+X^4\xi_1^{1+\ve}+
X^6 \xi_1^{\ve}\right),
\end{equation}
on extending the summation over $k$ to infinity.
We establish the claim using a simple argument involving exponential
sums. 

Breaking the sum over $A$ into residue classes modulo $k^2$ we
find that
\begin{align*}
 \sum_{|A|\leq z}\sigma_{k^2}(-4A^3)
&=\sum_{\substack{\alpha, \beta
 \bmod{k^2}\\ \Delta_{\alpha,\beta}\equiv 0 \bmod{k^2}}} 
\sum_{\substack{|A|\leq z\\ A\equiv \alpha \bmod{k^2}}}1
\\
&=\frac{1}{k^2}\sum_{\ell \bmod{k^2}}
\sum_{\substack{\alpha, \beta
 \bmod{k^2}\\ \Delta_{\alpha,\beta}\equiv 0 \bmod{k^2}}} 
\sum_{|A|\leq z} \e \left(\frac{\ell(\alpha-A)}{k^2}\right).
\end{align*}
The contribution from $\ell=0$ is clearly 
$$
\frac{\rho(k^2)}{k^2}\left(2z+O(1)\right) = 
\frac{2z\rho(k^2)}{k^2}+O(k^\ve),
$$
which is satisfactory. Likewise 
one finds that the contribution from non-zero $\ell$
is
$$
\ll 
\frac{1}{k^2}\sum_{\substack{-k^2/2< \ell \leq k^2/2\\ \ell\neq 0}}
\tau(k^2;\ell) \min\left\{ z , \left\|\frac{\ell}{k^2}\right\|^{-1}\right\}
\ll 
\sum_{1\leq  \ell \leq k^2/2}
\frac{\tau(k^2;\ell) }{\ell},
$$
where 
$$
\tau(q;\ell)=
\sum_{\substack{\alpha, \beta
 \bmod{q}\\ \Delta_{\alpha,\beta}\equiv 0 \bmod{q}}} \e \left(
\frac{\ell \alpha}{q}\right).
$$ 
The exponential sum $\tau(q;\ell)$ satisfies a basic
multiplicativity property in $q$ rendering it sufficient to understand when $q$ is
a prime power. In this way one easily concludes that
$\tau(k^2;\ell)\ll k^{1+\ve}(k,\ell)$. Thus the contribution from
non-zero $\ell$ is also seen to be satisfactory for \eqref{eq:onion},
after redefining $\ve$.

Turning to the large values of $k$ we write the congruence as an
equation and note that 
\begin{align*}
|S_3(X)|
&\leq 
\sum_{0<|m|\ll X^{12}/\xi_2^2}
\# \left\{
(A,B,k)\in \mathcal{A}(X)\times I_3: 
4A^3+27B^2=k^2m
\right\}.
\end{align*}
To estimate the summand we fix $A$ and consider it as a problem about
counting the representations of $4A^3$ by the binary quadratic form 
$mx^2-27y^2$.  The classical argument of Estermann \cite{estermann} shows that
there are $O((Am)^\ve)$ solutions $(x,y)$ for any $\ve>0$, whence
\begin{equation}
  \label{eq:S3}
S_3(X)\ll \frac{X^{16+\ve}}{\xi_2^2}.
\end{equation}
At this point  we can recover a preliminary estimate for $S(X)$ by
taking $\xi_1=\xi_2=X^{4}$. This gives a version of Theorem~\ref{t:ec}
with the weaker error term $O(X^{8+\ve})$, as remarked in the
introduction. 

We now come to the treatment of the middle range for $k$. 
Thus we have 
\begin{align*}
|S_2(X)|
&\leq  \sum_{\xi_1<k\leq \xi_2}
\mu^2(k) \# \{(A,B)\in \mathcal{A}(X): 4A^3+27B^2\equiv
0\bmod{k^2}\}.
\end{align*}
For given $k$ let us write 
$k=k_2k_3k' $ where $k_2=(k,2)$, $k_3=(k,3)$ and $k'$ is
coprime to $6$.  It readily follows that 
$k_2\mid B$ and $k_3\mid A$ in the summand. 
Making the change of variables $A=k_3A'$ and $B=k_2B'$
we deduce that
\begin{align*}
|S_2(X)|
&\leq  \sum_{\substack{\xi_1<k\leq \xi_2\\
k=k_2k_3k' }}
\mu^2(k) \# \left\{(x,y)\in \ZZ^2: 
\begin{array}{l}
k_3|A'|\leq X^4, ~k_2|B'|\leq X^6 \\
a^3B'^2+b^2A'^3\equiv 0\bmod{k'^2}
\end{array}
\right\},
\end{align*}
where $a=3/k_3$ and $b=2/k_2$.  
In particular it follows that $(ab,k')=1$ and $a,b\leq 3$.
We will need to account for possible common factors of $A'B'$ and
$k'^2$.   Drawing out the greatest common divisor of $B'$ and $k'$ we write 
$B'=hx$ and $k'=h\ell$, with $(x,\ell)=1$. It easily follows 
from the square-freeness of $k$ that
$h\mid A'$ and we can write $A'=hy$ with $(hxy,\ell)=1$. 
Hence 
$$
|S_2(X)|
\leq  \sum_{\substack{\xi_1<k\leq \xi_2\\
k=k_2k_3h\ell}}
\mu^2(k)
M\left(
\frac{X^6}{h},\frac{X^4}{h};a^3,b^2h;\ell^2\right),
$$
in the notation of \S \ref{s:intro}, with $(a^3b^2h,\ell^2)=1$.

Everything is now in place for an application of 
Theorem~\ref{cor:0}. 
 On noting that 
$s(\ell^2)=s_1(\ell^2)=1$, we deduce that 
\begin{align*}
M\left(\frac{X^6}{h},\frac{X^4}{h};a^3,b^2h;\ell^2\right)
\ll 
(\ell X)^\ve
 \left\{
\frac{X^{10}}{h^2\ell^2}
+ 
\frac{X^6}{h\ell}
+\frac{h^{1/2}\ell}{X^2}\right\}.
\end{align*}
Inserting  this into our bound for $S_2(X)$ we conclude that
$$
S_2(X)
\ll 
(\xi_2 X)^\ve \left\{
\frac{X^{10}}{\xi_1}
+ 
X^6
+\frac{\xi_2^{2}}{X^2}\right\}.
$$
We must now combine this with \eqref{eq:S1}, \eqref{eq:S3} and a
suitable choice of $\xi_1,\xi_2$. Taking $\xi_1=X^{3}$ and
$\xi_2=X^{9/2}$ 
readily leads to the statement of Theorem~\ref{t:ec}.

\section{Technical preliminaries}

\subsection{Gaussian weights}\label{s:gauss}

There are numerous ways in which one can approximate the
characteristic function of intervals using smooth weights. 
In our work, which will involve repeated applications of the Poisson
summation formula, it will be useful to have weights which transform
well under the Fourier
transform. We are naturally drawn to construct weight functions from
the Gaussian 
\begin{equation} \label{gaussiandef}
\Gamma(x):=\exp\left(-\pi x^2\right).
\end{equation}
Note that the usual Gamma function does not occur anywhere in our work and so we trust that this choice of notation doesn't cause confusion. 
Let $\chi$ denote the characteristic function 
of the interval $[-1/2,1/2)$. We will approximate this using the
weights 
$$
\Phi_{\pm}(x):=L\int\limits_{-1/2\mp \Delta^{1/2}}^{1/2 \pm
  \Delta^{1/2}} 
\Gamma\left((x-\mu)L\right) \dif \mu \pm \varepsilon_1 \Gamma(x),
$$
for any $L>2$, where 
$
\Delta:=L^{-1}$ and $\varepsilon_1:=24\exp(-L).$ 
Clearly, $\Phi_+$ and $\Phi_-$ are smooth (infinitely differentiable)
functions that have rapid decay at $-\infty$ and $\infty$. Thus 
$\Phi_{\pm}(x)=O\left(x^{-N}\right)$ as $|x|\rightarrow \infty$
for any fixed $N>0$. We proceed to show how they 
approximate $\chi$ in the following result.

\begin{lemma} \label{weightslemma} Assume that $L>2$. Then
$
\Phi_{-}(x)\le \chi(x)\le \Phi_{+}(x)
$
for all $x\in \mathbb{R}$. Moreover, we have 
\begin{equation} \label{closeto1}
\int\limits_{-\infty}^{\infty} \Phi_{\pm}(x) \dif x = 1\pm (2\Delta^{1/2}+\varepsilon_1). 
\end{equation}
\end{lemma}

\begin{proof} The equation \eqref{closeto1} is easily proved as
  follows.  We have 
\begin{align*} 
  \int\limits_{-\infty}^{\infty} \Phi_{\pm}(x) \dif x &=
   \int\limits_{-1/2\mp \Delta^{1/2}}^{1/2 \pm \Delta^{1/2}}
   \left(L\int\limits_{-\infty}^{\infty} \Gamma\left((x-\mu)L\right) 
\dif x\right) \dif \mu \pm \varepsilon_1 
   \int\limits_{-\infty}^{\infty} \Gamma\left(x\right) \dif x\\ 
   &=  \int\limits_{-1/2\mp \Delta^{1/2}}^{1/2 \pm \Delta^{1/2}}
   \left(\int\limits_{-\infty}^{\infty} \Gamma\left(y\right) \dif
   y\right) \dif \mu \pm \varepsilon_1 \\ 
  &=  1\pm (2\Delta^{1/2}+\varepsilon_1),
\end{align*}
as claimed.

Let $I=[-1/2,1/2)$. 
Next, we want to show that $\chi(x)\le \Phi_+(x)$ for all $x\in
\mathbb{R}$. If $x\not\in I$ then obviously
$\Phi_+(x)>0$. If $x \in I$ then 
\begin{align*}
\Phi_{+}(x) &\ge 
L\int\limits_{x-\Delta^{1/2}}^{x+\Delta^{1/2}}
\Gamma\left((x-\mu)L\right) \dif \mu
+\varepsilon_1\exp\left(-\frac{\pi}{4}\right) \\ 
&=  \int\limits_{-L^{1/2}}^{L^{1/2}} \Gamma\left(-y\right) \dif y
+ \varepsilon_1\exp\left(-\frac{\pi}{4}\right)\\ 
&=   1-2\int\limits_{L^{1/2}}^{\infty} \Gamma\left(y\right) \dif
y +  \varepsilon_1\exp\left(-\frac{\pi}{4}\right)\\ 
&=   1-\int\limits_{\pi L}^{\infty}
\frac{\exp\left(-z\right)}{\sqrt{\pi z}} \dif z +
\varepsilon_1\exp\left(-\frac{\pi}{4}\right)\\ 
&\ge  1-\exp(-L) + \frac{\varepsilon_1}{3}= 1+7 \exp(-L)\ge  1. 
\end{align*}
Hence we have established that $\chi(x)\le \Phi_+(x)$ for all $x\in \mathbb{R}$.

Finally we want to show that $\Phi_-(x)\le \chi(x)$ for all $x\in
\mathbb{R}$. If $x\in I$ then 
\begin{equation*} 
\Phi_{-}(x) < L\int\limits_{-\infty}^{\infty}
\Gamma\left((x-\mu)L\right) 
\dif \mu = \int\limits_{-\infty}^{\infty} \Gamma\left(y\right) \dif y = 1.
\end{equation*}
If $1/2\le |x|\le 1$ then
\begin{align*}
\Phi_{-}(x) &\le 
L\left(\int\limits_{-\infty}^{\infty}
  -\int\limits_{x-\Delta^{1/2}}^{x+\Delta^{1/2}}\right) \Gamma\left(
(x-\mu)L\right) \dif \mu -  
\varepsilon_1\exp(-\pi) \\
&=   2\int\limits_{L^{1/2}}^{\infty} \Gamma\left(y\right) \dif y - 
\varepsilon_1\exp(-\pi) \\
&=   \int\limits_{\pi L}^{\infty} 
\frac{\exp\left(-z\right)}{\sqrt{\pi z}} \dif z -
\varepsilon_1\exp(-\pi) \\ 
&\le  \exp(-L) - \frac{\varepsilon_1}{24} = 0. 
\end{align*}
If $|x|>1$ then for all $\mu \in \left[-1/2+\Delta^{1/2},1/2-\Delta^{1/2}\right]$ we have 
\begin{equation*}
|x-\mu|\ge |x|-|\mu|>|x|-1/2>|x|/2.
\end{equation*}
Hence
\begin{align*}
\Phi_{-}(x) &\le 
L \Gamma\left(\frac{x}{2} \cdot L\right)
\int\limits_{-1/2+\Delta^{1/2}}^{1/2-\Delta^{1/2}}  \dif \mu - 
\varepsilon_1 \Gamma(x) 
\le  L \exp\left(-\pi \frac{x^2L^2}{4}\right)
 - 
\varepsilon_1 \exp\left(-\pi x^2\right). 
\end{align*}
By taking logarithms one easily verifies that the last line is $\le 0$ if $|x|>1$ and 
\begin{equation*} 
\log \frac{L}{24}+L \le \pi\left(\frac{L^2}{4}-1\right),
\end{equation*}
which is true for all $L>2$. This completes the proof of $\Phi_-(x)\le \chi(x)$ for all $x\in \mathbb{R}$ and therefore the proof of the lemma.
\end{proof}

From Lemma \ref{weightslemma} one can easily deduce a similar result
for the characteristic function of a general interval $[a,b)$ by
making the change of variables $x\rightarrow (x-a)/(b-a)-1/2$, which
maps the interval $[a,b)$ bijectively onto $[-1/2,1/2)$. 
We include this
result here because it may be useful for future applications.  

\begin{lemma} 
Let $a,b$ be real numbers with $a<b$. Denote the characteristic function of the interval $[a,b)$ by  $\chi_{[a,b)}$. Suppose that $L>2$ and set 
$$
\Delta:=L^{-1},  \quad  \varepsilon_1:=24\exp(-L).
$$ 
Define
$$
\Phi^{a,b}_{\pm}(x):=L
\int\limits_{-1/2\mp \Delta^{1/2}}^{1/2 \pm \Delta^{1/2}} 
\Gamma\left(\left(\frac{x-a}{b-a}-\frac{1}{2}-\mu\right)L\right) \dif
\mu \pm \varepsilon_1 \Gamma
  \left(\frac{x-a}{b-a}-\frac{1}{2}\right). 
$$
Then we have 
$
\Phi^{a,b}_{-}(x)\le \chi_{[a,b)}(x)\le \Phi^{a,b}_{+}(x) 
$
for all $x\in \mathbb{R}$ and 
$$
\int\limits_{-\infty}^{\infty} \Phi_{\pm}^{a,b}(x) \dif x = \left(1\pm
  (2\Delta^{1/2}+\varepsilon_1)\right) (b-a).  
$$
Moreover $\Phi^{a,b}_+$ and $\Phi^{a,b}_-$ are infinitely
differentiable functions that have rapid decay at $\pm\infty$.
\end{lemma}

\subsection{Estimation of exponential integrals}\label{s:airy}

Recall the definition \eqref{gaussiandef} of the Gaussian function. We
begin with a general upper bound for exponential integrals weighted by
the Gaussian.

\begin{lemma}\label{lem:IK}
Let $f:\RR\rightarrow \RR$ be a smooth function and suppose that there
exists $\Lambda>0$ and $k \in \NN$ such that $|f^{(k)}(x)|\geq
\Lambda$ for every $x \in \RR$. Then we have 
$$
\int_{-\infty}^\infty \Gamma(x)\e\left(f(x)\right)\d x \ll \Lambda^{-1/k}.
$$
\end{lemma}

\begin{proof}
It will be sufficient to prove the corresponding bound for the integral over $[0,\infty)$ since the treatment of the integral over $(-\infty,0]$ is similar. 
For any $a>0$ it follows from integration by parts that 
$$
\int_{0}^{a} \Gamma(x)\e\left(f(x)\right)\d x =
\Gamma(a)\int_{0}^{a} \e\left(f(x)\right)\d x -
\int_{0}^{a} \Gamma'(y)\int_0^y \e\left(f(x)\right)\d x\d y.
$$
Now it follows from 
Lemma 8.10 in \cite{IwKo} and the hypotheses of the lemma that 
$$
\left| 
\int_0^y \e\left(f(x)\right)\d x \right|
\leq k2^k \Lambda^{-1/k},
$$
for any $y\leq a$. Hence it readily follows that 
$$
\left|\int_{0}^{a} \Gamma(x)\e\left(f(x)\right)\d x \right| \leq 
k2^k
\Lambda^{-1/k} \left(\Gamma(a)+
\int_{0}^{a} |\Gamma'(y)| \d y\right)=
k2^k
\Lambda^{-1/k}.
$$
Taking the limit $a\rightarrow \infty$ we are easily led to the desired bound. This completes the
proof. \end{proof}

The special case $f(x)=-\beta x^3-\alpha x$ will feature quite heavily
in our work, corresponding to weighted Airy--Hardy integrals. 
For any $\alpha,\beta \in \RR$ such that $\beta>0$ we set 
\begin{equation} \label{Fdef1}
F(\alpha,\beta):=\int\limits_{-\infty}^{\infty} \Gamma(x)\e(-\beta
x^3-\alpha x) \dif x. 
\end{equation}
We begin by recording the trivial estimate
\begin{equation} \label{Fbound2} 
|F(\alpha,\beta)|\leq 1.
\end{equation}
Furthermore, applying Lemma \ref{lem:IK} with $k=3$, we deduce
that 
\begin{equation} \label{another}
F(\alpha,\beta)\ll \beta^{-1/3}.
\end{equation}  
Moreover, if $\alpha>0$, then for all $x \in \RR$ we have
$$
\frac{\dif(-\beta x^3-\alpha x)}{\dif x} = -3\beta x^2-\alpha \le -\alpha<0.
$$
An application of Lemma \ref{lem:IK} with $k=1$ now gives
\begin{equation} \label{Fbound1}
F(\alpha,\beta)\ll \frac{1}{\alpha} \quad \mbox{for } \alpha>0.
\end{equation}

We henceforth assume that $\alpha<0$. In this case we have
stationary points  which give 
a main term contribution. We write
\begin{equation} \label{Fsplit2}
\begin{split}
F(\alpha,\beta)
&= \int\limits_{0}^{\infty} \Gamma(x)\e(-\beta x^3+|\alpha| x) \dif x +
\int\limits_{0}^{\infty} \Gamma(-x)\e(\beta x^3-|\alpha| x) \dif x\\
&= \overline{\int\limits_{0}^{\infty} \Gamma(-x)\e(\beta x^3-|\alpha|
  x) \dif x} + 
\int\limits_{0}^{\infty} \Gamma(-x)\e(\beta x^3-|\alpha| x) \dif x,
\end{split}
\end{equation}
where we take into account that our function $\Gamma$ is even.
We make a change of variables $t=\beta^{1/3}x$, getting
\begin{equation}\label{Hardy-Airy} 
\int\limits_{0}^{\infty} \Gamma(-x)\e(\beta x^3-|\alpha| x) \dif x = \beta^{-1/3}
\int\limits_{0}^{\infty} \Gamma(-\beta^{-1/3}t)\e(t^3-|\alpha| \beta^{-1/3} t) \dif t.
\end{equation}  
To estimate the integral on the right-hand side we use 
Theorem 2.2 in \cite{Ivic}, which we record here for convenience.

\begin{lemma}[Stationary phase with weights] \label{stationary} Let
  $f(z)$, $g(z)$ be two functions of the complex variable $z$ and
  $[a,b]$ be a real interval such that the following hold.
\begin{itemize} 
\item[(a)] 
For $a\le x\le b$ the function $f(x)$ is real and $f''(x)>0$.
\item[(b)] 
For a certain positive differentiable function $\mu(x)$, defined on
$[a,b]$, $f(z)$ and $g(z)$ are analytic for $a\le x\le b$,
$|z-x|\le \mu(x)$. 
\item[(c)] 
There exist positive functions $F(x)$, $G(x)$ defined on $[a,b]$ such
that for $a\le x\le b$, $|z-x|\le \mu(x)$ we have 
$$
g(z)\ll G(x), \quad f'(z)\ll \frac{F(x)}{\mu(x)}, \quad
|f''(z)|^{-1}\ll \frac{\mu(x)^2}{F(x)},  
$$
and the implied constants are absolute. 
\end{itemize}
Let $k\in \RR$ and
if $f'(x)+k$ has a zero in $[a,b]$ denote it by $x_0$. Let the values
of $f(x)$, $g(x)$ and so on, at $a$, $x_0$, and $b$ be characterised
by the suffixes $a$, $0$ and $b$, respectively. 
Then, for some
absolute constant $C>0$, we have 
\begin{equation} 
\label{einintegral} 
\begin{split}
\int\limits_a^b g(x)\e\left(f(x)+kx\right)\dif x 
=~&
\frac{g_0}{\sqrt{f_0''}}\cdot
\e\left(f_0+kx_0+\frac{1}{8}\right) 
+O\left(G_0\mu_0F_0^{-3/2}\right)
\\  &  
+O\left(\int\limits_a^b G(x)\exp\left(-C|k|\mu(x)-CF(x)\right)(\dif
  x+|\dif \mu(x)|)\right)\\ 
 & +O\left(G_a\left(|f_a'+k|+\sqrt{f_a''}\right)^{-1}+
G_b\left(|f_b'+k|+\sqrt{f_b''}\right)^{-1}\right).
\end{split}
\end{equation}
If $f'(x)+k$ has no zero in $[a,b]$, then the terms involving $x_0$ are to be omitted. 
\end{lemma}

We apply Lemma \ref{stationary} with $0<a<b$ and
\begin{align*}
f(z)&=z^3, \quad g(z)=\exp(-\pi \beta^{-2/3} z^2),\quad  
\mu(x)=x/2,\quad F(x)=x^3,\\ 
G(x)&=\exp(-\pi \beta^{-2/3} x^2/2),\quad k=-|\alpha|\beta^{-1/3},
\end{align*}
and let $a$ tend to $0$ and $b$ to $\infty$. Here we note that the choice of $G(x)$ above is possible since 
$$
g(z) \ll \exp\left(-\pi \beta^{-2/3}x^2/2\right) 
$$
if $z\in \mathbb{C}$ with $|z-x|\leq x/2$. 
From \eqref{Hardy-Airy} and \eqref{einintegral} we deduce that
\begin{equation} \label{Hardy-Airy-eval1}
\begin{split}
\int\limits_{0}^{\infty} \Gamma(-x)\e(\beta x^3-|\alpha|
x) \dif x 
=~& 
\frac{\exp\left(-|\pi\alpha|/(3\beta)\right)}{(12|\alpha|\beta)^{1/4}}
\cdot \e\left(\frac{1}{8}-\frac{2|\alpha|^{3/2}}{3^{3/2}\beta^{1/2}}\right)\\
&+O\left(\frac{1}{|\alpha|}+\frac{\beta^{1/4}}{|\alpha|^{7/4}}\right).
\end{split}
\end{equation}
To see this we note that the stationary point is
\begin{equation} \label{statpoint}
x_0=3^{-1/2}|\alpha|^{1/2}\beta^{-1/6}.
\end{equation}
Using this, we easily calculate the main term on the right-hand side
of \eqref{Hardy-Airy-eval1}. Then 
we simplify the integral on the right-hand side of \eqref{einintegral}
by removing the terms $\exp(-CF(x))$ and $G(x)$  
(dropping these terms can only make the integral larger) and observing that 
$|\dif \mu(x)|$ is the same as $\dif x/2$. What remains in the
integrand is the term $\exp(-Ck \mu(x))=\exp(-Ck x/2)$.
The integral of this term is $\ll 1/|k|$. Plugging in our value of $k$ and
multiplying by the factor $\beta^{-1/3}$, we obtain the term $1/|\alpha|$ in the $O$-term on the right-hand side of \eqref{Hardy-Airy-eval1}.
Furthermore, we calculate the term $G_0\mu_0F_0^{-3/2}$ on the right-hand side of \eqref{einintegral} using \eqref{statpoint}. Multiplying the result with $\beta^{-1/3}$, we get 
the term 
$$
\exp(-|\pi \alpha|/(6\beta))\cdot \frac{\beta^{1/4}}{|\alpha|^{7/4}} \ll 
\frac{\beta^{1/4}}{|\alpha|^{7/4}}
$$ 
on the right-hand side of \eqref{Hardy-Airy-eval1}. 
The term $G_a(|f_a'+k|+\sqrt{f_a''})^{-1}$ on the right-hand side of \eqref{einintegral} tends to $1/|k|$ as $a$ tends to $0$, which is the value with 
which we bounded the integral on the right-hand side of \eqref{einintegral}. Finally,
the term $G_b(|f_b'+k|+\sqrt{f_b''})^{-1}$ tends to 0 as $b$ tends to
infinity. Combining everything we obtain \eqref{Hardy-Airy-eval1}. 

From \eqref{Fsplit2}, \eqref{Hardy-Airy} and \eqref{Hardy-Airy-eval1}, we get
$$
F\left(\alpha,\beta\right) = \frac{2^{1/2}\exp\left(-|\pi \alpha|/(3\beta)\right)}{(3|\alpha|\beta)^{1/4}} \cdot \cos\left(2\pi\left(\frac{1}{8}-\frac{2|\alpha|^{3/2}}{3^{3/2} \beta^{1/2}}\right)\right)
+O\left(\frac{1}{|\alpha|}+\frac{\beta^{1/4}}{|\alpha|^{7/4}}\right),
$$
whence 
\begin{equation}\label{Fbound3c}
F\left(\alpha,\beta\right) = \frac{2^{1/2}\exp\left(-|\pi
    \alpha|/(3\beta)\right)}{(3|\alpha|\beta)^{1/4}} \cdot
\cos\left(2\pi\left(\frac{1}{8}-\frac{2|\alpha|^{3/2}}{3^{3/2}
      \beta^{1/2}}\right)\right)+O\left(\frac{1}{|\alpha|}\right). 
\end{equation}
Indeed, if $\alpha<-\beta^{1/3}$ then the estimate is obvious. 
If instead $-\beta^{1/3}\le \alpha<0$ then the main term is dominated
by the error term 
and the estimate
$F(\alpha,\beta)=O(1/|\alpha|)$ follows from \eqref{another}. Bringing
\eqref{Fbound3c} together with 
\eqref{Fbound1} we may
now record the upper bound 
\begin{equation}\label{Fbound3}
F\left(\alpha,\beta\right) \ll \frac{\exp\left(-|\pi \alpha|/(3\beta)\right)}{(|\alpha|\beta)^{1/4}} +\frac{1}{|\alpha|} \quad \mbox{for any } \alpha\not=0.
\end{equation}
 
We will  need yet another bound that is good for large $|\alpha|$. 
To this end, we take into account that $\Gamma$ extends to the entire
function $\Gamma(z)$ for $z\in \mathbb{C}$ and use
complex analysis. Shifting the line of integration, we get 
\begin{align*}
F(\alpha,\beta) 
&= \int\limits_{-\infty-ic}^{\infty-ic} \Gamma(s)\e(-\beta s^3-\alpha
s) \d s\\
&= \int\limits_{-\infty}^{\infty} \Gamma(x-ic)\e(-\beta
(x-ic)^3-\alpha (x-ic))\dif x\\ 
 &= 
\exp(\pi c^2-2\pi\alpha c+2\pi\beta c^3)
\int\limits_{-\infty}^{\infty} \exp(-\pi x^2-6\pi \beta c x^2)
\e(cx-\beta x^3 +3\beta c^2x-\alpha x)\dif x\\ 
 &\ll 
\exp(\pi c^2-2\pi\alpha c+2\pi\beta c^3)
\int\limits_{-\infty}^{\infty} \exp(-\pi(1+6\beta c) x^2)\dif x, 
\end{align*}
which converges if $c>-1/(6\beta)$. We choose $c:=\sign(\alpha)/(12\beta)$, getting
\begin{equation} \label{Fbound6}
F(\alpha,\beta)\ll \exp\left(\frac{\pi}{12\beta^2}-\frac{\pi
    |\alpha|}{6\beta}\right)
\ll \exp\left(-\pi\cdot\frac{|\alpha|}{12 \beta}\right),
\end{equation}
if $|\alpha|\ge 1/\beta$.

\subsection{Arithmetic functions}\label{s:arith}

In this section we collect together some of the arithmetic functions
that feature in our work and some of their basic properties. 
For any $a\in \RR$ we set 
$$
\sigma_{a}(n):=\sum_{d\mid n}d^{a}.
$$
The cases $a=-1/2$ and $a=-1$ will arise quite often in our analysis. 
Let $\omega(n)$ denote the number of distinct prime factors of $n$. 
Then it is straightforward to show that 
$$
\sigma_{-1/2}(n) \ll (1+\ve)^{\omega(n)},
$$
for any $\ve>0$.
Likewise we have 
$$
\sigma_{-1}(n) \ll \log\log n\ll (\log n)^{\ve}.
$$
Next let
$$
  \phi^*(n):=\frac{\phi(n)}{n}=\prod_{p\mid n} \left(1-
\frac{1}{p}\right).
$$
We observe that $\phi^{*}(ab) \phi^{*}(a,b)=\phi^{*}(a)
\phi^{*}(b)$, where we write $\phi^{*}(a,b)=\phi^{*}((a,b))$ for
short. It is then an elementary exercise to check that 
\begin{equation}
  \label{eq:wot}
  \sum_{\substack{k\mid n\\ (k,a)=1}}
\frac{\mu(k)}{k}=\frac{\phi^*(n)}{\phi^*(a,n)}, \quad 
\sum_{\substack{k\mid n\\ (k,a)=1}}
\frac{\mu(k)}{k\phi^{*}(bk)}=\frac{\phi^*(n,ab)}{\phi^*(bn)\phi^*(a,b,n)}
\prod_{\substack{p\mid n\\ p\nmid ab}}\left(1-\frac{2}{p}\right),
\end{equation}
for any $a,b\in \NN$.

Let $z\in \RR_{\geq 0}$ and $\theta\in [0,1)$. Then  we have the
estimates
\begin{equation} \label{familiarest}
\sum_{n\leq x} \frac{z^{\omega(n)}}{n^\theta}  \ll x^{1-\theta} (\log
x)^{z-1}, \quad
\sum_{n\leq x} \frac{z^{\omega(n)}}{n}  \ll  (\log x)^{z},
\end{equation}
following from \cite[$\S$ II.6]{Ten} and partial summation. Next we record the bound 
\begin{equation} \label{wellknown}
\sum\limits_{1\le n\le x} (n,r)^{1/2} \le \sum\limits_{d\mid r}
\sum\limits_{\substack{1\le n \le x\\ d\mid n}} d^{1/2} \le
x\sum\limits_{d\mid r} d^{-1/2} = x\sigma_{-1/2}(r)\ll
(1+\ve)^{\omega(r)} x,
\end{equation} 
for any $r\in \NN$.
Using partial summation and the fact that the function $\Gamma$, defined in \eqref{gaussiandef}, has rapid decay, it follows that 
\begin{equation} \label{wellknown'}
\sum_{n\in \NN}\Gamma \left(\frac{n}{x}  \right) (n,r)^{1/2} \ll
(1+\ve)^{\omega(r)} x. 
\end{equation} 
Finally we record the estimate
\begin{equation}
  \label{eq:ape}
\sum_{n\leq x} \frac{z^{\omega(n)}\sigma_1(n)}{n^2}  
\leq 
\sum_{kn'\leq x} \frac{z^{\omega(k)+\omega(n')}}{n'^2 k} 
\ll 
(\log x)^z.
\end{equation}
We will make frequent use of the estimates in this section with little
further comment.

\section{Exponential sums}\label{s:expsum}

In this section we collect together some facts concerning complete
exponential sums that will be required in our work. 
We begin with a general upper bound. 
For a polynomial $f\in
\mathbb{Z}[X]$ of degree $n$ with precisely $m$ distinct roots
  $\zeta_1,...,\zeta_m$ and factorisation  
$$
f(X)=A(X-\zeta_1)^{\eta_1}\cdots (X-\zeta_m)^{\eta_m},
$$
define the semi-discriminant of $f$ to be 
$$
\Delta(f):=A^{2n-2} \prod_{i\neq j} (\zeta_i-\zeta_j)^{\eta_i\eta_j}
$$
and the exponent of $f$ to be
$$
\eta(f):=\max\{\eta_1,...,\eta_m\}.
$$
With this notation in mind we have the following result.

\begin{lemma} \label{polyexpsumest} 
Let $Q\in \NN$ and let $g$ be a polynomial over $\mathbb{Z}$ with
degree $n+1$ for $n\ge 2$. Let 
$\Delta=\Delta(g')$ and $\eta=\eta(g')$, where $g'$ is 
the derivative of $g$. Then we have
$$
\left| \sum_{x\bmod{Q}} \e\left( \frac{g(x)}{Q}\right)\right| \leq Q^{1-1/(2\eta)}(\Delta,Q)^{1/(2\eta)} n^{\omega(Q)}.
$$
\end{lemma}

\begin{proof}
A version of this result with $n^{\omega(Q)}$ replaced by $\tau_n(Q)$
is the principal result in work of Loxton and Smith \cite{loxton}. An
inspection of the proof reveals that the sharper version recorded here
is also true. 
\end{proof}

Our work will hinge upon a decent understanding of the exponential sum
\begin{equation} \label{Eq0}
E(c,d;q):=\sideset{}{^*} \sum\limits_{z \bmod q} \e\left(\frac{cz^3-dz^2}{q}\right)
\end{equation} 
for $c,d,q\in \ZZ$ such that $q>0$.  
Here the asterisk denotes a summation in which $z$ runs over residue
classes modulo $q$ that are coprime to $q$. 
The rest of this section is
dedicated to its explicit evaluation as far as this is possible.
In fact it will suffice to do so for prime
power moduli $q=p^t$ since the sums satisfy a multiplicativity
property, as recorded in the following result.

\begin{lemma} \label{factorisation} 
Let $c,d\in \ZZ$ and assume 
that $q_1,...,q_r$ are
  pairwise coprime. Then we have
\begin{equation} \label{factorisationeq}
E(c,d;q_1\cdots q_r)=\prod\limits_{j=1}^r E(c \overline{\tilde{q}_j},d \overline{\tilde{q}_j}; q_j),
\end{equation}
where $\tilde{q}_j=q_1\cdots q_r/q_j$ and $\overline{\tilde{q}_j} \tilde{q}_j\equiv 1$ mod $q_j$. 
\end{lemma} 

\begin{proof}
We prove \eqref{factorisationeq} only for $r=2$. By induction, the result can then be easily extended to general $r$. If $(q_1,q_2)=1$ then
\begin{align*}
E(c,d;q_1q_2) &= \sideset{}{^*} \sum\limits_{z \bmod{ q_1q_2}} \e\left(\frac{cz^3-dz^2}{q_1q_2}\right) \\
&= \sideset{}{^*} \sum\limits_{x_1 \bmod{ q_1}} \ \sideset{}{^*} \sum\limits_{x_2 \bmod{ q_2}} \e\left(\frac{c(x_1q_2+x_2q_1)^3-d(x_1q_2+x_2q_1)^2}{q_1q_2}\right) \\
&= \sideset{}{^*} \sum\limits_{x_1 \bmod{ q_1}} \e\left(\frac{cq_2^2x_1^3-dq_2x_1^2}{q_1}\right) \sideset{}{^*} \sum\limits_{x_2 \bmod{ q_2}} \e\left(\frac{cq_1^2x_2^3-dq_1x_2^2}{q_2}\right).
\end{align*}
Making the change of variables $x_1=\overline{q_2}y_1$, $x_2=
\overline{q_1}y_2$ with $\overline{q_1}q_1\equiv \bmod{ \ q_2}$ and
$\overline{q_2}q_2\equiv \bmod{\  q_1}$, we deduce  that 
\begin{equation} \label{mult}
E(c,d;q_1q_2)=E(c\overline{q_2},d\overline{q_2};q_1)E(c\overline{q_1},d\overline{q_1};q_2).
\end{equation}
This completes the proof.
\end{proof}

It shall transpire 
that the evaluation of $E(c,d;p^t)$ for $p=2,3$ is more complicated
than for primes $p>3$. We begin by making a useful observation
concerning  these bad primes.

\begin{lemma} \label{badqs} Suppose that $4\mid q$ or $9\mid q$. Then we have
$$
E(c,d;q)\not=0 \Rightarrow \begin{cases} 2\mid c & \mbox{if $4\mid q$}, 
\\ 3\mid d & \mbox{if  $9\mid q$}. \end{cases}
$$
\end{lemma}

\begin{proof}
In view of Lemma \ref{factorisation} it will suffice to show that 
$$
E(c,d;p^t)\not=0 \Rightarrow  \begin{cases} 2\mid c & \mbox{if $p=2$}, 
\\ 3\mid d & \mbox{if  $p=3$}, \end{cases}
$$
if $t\geq 2$. 
But clearly 
\begin{align*} 
E(c,d;p^t) &= \sideset{}{^*} \sum\limits_{x=1}^{p^{t-1}} \sum\limits_{y=1}^p \e\left(\frac{c(x+y p^{t-1})^3-d(x+y p^{t-1})^2}{p^t} \right)\\
&= \sideset{}{^*} \sum\limits_{x=1}^{p^{t-1}} \e\left(\frac{cx^3-dx^2}{p^t}\right)\sum\limits_{y=1}^p \e\left(\frac{3cx^2-2dx}{p}\cdot y \right),
 \end{align*}
and if $p=2$ (resp. $p=3$) the inner sum in the last line equals $0$ unless
$2\mid c$ (resp. $3\mid d$).
\end{proof}

For $q\in \mathbb{N}$, we set
\begin{equation} \label{w2def}
w_2(q):= \begin{cases} 1 & \mbox{if  $4\mid q$},  \\ 
0 & \mbox{otherwise,} \end{cases}
\end{equation}
and
\begin{equation} \label{w3def}
w_3(q):= \begin{cases} 1 & \mbox{if $9\mid q$}, \\
 0 & \mbox{otherwise.} \end{cases}
\end{equation}
By Lemma \ref{badqs} it will suffice to study exponential sums of the form
$$
E_{f_2,f_3}(c,d;q):=E(2^{f_2}c,3^{f_3}d;q),
$$
where 
\begin{equation} \label{generalassump}
w_2(q)\le f_2 \le 1 \quad \mbox{and} \quad w_3(q)\le f_3\le 1.
\end{equation}
We will assume \eqref{generalassump} throughout the sequel. We note that equation \eqref{mult} translates into
\begin{equation} \label{mult1}
E_{f_2,f_3}(c,d;q_1q_2)=E_{f_2,f_3}(c\overline{q_2},d\overline{q_2};q_1)E_{f_2,f_3}(c\overline{q_1},d\overline{q_1};q_2). 
\end{equation}

\subsection{Prime power moduli}
In the following we evaluate the exponential sum $E_{f_2,f_3}(c,d;q)$ for prime power moduli $q=p^t$. We set
\begin{equation} \label{rpdef}
r(p):=
\begin{cases} 6 & \mbox{if  $p=2$},   \\ 
5 & \mbox{if $p=3$}, \\  
2 & \mbox{otherwise.} \end{cases}
\end{equation}
We shall treat $E_{f_2,f_3}(c,d;p^t)$ differently for the cases $t<r(p)$ and $t\ge r(p)$.

\subsubsection*{Case 1: $1\le t<r(p)$}

For $p=2,3$, we shall simply use the trivial estimate
$$
|E_{f_2,f_3}(c,d;p^t)|\le 100 \quad \mbox{if } p=2,3 \mbox{ and } t<r(p).
$$
Assume now that $p>3$ and $t<r(p)$, in which case we have $t=1$. If
$p\mid c$ and $p\mid d$, then we have  
$$
E_{f_2,f_3}(c,d;p)=E_{f_2,f_3}(0,0;p)=\varphi(p).
$$
 If $p\mid c$ and $p\nmid d$, then $E_{f_2,f_3}(c,d;p)$ is a quadratic Gauss sum
\begin{align*}
E_{f_2,f_3}(c,d;p)=E_{f_2,f_3}(0,d;p)=\sum\limits_{z=1}^{p-1}
\e\left(-\frac{3^{f_3}dz^2}{p} \right)
&= \sum\limits_{z=1}^{p}
\e\left(-\frac{3^{f_3}dz^2}{p}\right)-1\\ 
&= \sqrt{p}
\left(\frac{-3^{f_3}d}{p}\right)\epsilon_p-1,  
\end{align*}
where $\left(\frac{\cdot}{p}\right)$ is the Legendre symbol and 
\begin{equation} \label{epsdef}
\epsilon_q := \begin{cases} 1 & \mbox{ if $q\equiv 1 \bmod{ 4}$}, \\
i & \mbox{ if $q\equiv -1 \bmod{ 4}$}. \end{cases}
\end{equation}
If $p\nmid c$, then there is no simple closed expression for
$E_{f_2,f_3}(c,d;p)$, but by \cite{weil} we have the Weil bound
$$
|E_{f_2,f_3}(c,d;p)| \le 2\sqrt{p}.
$$

\subsubsection*{Case 2: $t\ge r(p)$}  Here we first prove the
following result.

\begin{lemma} \label{mncaseslemma} Let $p$ be a prime and let $t\ge
  r(p)=r$. Then we have the following.
\begin{itemize}
\item[{\rm (i)}] 
If $p^{t-r+2} \nmid c$ and $p^{t-r+2} \nmid d$, then
$$
E_{f_2,f_3}(c,d;p^t)\not=0 \Rightarrow (c,p^t)=p^s=(d,p^t) \mbox{ for some $s$.} 
$$
In this case, we have 
\begin{equation} \label{red1}
E_{f_2,f_3}(c,d;p^t)=p^s E_{f_2,f_3}\left(\frac{c}{p^s},\frac{d}{p^s}; p^{t-s}\right).
\end{equation}
\item[{\rm (ii)}] If $p^{t-r+2}\mid c$ or $p^{t-r+2}\mid d$, then 
$$
E_{f_2,f_3}(c,d;p^t)\not=0 \Rightarrow p^{t-r+1}\mid (c,d).
$$
In this case, we have 
$$
E_{f_2,f_3}(c,d;p^t)=p^{t-r+1} E_{f_2,f_3}\left(\frac{c}{p^{t-r+1}},\frac{d}{p^{t-r+1}}; p^{r-1}\right).
$$
\end{itemize}
\end{lemma}

\begin{proof} We recall that we assume \eqref{generalassump} holds, where
  here $q=p^t$. Hence, under the conditions of this lemma, we have
  $f_2=1$ if $p=2$ and $f_3=1$ if $p=3$.  

Our starting point is the following computation, valid for all primes $p$. Set 
$$
\nu=\nu(p):=\begin{cases} 1 & \mbox{if  $p=2,3$},  \\ 0 & 
\mbox{otherwise.} \end{cases}
$$
Then
\begin{align*}
E_{f_2,f_3}(c,d;p^t) &=
 \sideset{}{^*} \sum\limits_{x=1}^{p^{t-1-\nu}} \ \sum\limits_{y=1}^{p^{1+\nu}} 
\e\left(\frac{2^{f_2}c(x+yp^{t-1-\nu})^3-3^{f_3}d(x+yp^{t-1-\nu})^2}{p^t}
\right)\\ 
&= \sideset{}{^*} \sum\limits_{x=1}^{p^{t-1-\nu}} \e\left(\frac{2^{f_2}cx^3-3^{f_3}dx^2}{p^t}\right)\sum\limits_{y=1}^{p^{1+\nu}} 
\e\left(\frac{3\cdot 2^{f_2}cx^2-2\cdot 3^{f_3}dx}{p^{1+\nu}}\cdot y \right) \\
&= p^{1+\nu} \sideset{}{^*} \sum\limits_{\substack{x=1\\ 3\cdot 2^{f_2}cx \equiv 2\cdot 3^{f_3}d \bmod{ p^{1+\nu}}}}^{p^{t-1-\nu}} \e\left(\frac{2^{f_2}cx^3-3^{f_3}dx^2}{p^t}\right)\\
&= p^{1+\nu} \sideset{}{^*} \sum\limits_{\substack{x=1\\ 3^{1-f_3}cx
 \equiv 2^{1-f_2}d \bmod{ p}}}^{p^{t-1-\nu}}
 \e\left(\frac{2^{f_2}cx^3-3^{f_3}dx^2}{p^t}\right). 
\end{align*}
 We observe that this equals $0$ if $(c,p)\not= (d,p)$ because in this case the congruence 
 $$
 3^{1-f_3}cx\equiv 2^{1-f_2}d \bmod{ p}
 $$ 
 is not solvable for $x$ coprime to $p$. Moreover, if $(c,p)=p=(d,p)$, then we have 
\begin{equation} \label{reduction}
E_{f_2,f_3}(c,d;p^t)=pE_{f_2,f_3}\left(\frac{c}{p},\frac{d}{p};p^{t-1}\right).
\end{equation}
Using this observation, we prove Lemma \ref{mncaseslemma} by induction
over $t$. The base case $t=r$ is a trivial consequence of our work so
far, where, for the proof of \eqref{red1}, it is important to note that $t>s$ by the assumptions in (i).

Now assume that $t>r$ and that the
assertion has been established for the exponent $t-1$.  
First, let the conditions $p^{t-r+2}\nmid c$ and $p^{t-r+2} \nmid d$
in  case (i) be satisfied and suppose that
$E_{f_2,f_3}(c,d;p^t)\not=0$. Then, by the afore-mentioned
observation, we have $(c,p^t)=1$ if and only if $(d,p^t)=1$. If
$(c,p^t)>1$ and $(d,p^t)>1$, then from \eqref{reduction} and the
induction hypothesis, it follows that
$(c/p,p^{t-1})=p^{s-1}=(d/p,p^{t-1})$ and hence $(c,p^t)=p^s=(c,p^t)$
for some $s$, and 
$$
E_{f_2,f_3}(c,d;p^t)=pE_{f_2,f_3}\left(\frac{c}{p},\frac{d}{p};p^{t-1}\right)=p^s E_{f_2,f_3}\left(\frac{c}{p^s},\frac{d}{p^s}; p^{t-s}\right),
$$
as claimed.

For case (ii) we suppose that $p^{t-r+2}\mid c$ and
$E_{f_2,f_3}(c,d;p^t)\not=0$. Then, by our observation above, we have
$(d,p^t)>1$, and from \eqref{reduction} and the induction hypothesis,
it follows that $p^{t-r}\mid (d/p)$ and hence $p^{t-r+1}\mid d$. Thus 
$$
E_{f_2,f_3}(c,d;p^t)=pE_{f_2,f_3}\left(\frac{c}{p},\frac{d}{p};p^{t-1}\right)=p^{t-r+1}
E_{f_2,f_3}\left(\frac{c}{p^{t-r+1}},\frac{d}{p^{t-r+1}};
  p^{r-1}\right), 
$$
as claimed. The case $p^{t-r+2}\mid d$ is handled  similarly.
\end{proof}

Let $t\geq r(p)=r$. By Lemma \ref{mncaseslemma} we have 
$$
E_{f_2,f_3}(c,d;p^t)=p^{t-r+1}E_{f_2,f_3}\left(\frac{c}{p^{t-r+1}},\frac{d}{p^{t-r+1}};p^{r-1}\right)
$$
if $p^{t-r+1}\mid (c,d)$, and the exponential sum on the right-hand
side has been dealt with before. If $p^{t-r+1} \nmid (c,d)$, then by
the same Lemma \ref{mncaseslemma}, we have $(c,p^t)=p^s=(d,p^t)$ for
some $s\le t-r$. The following result 
gives an explicit evaluation of $E_{f_2,f_3}(c,d;p^t)$ in this
situation.  
Before stating the lemma, we define a residue symbol 
\begin{equation} \label{strangesymbol}
\left\{\frac{\cdot }{2^u}\right\}:=\begin{cases}
\frac{1}{\sqrt{2}}  \left(1+\e(\cdot /4)\right) &  \mbox{ if  $2\nmid
  u$},  \\ \e(\cdot /8) & \mbox{ if $2\mid u$}. 
\end{cases}
\end{equation}
for $u\ge 3$. Furthermore for any odd integer $n$ we let
$\left(\frac{\cdot}{n}\right)$ denote the Jacobi symbol. Throughout the following, in place of $\overline{x}$, we often use the notation $x^{-1}$ for the multiplicative inverse. 
We now have the following result. 

\begin{lemma} \label{explicitevaluation} Let $p$ be a prime and $t\ge r(p)=r$. Suppose that $(c,p^t)=p^s=(d,p^t)$ and $s\le t-r$. Set $c_1=c/p^s$, $d_1=d/p^s$ and $u=t-s$. Suppose that
$$
A\equiv 4^{1-f_2}\cdot 27^{f_3-1} \bmod{p^u}.
$$
Then
\begin{equation} \label{primepowersexplicit}
E_{f_2,f_3}(c,d;p^t)= p^{t-u/2} \e\left(\frac{-A  \overline{c_1}^2d_1^3}{p^{u}}\right) R(f_3,d_1,p^u),
\end{equation}
where 
$$
R(f_3,d_1,p^u):=\begin{cases}
\left(\frac{3^{f_3}d_1}{p^u}\right)\epsilon_{p^u} & 
\mbox{ if $p>3$},  \\
\sqrt{3} \left(\frac{d_1}{3^{u-1}}\right)\epsilon_{3^{u-1}} & \mbox{
  if $p=3$}, \\ 
\sqrt{2} \left\{\frac{3^{f_3}d_1}{2^u}\right\} & \mbox{ if $p=2$}.
\end{cases} 
$$
\end{lemma}

\begin{proof}
By \eqref{red1}, we have
\begin{equation} \label{red4}
E_{f_2,f_3}(c,d;p^t)=p^sE_{f_2,f_3}(c_1,d_1;p^u)=p^sE(p^u),
\end{equation} 
say. 
We further observe that
$$
(c_1,p)=1=(d_1,p), \quad  u\ge r.
$$
Now it remains to consider $E(p^u)$.

If $u$ is even then we set $u=2v$ and if $u$ is odd then we set $u=2v+1$. Note that $v\ge 1$ since $u\ge 2$. The starting point of our proof are the identities
\begin{equation} 
\begin{split}
\label{evencase}
E(p^u) &= \sideset{}{^*} \sum\limits_{x \bmod{ p^v}} \sum\limits_{y \bmod{ p^v}} \e\left(\frac{2^{f_2}c_1(x+yp^{v})^3-3^{f_3}d_1(x+yp^{v})^2}{p^{2v}} \right)\\
&= \sideset{}{^*} \sum\limits_{x \bmod{ p^{v}}} \e\left(\frac{2^{f_2}c_1x^3-3^{f_3}d_1x^2}{p^{2v}}\right)\sum\limits_{y \bmod{ p^v}} \e\left(\frac{3\cdot 2^{f_2}c_1x^2-2\cdot 3^{f_3}d_1x}{p^v}\cdot y \right) \\
&= p^v\sideset{}{^*} \sum\limits_{\substack{x \bmod{ p^v}\\ 3 \cdot 2^{f_2}c_1x \equiv 2 \cdot 3^{f_3}d_1 \bmod{ p^v}}} \e\left(\frac{2^{f_2}c_1x^3-3^{f_3}d_1x^2}{p^{2v}}\right)
 \end{split}
\end{equation}
for $u$ even, and 
\begin{equation} 
\begin{split}
\label{oddcase} 
E(p^u) &= \sideset{}{^*} \sum\limits_{x \bmod{ p^{v+1}}} \sum\limits_{y \bmod{ p^v}} \e\left(\frac{2^{f_2}c_1(x+yp^{v+1})^3-3^{f_3}d_1(x+yp^{v+1})^2}{p^{2v+1}} \right)\\
&= \sideset{}{^*} \sum\limits_{x \bmod{ p^{v+1}}} \e\left(\frac{2^{f_2}c_1x^3-3^{f_3}d_1x^2}{p^{2v+1}}\right)\sum\limits_{y \bmod{ p^v}} \e\left(\frac{3\cdot 2^{f_2}c_1x^2-2\cdot 3^{f_3}d_1 x}{p^v}\cdot y \right) \\
&= p^v\sideset{}{^*} \sum\limits_{\substack{x \bmod{ p^{v+1}}\\ 3 \cdot 2^{f_2}c_1x \equiv 2 \cdot 3^{f_3}d_1\bmod{ p^{v}}}} \e\left(\frac{ 2^{f_2}c_1x^3-3^{f_3}d_1x^2}{p^{2v+1}}\right)
\end{split}
\end{equation}
for $u$ odd. In the following we treat the cases $p>3$, $p=3$, $p=2$, $u$ even, $u$ odd separately.

\subsubsection*{Case I: $p>3$} For convenience, we use the notations 
\begin{equation} \label{uppsala}
m=2^{f_2}c_1,  \quad n=3^{f_3}d_1. 
\end{equation}
Suppose first that $u$ is even.  
From \eqref{evencase}, it follows that
$$
E(p^u) = p^v\sideset{}{^*} \sum\limits_{\substack{x \bmod{ p^v}\\ 3mx \equiv 2n \bmod{ p^v}}} \e\left(\frac{mx^3-nx^2}{p^{2v}}\right).
$$
 The congruence in the last line has precisely one solution modulo
 $p^v$, and using Hensel's lemma, this solution can be lifted uniquely
 to a solution of $3mx \equiv 2n \bmod{ p^{2v}}$, given by 
$$
 x\equiv \overline{3m}\cdot 2n \bmod{ p^{2v}},
$$
 where $\overline{3m}$ is the multiplicative inverse of $3m$ modulo
 $p^{2v}$. Hence the sum collapses into  a single term  
\begin{align*}
\e\left(\frac{mx^3-nx^2}{p^{2v}}\right) =
\e\left(\frac{m(\overline{3m}\cdot 2n)^3-n(\overline{3m}\cdot
    2n)^2}{p^{2v}}\right)
&= \e\left(\frac{(8\cdot \overline{27}-4\cdot
    \overline{9})\overline{m}^2n^3}{p^{2v}}\right)
\\ &=\e\left(\frac{-4\cdot
    \overline{27}\cdot \overline{m}^2n^3}{p^{2v}}\right),
\end{align*}
whence 
\begin{equation} \label{cop3}
E(p^u)=p^{u/2} \e\left(\frac{-4\cdot \overline{27}\cdot \overline{m}^2n^3}{p^{u}}\right).
\end{equation}

Next suppose that $u$ is odd.  From \eqref{oddcase}, it follows that
\begin{equation} 
\begin{split}
\label{cop4} 
E(p^u) &= p^v\sideset{}{^*} \sum\limits_{\substack{x \bmod{ p^{v+1}}\\ 3mx \equiv 2n \bmod{ p^{v}}}} \e\left(\frac{mx^3-nx^2}{p^{2v+1}}\right)\\
&= p^v\sum\limits_{h=1}^p \e\left(\frac{m(x+hp^v)^3-n(x+hp^v)^2}{p^{2v+1}}\right) \\
&= p^v \e\left(\frac{mx^3-nx^2}{p^{2v+1}}\right)\sum\limits_{h=1}^p \e\left(\frac{3mx^2hp^v+3mxh^2p^{2v}-2nxhp^v-nh^2p^{2v}}{p^{2v+1}}\right), 
\end{split}
\end{equation}
where in the last two lines $x$ satisfies the congruence  $3mx \equiv
2n \bmod{ p^{v}}$. Again this solution  can be
lifted uniquely to a solution 
$$
 x\equiv \overline{3m}\cdot 2n \bmod{ p^{2v+1}},
$$
 where $\overline{3m}$ is the multiplicative inverse of $3m$ modulo
 $p^{2v+1}$.  Now, similarly as in the case of even $u$,  the numerator in the
 first exponential in the last line of \eqref{cop4} takes the form  
$mx^3-nx^2\equiv -4 \cdot \overline{27} \cdot \overline{m}^2n^3$ mod $p^{2v+1}$.
Moreover, we have $3mx^2 \equiv \overline{3m} \cdot 4n^2$ mod $p^{2v+1}$, $2nx\equiv \overline{3m} \cdot 4n^2$ mod $p^{2v+1}$ and $3mx\equiv 2n$ mod $p^{2v+1}$. Therefore, the numerator in the second exponential in the last line of \eqref{cop4} simplifies into $nh^2p^{2v}$. It follows that
$$
E(p^u)=p^v \e\left(\frac{-4 \cdot \overline{27} \cdot
    \overline{m}^2n^3}{p^{2v+1}}\right) \sum\limits_{h=1}^p
\e\left(\frac{nh^2}{p}\right). 
$$
Evaluating the quadratic Gauss sum on the right-hand side, we obtain
$$
E(p^u)=p^{v+1/2} \e\left(\frac{-4 \cdot \overline{27} \cdot \overline{m}^2n^3}{p^{2v+1}}\right) \left(\frac{n}{p}\right) \epsilon_p=
p^{u/2} \e\left(\frac{-4 \cdot \overline{27} \cdot \overline{m}^2n^3}{p^{u}}\right) 
\left(\frac{n}{p}\right) \epsilon_p.
$$
Combining this with \eqref{cop3} into a single equation, we deduce that
\begin{equation} \label{finalprimepowers}
E(p^u)=p^{u/2} \e\left(\frac{-4 \cdot \overline{27} \cdot \overline{m}^2n^3}{p^{u}}\right) 
\left(\frac{n}{p^u}\right) \epsilon_{p^u},
\end{equation}
where $\epsilon_{p^u}$ is defined as in \eqref{epsdef}.  Combining
\eqref{red4}, \eqref{uppsala} and \eqref{finalprimepowers}, we obtain
\eqref{primepowersexplicit} for $p>3$. 

\subsubsection*{ Case II: $p=3$} 

We note that $f_3=1$ by \eqref{generalassump} and use the notations
\begin{equation} \label{uppsala1}
m:=2^{f_2}c_1, \quad n:=d_1. 
\end{equation}
Suppose first that $u$ is even.  Note that $v\ge 3$ since $u\ge
6$. From \eqref{evencase} it follows that 
\begin{equation} 
\begin{split}
\label{cop23}
E(3^u) = ~& 
3^v\sideset{}{^*} \sum\limits_{\substack{x \bmod{ 3^v}\\ mx \equiv 2n
    \bmod{ 3^{v-1}}}} \e\left(\frac{mx^3-3nx^2}{3^{2v}}\right)\\ 
= ~& 
3^v\sum\limits_{h=1}^3 \e\left(\frac{m(x+h\cdot 3^{v-1})^3-3n(x+h\cdot 3^{v-1})^2}{3^{2v}}\right)\\
= ~& 3^v \e\left(\frac{mx^3-3nx^2}{3^{2v}}\right)\\
&\times \sum\limits_{h=1}^3
\e\left(\frac{mx^2h\cdot 3^v+mxh^2\cdot 3^{2v-1}-2nxh\cdot
    3^v-nh^2\cdot 3^{2v-1}}{3^{2v}}\right),  
\end{split}
\end{equation}
where in the last two lines $x$ satisfies the congruence  $mx \equiv 2n \bmod{ 3^{v-1}}$. The only solution of this congruence can be lifted uniquely to a solution of $mx \equiv 2n \bmod{ 3^{2v}}$, given by
$$
 x\equiv \overline{m}\cdot 2n \bmod{ 3^{2v}},
$$
 where $\overline{m}$ is the multiplicative inverse of $m$ modulo $3^{2v}$.  Now, the numerator in the first exponential in the last line of \eqref{cop23} takes the form 
$mx^3-3nx^2\equiv -4 \cdot \overline{m}^2n^3$ mod $3^{2v}$.
Moreover 
the numerator in the second exponential in the last line of
 \eqref{cop23} simplifies into $nh^2\cdot 3^{2v-1}$. It follows that 
$$
E(3^u)=3^v \e\left(\frac{-4 \cdot \overline{m}^2n^3}{3^{2v}}\right) \sum\limits_{h=1}^3 \e\left(\frac{nh^2}{3}\right).
$$
Evaluating the quadratic Gauss sum on the right-hand side, we obtain
\begin{equation} \label{cop63}
E(3^u)=3^{v+1/2} \e\left(\frac{-4 \cdot \overline{m}^2n^3}{3^{2v}}\right) \left(\frac{n}{3}\right) \epsilon_3=
3^{(u+1)/2} \e\left(\frac{-4 \cdot \overline{m}^2n^3}{3^{u}}\right) 
\left(\frac{n}{3}\right) \epsilon_3.
\end{equation} 

Next suppose that $u$ is odd.
Note that $v\ge 2$ since $u\ge 5$. Here we slightly differ from
\eqref{oddcase} and observe that 
\begin{equation} 
\begin{split}
\label{cop43}
E(3^u) &= \sideset{}{^*} \sum\limits_{x \bmod{ 3^{v}}} \sum\limits_{y \bmod{ 3^{v+1}}} \e\left(\frac{m(x+y\cdot 3^{v})^3-3n(x+y\cdot 3^{v})^2}{3^{2v+1}} \right)\\
&= \sideset{}{^*} \sum\limits_{x \bmod{ 3^{v}}} \e\left(\frac{mx^3-3nx^2}{3^{2v+1}}\right)\sum\limits_{y \bmod{ 3^{v+1}}} \e\left(\frac{mx^2-2nx}{3^v}\cdot y \right) \\
&= 3^{v+1} \sideset{}{^*} \sum\limits_{\substack{x \bmod{ 3^v}\\ mx \equiv 2n \bmod{ 3^{v}}}} \e\left(\frac{mx^3-3nx^2}{3^{2v+1}}\right).
\end{split}
\end{equation}
The only solution of the congruence in the last line can be lifted
uniquely to the solution 
$$
 x\equiv \overline{m}\cdot 2n \bmod{ 3^{2v+1}},
$$
 where $\overline{m}$ is the multiplicative inverse of $m$ modulo
 $3^{2v+1}$.  Now the numerator in the
 exponential in the last line of \eqref{cop43} takes the form
 $mx^3-3nx^2\equiv -4 \cdot \overline{m}^2n^3$ mod $3^{2v+1}$. Hence
\begin{equation} \label{cop63*}
E(3^u)=3^{v+1} \e\left(\frac{-4 \cdot \overline{m}^2n^3}{3^{2v+1}}\right)=
3^{(u+1)/2} \e\left(\frac{-4 \cdot \overline{m}^2n^3}{3^{u}}\right).
\end{equation} 

We combine \eqref{cop63} and \eqref{cop63*} into a single equation, namely
\begin{equation} \label{finalprimepowers3}
E(3^u)=3^{(u+1)/2} \e\left(\frac{-4 \cdot \overline{m}^2n^3}{3^{u}}\right) 
\left(\frac{n}{3^{u-1}}\right) \epsilon_{3^{u-1}}.
\end{equation}
Combining \eqref{red4}, \eqref{uppsala1} and \eqref{finalprimepowers3}, we obtain \eqref{primepowersexplicit} for $p=3$.

\subsubsection*{Case III: $p=2$} 

We note that $f_2=1$ by \eqref{generalassump} and use the notations
\begin{equation} \label{uppsala2}
m:=c_1,  \quad  n:=3^{f_3}d_1. 
\end{equation}
Suppose first that $u$ is even.  Note that $v\ge 3$ since $u\ge 6$. From \eqref{evencase}, it follows that
\begin{equation} 
\begin{split}
\label{cop22}
E(2^u) =~& 2^v\sideset{}{^*}
\sum\limits_{\substack{x \bmod{ 2^v}\\ 3mx \equiv n \bmod{ 2^{v-1}}}}
\e\left(\frac{2mx^3-nx^2}{2^{2v}}\right)\\ 
=~& 2^v \sum\limits_{h=1}^2 \e\left(\frac{2m(x+h\cdot
    2^{v-1})^3-n(x+h\cdot 2^{v-1})^2}{2^{2v}}\right)\\ 
=~& 2^v \e\left(\frac{2mx^3-nx^2}{2^{2v}}\right)\\
&\times\sum\limits_{h=1}^2
\e\left(\frac{3mx^2h\cdot 2^v+6mxh^2\cdot 2^{2v-2}-nxh\cdot
    2^v-nh^2\cdot 2^{2v-2}}{2^{2v}}\right),  
\end{split}
\end{equation}
where in the last two lines, $x$ satisfies the congruence  $3mx \equiv
n \bmod{ 2^{v-1}}$. The only solution of this congruence can be lifted
uniquely to the solution 
$$
 x\equiv \overline{3m}\cdot n \bmod{ 2^{2v}},
$$
 where $\overline{m}$ is the multiplicative inverse of $m$ modulo $2^{2v}$.  Now the numerator in the first exponential in the last line of \eqref{cop22} takes the form 
$2mx^3-nx^2\equiv -\overline{27} \cdot \overline{m}^2n^3$ mod $2^{2v}$.
Moreover the numerator in the 
second exponential in the last line of \eqref{cop22} simplifies into
$nh^2\cdot 2^{2v-2}$. It follows that
\begin{equation}
\label{cop62*}\begin{split}
E(2^u)&=2^v \e\left(\frac{-\overline{27} \cdot
    \overline{m}^2n^3}{2^{2v}}\right) \sum\limits_{h=1}^2
\e\left(\frac{nh^2}{4}\right)\\ 
&= 2^v  \e\left(\frac{-\overline{27} \cdot \overline{m}^2n^3}{2^{2v}}\right)  \left(1+\e\left(\frac{n}{4}\right)\right)\\
&= 2^{(u+1)/2} \e\left(\frac{-\overline{27} \cdot
    \overline{m}^2n^3}{2^{u}}\right) \cdot
\frac{\left(1+\e\left(n/4\right)\right)}{\sqrt{2}}.
 \end{split}\end{equation}

Next suppose that $u$ is odd.  
Note that $v\ge 3$ since $u\ge 7$. From \eqref{oddcase}, it follows that
\begin{equation} 
\begin{split}
\label{cop42}
E(2^u)
=~& 2^{v} \sideset{}{^*} \sum\limits_{\substack{x \bmod{ 2^{v+1}}\\
    3mx \equiv n \bmod{ 2^{v-1}}}}
\e\left(\frac{2mx^3-nx^2}{2^{2v+1}}\right)\\ 
=~& 2^{v} \sum\limits_{h=1}^4 \e\left(\frac{2m(x+h\cdot 2^{v-1})^3-n(x+h\cdot 2^{v-1})^2}{2^{2v+1}}\right)\\
=~& 2^{v} \cdot \e\left(\frac{2mx^3-nx^2}{2^{2v+1}}\right) \\
&\times \sum\limits_{h=1}^4 \e\left(\frac{3mx^2h\cdot 2^{v}+6mxh^2\cdot
    2^{2v-2}-nxh\cdot 2^v-nh^2\cdot 2^{2v-2}}{2^{2v+1}}\right), 
\end{split}
\end{equation}
where in the last two lines, $x$ satisfies the congruence  $3mx \equiv n \bmod{ 2^{v-1}}$.
The only solution of the congruence in the last line can be lifted
uniquely to the solution 
$$
 x\equiv \overline{3m}\cdot n \bmod{ 2^{2v+1}},
$$
where $\overline{m}$ is the multiplicative inverse of $m$ modulo
$2^{2v+1}$.  Now the numerator in the
first exponential in the last line of \eqref{cop42} takes the form
$2mx^3-nx^2\equiv -\overline{27}\cdot \overline{m}^2n^3$ mod
$2^{2v+1}$, and the numerator in the second exponential simplifies
into $nh^2 \cdot 2^{2v-2}$. Hence 
$$
E(2^u)=2^{v} \e\left(\frac{-\overline{27} \cdot \overline{m}^2n^3}{2^{2v+1}}\right) \sum\limits_{h=1}^4 \e\left(\frac{nh^2}{8}\right).
$$
Recall that we suppose $(n,2)=1$. In this case it is easily seen that
$$
\sum\limits_{h=1}^4 \e\left(\frac{nh^2}{8}\right)=2\e\left(\frac{n}{8}\right).
$$
Hence
\begin{equation} \label{ugly}
E(2^u)=2^{v+1} \cdot \e\left(\frac{-\overline{27} \cdot \overline{m}^2n^3}{2^{2v+1}}\right) \cdot \e\left(\frac{n}{8}\right)=2^{(u+1)/2} \cdot \e\left(\frac{-\overline{27} \cdot \overline{m}^2n^3}{2^{u}}\right) \cdot \e\left(\frac{n}{8}\right).
\end{equation}

We combine \eqref{cop62*} and \eqref{ugly} into a single equation, namely
\begin{equation} \label{finalprimepowers2}
E(2^u)=2^{(u+1)/2} \cdot \e\left(\frac{-\overline{27} \cdot \overline{m}^2n^3}{2^{u}}\right)\cdot \left\{\frac{n}{2^u}\right\},
\end{equation}
where the residue symbol $\left\{\frac{\cdot}{2^u}\right\}$ is defined as in \eqref{strangesymbol}.
Combining \eqref{red4}, \eqref{uppsala2} and \eqref{finalprimepowers2}, we obtain
\eqref{primepowersexplicit} for $p=2$. 
\end{proof}

\subsection{Composite moduli}

Now we look at composite moduli $q$, beginning with the following
general estimate. 

\begin{lemma} \label{generalexpestimate} 
Let $c,d,q\in \ZZ$ such that $q>0$ and let 
$f_2,f_3\in \{0,1\}$. Then we have  
$$
|E_{f_2,f_3}(c,d;q)| \le C (c,d,q)^{1/2}q^{1/2}2^{\omega(q)},
$$
where $C=10000$.
\end{lemma} 

\begin{proof}
Combining all of our results from the previous section 
we deduce that the estimate in the lemma holds with $C=1$ if $q$ is a power of
a prime $p>3$ and with $C=100$ if $q$ is a power of $2$ or $3$.  This
together with Lemma \ref{factorisation} implies the statement of the
lemma for general $q$. 
\end{proof}

The following lemma contains a precise evaluation of
$E_{f_2,f_3}(c,d;q)$ for power-full moduli $q$. Throughout the sequel,
we denote by $\rad(n)$ the largest square-free divisor of $n\in
\mathbb{N}$. 

\begin{lemma} \label{compositeexplicitevaluation1}  
Let $c,d,q\in \ZZ$ such that $q>0$, let 
$f_1,f_2\in
  \{0,1\}$ and assume that
  \eqref{generalassump} holds. Suppose that $q=2^{t_2}3^{t_3}q_1$,
  where $(q_1,6)=1$. Set $e_2:=\sign(t_2)$ and
  $e_3:=\sign(t_3)$. Then we have the following.
\begin{itemize}
\item[{\rm (i)}] If $2^{6e_2}3^{5e_3}\rad(q_1)^2$ divides $q/(c,q)$ or $q/(d,q)$, then 
\begin{equation} \label{noteqzero1}
E_{f_2,f_3}(c,d;q)\not= 0 \Rightarrow (c,q)=(d,q).
\end{equation}
In this case, set
\begin{equation} \label{df*}
\begin{split}
q'&=(c,q)=(d,q), \quad \tilde{q}=\frac{q}{q'}, \quad c'=\frac{c}{q'}, \quad  d'=\frac{d}{q'},\\
q_1'&=(c,q_1)=(d,q_1), \quad \tilde{q}_1=\frac{q_1}{q_1'}.
\end{split}
\end{equation}
Then we have
\begin{equation} \label{compositemoduliexplicit}
E_{f_2,f_3}(c,d;q)= \frac{q}{\sqrt{\tilde{q}}} \cdot \e\left(\frac{-A\cdot  \overline{c'}^2{d'}^3}{\tilde{q}}\right) \left(\frac{d'}{\tilde{q}_1}\right) \cdot P,
\end{equation}
where $|P|\leq 2 \sqrt{3}$ depends at most on $\tilde{q}$, $f_3$ and the residue class of
 $d'$ modulo $8^{e_2}3^{e_3}$, and moreover, 
\begin{equation} \label{wAdefi}
A\equiv 4^{1-f_2}\cdot 27^{f_3-1} \bmod{\tilde{q}}.
\end{equation}

\item[{\rm (ii)}] 
If $E_{f_2,f_3}(c,d;q)\not=0$, then the following equivalences hold:
\begin{align*}
\mu^2\left(\frac{q_1}{(c,q_1)}\right)=1 &\Longleftrightarrow \mu^2\left(\frac{q_1}{(d,q_1)}\right)=1,\\
\frac{2^{t_2}}{(c,2^{t_2})}\mid 2^5 &\Longleftrightarrow \frac{2^{t_2}}{(d,2^{t_2})}\mid 2^5,\\
\frac{3^{t_3}}{(c,3^{t_3})}\mid 3^4 &\Longleftrightarrow \frac{3^{t_3}}{(d,3^{t_3})}\mid 3^4.
\end{align*}
\end{itemize}
\end{lemma}

\begin{proof}  
Part (ii) and the implication \eqref{noteqzero1} in part (i) are 
straightforward consequences of Lemmas \ref{factorisation} and \ref{mncaseslemma}.  
It remains to establish \eqref{compositemoduliexplicit} under the conditions of part (i). Throughout the following, we set
\begin{equation} \label{df**}
\begin{split}
2^{s_2}&=(c,2^{t_2})=(d,2^{t_2}), \quad u_2=t_2-s_2\ge 6e_2,\\ 
3^{s_3}&=(c,3^{t_3})=(d,3^{t_3}), \quad u_3=t_3-s_3\ge 5e_3.
\end{split}
\end{equation}
First we assume that $t_2=0$, so that $q$ is odd. In this case we shall prove that
\begin{equation} \label{compositemoduliexplicit***}
E_{f_2,f_3}(c,d;q)= \frac{q}{\sqrt{\tilde{q}/3^{e_3}}} \cdot
\e\left(\frac{-A\cdot  \overline{c'}^2{d'}^3}{\tilde{q}}\right)
\left(\frac{3^{f_3-e_3}d'}{w\tilde{q}_1}\right)
\epsilon_{w\tilde{q}_1}, 
\end{equation}
where 
$
w=3^{u_3-e_3}.
$ 

We first treat the case $t_2=0=t_3$, which is equivalent to
$\tilde{q}_1=\tilde{q}$. In this case
\eqref{compositemoduliexplicit***} takes the form 
\begin{equation} \label{compositemoduliexplicitsimple}
E_{f_2,f_3}(c,d;q)=\frac{q}{\sqrt{\tilde{q}}} \cdot \e\left(\frac{-A\cdot  \overline{c'}^2{d'}^3}{\tilde{q}}\right) \left(\frac{3^{f_3}d'}{\tilde{q}}\right) \epsilon_{\tilde{q}}.
\end{equation}
This holds trivially if $q=1$.  For $q>1$ and $(q,6)=1$,  we prove
\eqref{compositemoduliexplicitsimple}  by induction over the number of
prime divisors of $q$. If $q=p^t$ for some prime $p\not=2,3$ and $p^2$
divides $q/(c,q)$, then \eqref{compositemoduliexplicitsimple}
coincides with \eqref{primepowersexplicit}. Now assume that $g$ is the
number of prime divisors of $q$ and that
\eqref{compositemoduliexplicitsimple} has been established for all
moduli $r$ with $g-1$ prime divisors such that $(r,6)=1$ and
$\rad(r)^2$ divides $r/(r,d)$. Suppose that $q=r_1r_2$, where
$(r_1,r_2)=1$ and $r_2=p^t$ is a prime power. 
Let $\overline{r_1}, \overline{r_2}$ be such that  
$r_1\overline{r_1}+r_2 \overline{r_2}=1$. Then using
\eqref{mult1}, the induction hypothesis and Lemma
\ref{explicitevaluation}, we have 
\begin{equation} 
\begin{split}
\label{compfac}
E_{f_2,f_3}(c,d;q) =~&
E_{f_2,f_3}(c\overline{r_2},d\overline{r_2};r_1)E_{f_2,f_3}(c\overline{r_1},d\overline{r_1};r_2)\\ 
=~& \left(\frac{r_1}{\sqrt{\tilde{r}_1}} \cdot \e\left(\frac{-A \cdot
      \overline{r_2}\cdot \overline{c_1}^2d_1^3}{\tilde{r}_1}\right)
  \left(\frac{3^{f_3}\overline{r_2}d_1}{\tilde{r}_1}\right)
  \epsilon_{\tilde{r}_1}\right)\\ 
& \times 
 \left(\frac{r_2}{\sqrt{\tilde{r}_2}} \cdot \e\left(\frac{-A \cdot
       \overline{r_1}\cdot \overline{c_2}^2d_2^3}{\tilde{r}_2}\right)
   \left(\frac{3^{f_3}\overline{r_1}d_2}{\tilde{r}_2}\right)
   \epsilon_{\tilde{r}_2}\right),\end{split}
\end{equation} 
where for $i=1,2$ we set 
$$
r_i'=(c,r_i)=(d,r_i), \quad  \tilde{r_i}=\frac{r_i}{r_i'}, \quad
c_i=\frac{c}{r_i'}, \quad d_i=\frac{d}{r_i'}. 
$$
We note that
\begin{align*}
c_1&=\frac{c}{r_1'}=\frac{cr_2'}{q'}=c'r_2', \quad  
c_2=\frac{c}{r_2'}=\frac{cr_1'}{q'}=c'r_1',\\ 
d_1&=\frac{d}{r_1'}=\frac{dr_2'}{q'}=d'r_2', \quad  
d_2=\frac{d}{r_2'}=\frac{dr_1'}{q'}=d'r_1'.
\end{align*}
Combining the exponential terms we obtain
\begin{align*}
\e\left(\frac{-A \cdot \overline{r_2}\cdot
    \overline{c_1}^2d_1^3}{\tilde{r}_1}\right) &\cdot \e\left(\frac{-A
    \cdot \overline{r_1}\cdot
   \overline{c_2}^2d_2^3}{\tilde{r}_2}\right)  \\ 
&= \e\left(\frac{-A
    \cdot \overline{r_1}\cdot
    \overline{c'r_1'}^2(d'r_1')^3}{\tilde{r}_2}\right) \cdot
\e\left(\frac{-A \cdot \overline{r_2}\cdot \overline{c'
      r_2'}^2(d'r_2')^3}{\tilde{r}_1}\right)\\ 
&= \e\left(\frac{-A \cdot
    \overline{c'}^2{d'}^3(\overline{r_2}r_2'\tilde{r}_2+\overline{r_1}r_1'
    \tilde{r}_1)}{\tilde{r}_1\tilde{r}_2}\right)\\ 
&= \e\left(\frac{-A \cdot \overline{c'}^2{d'}^3}{\tilde{q}}\right), 
\end{align*}
where $\overline{c'}c' \equiv 1 \bmod{\tilde{q}}$. Furthermore, by the
multiplicativity of the Jacobi symbol, we have 
\begin{align*}
 \left(\frac{3^{f_3}\overline{r_2}d_1}{\tilde{r}_1}\right) 
 \left(\frac{3^{f_3}\overline{r_1}d_2}{\tilde{r}_2}\right) &= \left(\frac{3^{f_3}\overline{r_2}d'r_2'}{\tilde{r_1}}\right) 
 \left(\frac{3^{f_3}\overline{r_1}d'r_1'}{\tilde{r}_2}\right)\\ 
&=  \left(\frac{3^{f_3}\overline{r_2}d'r_2 / \tilde{r}_2}{\tilde{r}_1}\right)  
 \left(\frac{3^{f_3}\overline{r_1}d'r_1 / \tilde{r}_1}{\tilde{r}_2}\right)\\ 
 &= \left(\frac{3^{f_3}d'}{\tilde{r}_1}\right)  \left( \frac{3^{f_3}d'}{\tilde{r}_2}\right)  \left(\frac{\overline{r_2}r_2/\tilde{r}_2}{\tilde{r}_1}\right)  \left(\frac{\overline{r_1}r_1/\tilde{r}_1}{\tilde{r}_2}\right)\\ &= \left(\frac{3^{f_3}d'}{\tilde{q}}\right)  
 \left(\frac{\tilde{r}_1}{\tilde{r}_2}\right)  \left(\frac{\tilde{r}_2}{\tilde{r}_1}\right).
 \end{align*}
 Moreover, by quadratic reciprocity,
$$
 \left(\frac{\tilde{r}_1}{\tilde{r}_2}\right) \left(\frac{\tilde{r}_2}{\tilde{r}_1}\right)\epsilon_{\tilde{r}_1}\epsilon_{\tilde{r}_2} = \epsilon_{\tilde{r}_1\tilde{r}_2}=\epsilon_{\tilde{q}}.
$$
 Combining these with \eqref{compfac} we obtain \eqref{compositemoduliexplicitsimple}. 
 
Next, we turn to the case when $t_2=0$ and $t_3>0$. Note that 
$f_3=e_3=1$. Using \eqref{mult1}, we have 
$$
E_{f_2,f_3}(c,d;q)=E(2^{f_2}c,3d;q)=E(2^{f_2}3^{-t_3}c,3^{1-t_3}d;q_1)E(2^{f_2}c\overline{q_1},3d\overline{q_1};3^{t_3}).
$$
Applying \eqref{compositemoduliexplicitsimple} and Lemma \ref{explicitevaluation} for $p=3$ to the right-hand side, we obtain
\begin{align*}
E(2^{f_2}c,3^{f_3}d;q) =~& 
\frac{q_1}{\sqrt{\tilde{q}_1}} \cdot \e\left(\frac{-A\cdot
    3^{-t_3}\cdot  \overline{c_1}^2d_1^3}{\tilde{q}_1}\right)
\left(\frac{3^{1-t_3}d_1}{\tilde{q}_1}\right)
\epsilon_{\tilde{q}_1}\\ &\times 
3^{t_3-(u_3-1)/2}\cdot \e\left(\frac{-A \cdot \overline{q_1}\cdot \overline{c_2}^2d_2^3}{3^{u_3}}\right) 
\left(\frac{\overline{q_1}d_2}{3^{u_3-1}}\right) \epsilon_{w},
\end{align*}
where $w=3^{u_3-e_3}$ and 
$$
c_1=\frac{c}{q_1'}, \quad d_1=\frac{d}{q_1'}, \quad c_2=\frac{c}{3^{s_3}}, \quad d_2=\frac{d}{3^{s_3}}.
$$
Recall the definitions in \eqref{df*} and \eqref{df**}. 
Proceeding along the same lines as in the case $(q,6)=1$, we deduce that
$$
E_{f_2,f_3}(c,d;q)=
 \frac{q}{\sqrt{\tilde{q}/3}}\cdot  \e\left(\frac{-A \cdot
     \overline{c'}^2{d'}^3}{\tilde{q}}\right)
 \left(\frac{d'}{w\tilde{q}_1}\right) \epsilon_{w\tilde{q}_1}, 
$$
which establishes \eqref{compositemoduliexplicit***} for the case
$t_2=0$ and $t_3>0$ and so completes its proof. 

From \eqref{compositemoduliexplicit***}, it follows that if $t_2=0$ 
then \eqref{compositemoduliexplicit} holds with
$$
P=\sqrt{3^{e_3}} \left(\frac{3^{f_3-e_3}}{w\tilde{q}_1}\right)
\left(\frac{d'}{w}\right) \epsilon_{w\tilde{q}_1}. 
$$
Thus $|P|\leq \sqrt{3}$ and one sees that $P$ depends at most on 
$\tilde{q}, f_3$ and the residue class of $d'$ modulo $3$.

It remains to consider the case when $t_2>0$, for which 
$f_2=e_2=1$. Using  \eqref{mult1} we  have
$$
E_{f_2,f_3}(c,d;q)=E(2c,3^{f_3}d;q)=E(2^{1-t_2}c,2^{-t_2}3^{f_3}d;3^{t_3}{q}_1)E(2\cdot 3^{-t_3}\overline{{q}_1}c,3^{f_3-t_3}\overline{{q}_1}d;2^{t_2}).
$$
Applying Lemma \ref{explicitevaluation} for $p=2$ and \eqref{compositemoduliexplicit***} to the right-hand side, we obtain
\begin{align*}
E_{f_2,f_3}(c,d;q) =~& \frac{3^{t_3}q_1}{\sqrt{3^{u_3-e_3}\tilde{q}_1}} \cdot \e\left(\frac{- A \cdot 2^{-t_2}\cdot \overline{c_1}^2{d_1}^3}{3^{u_3}\tilde{q}_1}\right) \left(\frac{2^{-t_2}3^{f_3-e_3}d_1}{w\tilde{q}_1}\right) \epsilon_{w\tilde{q}_1}\\  &\times
2^{t_2-(u_2-1)/2}\cdot \e\left(\frac{- A\cdot 3^{-t_3}\cdot \overline{q_1}\cdot \overline{c_2}^2d_2^3}{2^{u_2}}\right) 
\left\{\frac{3^{f_3-t_3}\overline{{q}_1}d_2}{2^{u_2}}\right\},
\end{align*} 
where 
$$
c_1=\frac{c}{3^{s_3}q_1'}, \quad d_1=\frac{d}{3^{s_3}q_1'}, \quad c_2=\frac{c}{2^{s_2}}, \quad d_2=\frac{d}{2^{s_2}}.
$$
Proceeding along the same lines as in the case $(q,6)=1$, we deduce that
$$
E_{f_2,f_3}(c,c;q)=
\frac{q}{\sqrt{\tilde{q}/(2\cdot 3^{e_3})}}\cdot  \e\left(\frac{-A \cdot \overline{c'}^2{d'}^3}{\tilde{q}}\right) \left(\frac{2^{u_2}3^{f_3-e_3}d'}{w\tilde{q}_1}\right) \epsilon_{w\tilde{q}_1} 
\cdot \left\{\frac{3^{f_3-u_3}\overline{\tilde{q}_1}d'}{2^{u_2}}\right\}.
$$
It follows that \eqref{compositemoduliexplicit} holds with
$$
P=\sqrt{2\cdot 3^{e_3}} \cdot \left(\frac{2^{u_2}3^{f_3-e_3}}{w\tilde{q}_1}\right) \left(\frac{d'}{w}\right) \epsilon_{w\tilde{q}_1} \cdot \left\{\frac{3^{f_3-u_3}\overline{\tilde{q}_1}d'}{2^{u_2}}\right\}
$$
in this case. Recalling the definition of the 
residue symbol \eqref{strangesymbol} one easily deduces that $|P|\leq 2\sqrt{3}$
and $P$ depends only on $\tilde{q}, f_3$ and the residue of 
$d'$ modulo $8 \cdot 3^{e_3}$.
\end{proof}

Since it will turn out to be very convenient to have the condition 
$$
2^{6e_2}3^{5e_3}\rad(q_1)^2\mid \frac{q}{(c,q)}
$$
in part (i) of Lemma \ref{compositeexplicitevaluation1} replaced by the weaker condition 
$$
\rad(q_1)^2\mid \frac{q}{(c,q)},
$$
we provide a result based on the previous one.  
Let $v_p(n)$ the $p$-adic
valuation of a natural number $n$ and recall the definition 
\eqref{rpdef} of $r(p)$. Then we have the following result. 

\begin{lemma} \label{composi}  
Let $c,d,q\in \ZZ$ such that $q>0$,  let 
$f_2,f_3\in \{0,1\}$
and assume that \eqref{generalassump} holds, with 
$
E_{f_2,f_3}(c,d;q)\not=0.
$
Suppose that $q=2^{t_2}3^{t_3}q_1$, where $(q_1,6)=1$. Set
$$
q':=(c,q), \quad \tilde{q}:=\frac{q}{q'}, \quad q_1':=(c,q_1), \quad 
\tilde{q}_1:=\frac{q_1}{q_1'}, \quad c':=\frac{c}{(c,q)}=\frac{c}{q'},
\quad d':=\frac{d}{(d,q)}. 
$$
Suppose that $\rad(q_1)^2\mid \tilde{q}$. Let 
$$
q_0:=\prod\limits_{\substack{p\\ 
v_p(\tilde{q})\ge r(p)}} p^{v_{p}(q)}, \quad q_0':=(c,q_0),\quad
\tilde{q}_0:=\frac{q_0}{q_0'},
$$
and 
\begin{equation} \label{jdef}
j_p:=v_p((d,q))-v_p(q').
\end{equation}
Then we have
\begin{equation} \label{bruch}
\frac{\tilde{q}}{\tilde{q}_0}=2^{l_2}3^{l_3} \quad \mbox{for some } l_2,l_3\in \mathbb{Z} \mbox{ with } 0\le l_2\le 5 \mbox{ and } 0\le l_3\le 4, 
\end{equation}
\begin{equation} \label{dasistzuzeigen}
|j_2|\le 5, \quad  |j_3|\le 4, \quad j_p=0 \mbox{ if } p>3.
\end{equation}
Furthermore
\begin{equation} \label{compositemoduliexplicitnew}
E_{f_2,f_3}(c,d;q)= \frac{q}{\sqrt{\tilde{q}}} \cdot \e\left(\frac{-  2^{g_2}3^{g_3} \cdot  \overline{c'}^2{d'}^3}{\tilde{q}_0}\right) \left(\frac{d'}{\tilde{q}_1}\right) \cdot Q,
\end{equation}
where $g_2,g_3\in \mathbb{Z}$ depend at most on $q$, $q'$, $f_2$,
$f_3$ and satisfy 
$$
|g_2|\le 20, \quad |g_3|\le 17, \quad g_p= 0 \mbox{ if }
p\mid\tilde{q}_0 \mbox{ for } p=2,3,   
$$
and $Q$ depends at most on $q$, $q'$, $f_2$, $f_3$, $j_2$, $j_3$ 
and the residue classes of $c'$ and $d'$ modulo $2^{5}3^4$
and satisfies  $Q=O(1)$, with an absolute implied constant. 
\end{lemma}

\begin{proof} The assertion in \eqref{bruch} follows from the
  definitions of $\tilde{q}$ and $\tilde{q}_0$ and the hypothesis that
$\rad(q_1)^2\mid \tilde{q}$.  The relations in
  \eqref{dasistzuzeigen} are consequences of this assumption
and part (ii) and
  \eqref{noteqzero1} in part (i) of Lemma
  \ref{compositeexplicitevaluation1}. It remains to show 
\eqref{compositemoduliexplicitnew} with $g_2$, $g_3$ and $Q$ having
the claimed properties. 

By the definition of $q_0$ and the assumption $\rad(q_1)^2\mid \tilde{q}$, we have 
\begin{equation} \label{qfacto}
q=2^{h_2}3^{h_3}q_0,
\end{equation} 
with 
$$
h_p=\begin{cases}
t_p & \mbox{if $p\nmid q_0$,}\\
0 & \mbox{if $p\mid q_0$.}
\end{cases}
$$
In particular 
$(2^{h_2}3^{h_3},q_0)=1$. 
From \eqref{mult1} it therefore follows that  
\begin{equation} \label{multapplied}
\begin{split}
E_{f_2,f_3}(c,d;q)
&=E_{f_2,f_3}(2^{-h_2}3^{-h_3}c,2^{-h_2}3^{-h_3}d;q_0)E_{f_2,f_3}(\overline{q_0}c,\overline{q_0}d;2^{h_2}3^{h_3})\\
&=E^{(1)}E^{(2)},
\end{split}
\end{equation}
say.
We can  apply Lemma \ref{compositeexplicitevaluation1} to the first exponential sum on the right-hand side, obtaining 
$$
(c,q_0)=q_0'=(d,q_0)
$$
since this exponential sum is non-zero by hypothesis. Furthermore we
have 
$$
E^{(1)}= 
\frac{q_0}{\sqrt{\tilde{q}_0}} \cdot \e\left(\frac{-A\cdot
    2^{-h_2}3^{-h_3}\overline{c''}^2{d''}^3}{\tilde{q}_0}\right)
\left(\frac{2^{h_2}3^{h_3}d''}{\tilde{q}_1}\right) \cdot P, 
$$
where $A$ is defined as in \eqref{wAdefi}, 
$$
c''=\frac{c}{q_0'}, \quad d''=\frac{d}{q_0'},
$$
and $P=O(1)$ depends only on $\tilde{q}_0$, $f_3$ and the (well-defined) residue
class of $2^{-h_2}3^{-h_3}d''$ modulo $8^{e_2}3^{e_3}$, with
$e_p:=\sign(v_p(\tilde{q}_0))$. Moreover, we may write 
$$
c''=c'\cdot \frac{q'}{q_0'} \quad \mbox{and} \quad d''=d' \cdot \frac{(d,q)}{q_0'}.
$$
By \eqref{jdef} and \eqref{dasistzuzeigen}, we have
$$
d''=d'\cdot \frac{q'}{q_0'}\cdot 2^{j_2}3^{j_3}
$$
with $|j_2|\le 5$ and $|j_3|\le 4$. 
The construction of $q_0$ ensures that we have $j_p=0$ if $p\mid
\tilde{q}_0$ for $p=2,3$. Combining this with our work so far  
we obtain 
\begin{align*}
E^{(1)}=
\frac{q_0}{\sqrt{\tilde{q}_0}} \cdot \e\left(\frac{-A\cdot
    2^{3j_2-h_2}3^{3j_3-h_3}\cdot \frac{q'}{q_0'}\cdot
    \overline{c'}^2{d'}^3}{\tilde{q}_0}\right)
\left(\frac{2^{h_2+j_2}3^{h_3+j_3}\cdot \frac{q'}{q_0'} \cdot
    d'}{\tilde{q}_1}\right) \cdot P. 
\end{align*}
In view of \eqref{qfacto} we have
$$
\frac{q'}{q_0'}=\frac{(c,2^{h_2}3^{h_3}q_0)}{(c,q_0)}=(c,2^{h_2}3^{h_3}),
$$
whence 
$0\leq v_p(q'/q_0')\leq h_p$. 
It is now easy to deduce that
$v_p(q'/q_0')=0$ if $p>3$ and 
\begin{equation} \label{ohoh}
\max\{0,h_2-5\}\le v_2\left(\frac{q'}{q_0'}\right)\le h_2, 
\quad \max\{0,h_3-4\}
\le v_3\left(\frac{q'}{q_0'}\right)\le h_3.
\end{equation}
It follows that 
\begin{equation} \label{mops}
E^{(1)} =
\frac{q_0}{\sqrt{\tilde{q}_0}} \cdot
\e\left(\frac{-2^{g_2}3^{g_3}\cdot
    \overline{c'}^2{d'}^3}{\tilde{q}_0}\right)
\left(\frac{d'}{\tilde{q}_1}\right) \cdot P', 
\end{equation}
where $g_2$ and $g_3$ have the properties claimed in Lemma
\ref{composi} (where we also use that $f_p=1$ if $p\mid \tilde{q}_0$
for $p=2,3$), and $|P'|\leq 2\sqrt{3}$ depends only on $q$, $q'$, $f_3$, $j_2$,
$j_3$ and the residue class of $d'$ modulo $8^{e_2}3^{e_3}$. Here and
later, note that $q$, $q'$ determine $\tilde{q}$, $q_0$, $q_0'$,
$\tilde{q}_0$, $q_1$, $q_1'$, $\tilde{q}_1$, $h_2$ and $h_3$. 
  
Now we turn to the second exponential sum on the right-hand side of
\eqref{multapplied}. We write  
$$
c=q'c', \quad d=2^{j_2}3^{j_3}q'd', 
$$
and hence
\begin{align*}
E^{(2)}&=E_{f_2,f_3}(\overline{q_0}q'c',2^{j_2}3^{j_3}\overline{q_0}q'd';2^{h_2}3^{h_3})\\ 
&=  E_{f_2,f_3}\left(\overline{\tilde{q}_0}\cdot \frac{q'}{q_0'}\cdot c',\overline{\tilde{q}_0} \cdot 2^{j_2}3^{j_3} \cdot \frac{q'}{q_0'}\cdot d';2^{h_2}3^{h_3}\right),
\end{align*}
where 
$2^{j_2}3^{j_3} \cdot q'/q_0' \in \mathbb{N}.$
In particular $|E^{(2)}|\leq 2^{h_2}3^{h_3}$ and $E^{(2)}$ depends
only on $q,q',f_2,f_3,j_2,j_3$ and the residue classes of $c',d'$ modulo
$2^53^4$. To see the latter claim we note that we are only interested
in 
$\overline{\tilde{q}_0}C\cdot c'$ and $\overline{\tilde{q}_0}D \cdot 
d'$ modulo $2^{h_2}3^{h_3}$, where $C=q'/q_0'$ and 
$D=2^{j_2}3^{j_3}  q'/q_0'$. But then the claim follows from
\eqref{dasistzuzeigen} and \eqref{ohoh}. Once combined with
\eqref{mops} this therefore yields \eqref{compositemoduliexplicitnew} 
with
$g_2,g_3$ and $Q$ having the properties recorded in the lemma. 
\end{proof}

\section{Weighted solutions of $ax^2+by^3\equiv 0
  \bmod{q}$}\label{s:thec}

Let $\Gamma$ denote the Gaussian weight in
\eqref{gaussiandef}. 
We note that the Fourier transform $\hat\Gamma$ of $\Gamma$ satisfies 
$$
\hat\Gamma(x):= \int\limits_{-\infty}^{\infty} \Gamma(t)\e(-tx)\dif t
=\exp(-\pi x^2)=\Gamma(x) 
$$
for all $x\in \mathbb{R}$. 
Our starting point for the analysis of the counting
functions 
$M(B,\XX,\YY;a,b;q)$ and 
$M(\XX,\YY;a,b;q)$ from  \S \ref{s:intro} lies with an  initial
investigation of the weighted sum
\begin{equation} \label{SGammadef}
 \mathcal{S}(x_0,y_0,X,Y;a,b;q)  
:= \sum\limits_{\substack{x,y\in \mathbb{Z}  \\  (xy,q)=1 \\
    ax^2+by^3\equiv 0 \bmod{ q}}}
\Gamma\left(\frac{x-x_0}{X}\right)
\Gamma\left(\frac{y-y_0}{Y}\right),  
\end{equation}
for given $x_0,y_0\in \RR$ and $X,Y\geq 2$. 
We will write 
$\mathcal{S}=\mathcal{S}(x_0,y_0,X,Y;a,b;q)$ for short and we note
that without loss of generality we may assume that $a,b>0$ in all that
follows. 

We analyse $\mathcal{S}$ via Poisson summation, the first task
being to break the sum into residue classes modulo $q$. Thus we may write
$$
 \mathcal{S}=  \sideset{}{^*}\sum\limits_{\substack{c,d \bmod{q}\\ ac^2+bd^3 \equiv 0 \bmod{q}}} 
\sum\limits_{x\equiv c \bmod{q}} \Gamma\left(\frac{x-x_0}{X} \right)
\sum\limits_{y\equiv d \bmod{q}} \Gamma\left(\frac{y-y_0}{Y}\right), 
$$
where the asterisk attached to the summation symbol indicates coprimality
of the variables $c,d$ with $q$.
Applying Poisson summation after a linear change of variables to the
sums over $x$ and $y$ on the right-hand side, we obtain 
\begin{align*}
\sum\limits_{x\equiv c \bmod{q}} \Gamma\left(\frac{x-x_0}{X} \right)
&= \frac{X}{q} \sum\limits_{m\in \mathbb{Z}} \Gamma
\left(\frac{mX}{q}\right)   \e\left(\frac{(c-x_0)m}{q}\right), \\
\quad \sum\limits_{y\equiv d \bmod{q}}
\Gamma\left(\frac{y-y_0}{Y}\right) &= \frac{Y}{q} \sum\limits_{n\in
  \mathbb{Z}} \Gamma \left(\frac{nY}{q}\right)
\e\left(\frac{(d-y_0)n}{q}\right). 
\end{align*}
Hence 
\begin{equation} \label{SRaftertrans}
\mathcal{S}= \frac{XY}{q^2} \sum\limits_{m\in \mathbb{Z}} \sum\limits_{n \in \mathbb{Z}} \e\left(-\frac{mx_0+ny_0}{q}\right) \Gamma\left(\frac{mX}{q}\right) \Gamma\left(\frac{nY}{q}\right) \mathcal{E}(m,n;q),
\end{equation}
where
$$
\mathcal{E}(m,n;q):= \sideset{}{^*}\sum\limits_{\substack{c,d \bmod{ q }\\ ac^2+bd^3\equiv 0 \bmod{ q }}} \e\left(\frac{cm+dn}{q}\right).
$$

Inspired by the work of Pierce \cite{pierce}, it will be convenient to 
rewrite the exponential sum $\mathcal{E}(m,n;q)$
as a sum over a single variable. 
Recall that $(ab,q)=1$. We consider the map
$$
f : \left\{(c,d) \in \mathbb{Z}_q^* \times \mathbb{Z}_q^* : ac^2+bd^3\equiv 0 \bmod q\right\} \rightarrow \mathbb{Z}_q^*,
$$
defined by
$$
f(c,d):=-ac\overline{d}.
$$
In the above, $\mathbb{Z}_q^*$ is the group of units of $\mathbb{Z}_q=\mathbb{Z}/q\mathbb{Z}$, and  $\overline{d}$ is the multiplicative inverse of $d\in \mathbb{Z}_q^*$. This map is bijective. Indeed, we check that 
$$
f^{-1}(z):=\left(\overline{a^2b} z^3,-\overline{ab} z^2\right)
$$
is the inverse map. Thus
$$
f^{-1} \circ f(c,d)=\left(-\overline{a^2b} \cdot a^3c^3 \overline{{d}^3},-\overline{ab} \cdot a^2c^2 \overline{{d}^2}\right)=(c,d)
$$
and 
$$
f\circ f^{-1}(z)=-a \cdot\overline{a^2b}\cdot z^3 \cdot \overline{-\overline{ab}\cdot z^2}=z.
$$
Hence we may parametrize $c$ and $d$ in the definition of $\mathcal{E}(m,n;q)$ by
$c=\overline{a^2b} z^3$ and $d=-\overline{ab} z^2$ for $
z\in \mathbb{Z}_q^*$. It follows that
$$
\mathcal{E}(m,n;q)=\sideset{}{^*} \sum\limits_{z \bmod q} \e\left(\frac{\overline{a^2b} mz^3-\overline{ab} nz^2}{q}\right).
$$
Making the change of variables $z\rightarrow abz$, the above can also be written in the form
\begin{equation} \label{relexp}
\mathcal{E}(m,n;q)=\sideset{}{^*} \sum\limits_{z \bmod q}
\e\left(\frac{ab^2mz^3-ab nz^2}{q}\right)=E(ab^2m,abn;q),
\end{equation}
in the notation of \eqref{Eq0}.

\bigskip

We now split the sum $\mathcal{S}$ in \eqref{SRaftertrans}
into three. We let $\mathcal{S}_{0}$ denote the
contribution of $m=0$ and $n=0$, we let $\mathcal{S}_{1}$ denote the contribution of
$m\not=0$ and $n$ arbitrary, and we let 
$\mathcal{S}_{2}$ denote the contribution of
$m=0$ and $n\not=0$. The treatment of $\mathcal{S}_{0}$ and
$\mathcal{S}_{2}$ is straightforward.  
 
The term $\mathcal{S}_{0}$ will be the main term. Obviously, we have
$\mathcal{E}(0,0;q)=\varphi(q)$ and therefore
\begin{equation} \label{S00ev}
\mathcal{S}_{0}= \frac{\varphi(q)}{q^2} \cdot XY.
\end{equation}
Since $(ab,q)=1$, it follows from Lemma
\ref{compositeexplicitevaluation1}(ii) that
$\mathcal{E}(0,n;q)=0$ unless $(q/\rad(q))\mid 2^43^3n$. Setting 
$2^43^3n=n_1\cdot q/\rad(q)$, we therefore get
\begin{align*}
\mathcal{S}_{2} &= \frac{XY}{q^2} \sum\limits_{n\not=0}  \e\left(-\frac{ny_0}{q}\right)\Gamma\left(\frac{nY}{q}\right)   \mathcal{E}(0,n;q) \\
&\ll \frac{XY}{q^2} \sum\limits_{n_1\not= 0}  \left| \Gamma\left(\frac{n_1Y}{2^43^3\rad(q)}\right)\right|  \left|\mathcal{E}
\left(0,\frac{n_1q}{2^43^3\rad(q)};q\right)\right|, 
\end{align*}
where we use the convention that $\mathcal{E}(0,z;q)=0$ if $z\not\in \mathbb{Z}$.
Using Lemma \ref{generalexpestimate}, we further have
$$
\left|\mathcal{E}\left(0,\frac{n_1q}{2^43^3\rad(q)};q\right)\right|  \ll \frac{q}{\rad(q)^{1/2}} \cdot (n_1,\rad(q))^{1/2}\cdot 2^{\omega(q)}. 
$$
Combining these estimates with \eqref{wellknown'}, we get
\begin{equation} \label{S10est}
\mathcal{S}_{2}\ll \frac{(2+\ve)^{\omega(q)}\rad(q)^{1/2}}{q} \cdot X.
\end{equation} 
Here we have observed that $(1+\ve)^{\omega(\rad(q))}2^{\omega(q)}\leq
(2+\ve)^{\omega(q)}$, where henceforth we adhere to the convention that $\ve$ may
take different values from appearance to appearance. 

We now turn to the estimation 
 \begin{equation} \label{S11}
\mathcal{S}_{1}= \frac{XY}{q^2} \sum\limits_{m\not=0} \sum\limits_{n} \e\left(-\frac{mx_0+ny_0}{q}\right) \Gamma\left(\frac{mX}{q}\right) \Gamma\left(\frac{nY}{q}\right) \mathcal{E}(m,n;q).
\end{equation} 
We always keep in mind that by \eqref{relexp} and the fact that 
$(ab,q)=1$, statements about $E(c,d;q)$ translate into corresponding statements about $\mathcal{E}(m,n;q)$.
Using Lemma \ref{badqs} we deduce  that 
$$
\mathcal{S}_{1}= \frac{XY}{q^2} \sum\limits_{m\not=0} \sum\limits_{n}
\e\left(-\frac{2^{f_2}mx_0+3^{f_3}ny_0}{q}\right)
\Gamma\left(\frac{2^{f_2}mX}{q}\right)
\Gamma\left(\frac{3^{f_3}nY}{q}\right) \mathcal{E}_{f_2,f_3}(m,n;q), 
$$
where $f_i:=w_i(q)$ for $i=2,3$, defined as in \eqref{w2def} and \eqref{w3def}, and
$$
\mathcal{E}_{f_2,f_3}(m,n;q):=\mathcal{E}(2^{f_2}m,3^{f_3}n;q).
$$
 We split $\mathcal{S}_{1}$ according to
 the greatest common divisor of $m$ and $q$. Thus we write 
$$
\mathcal{S}_{1}= \frac{XY}{q^2} \sum\limits_{d\mid q}
\sum\limits_{\substack{m\not=0\\ (m,q)=d}} \sum\limits_{n}
\e\left(-\frac{2^{f_2}mx_0+3^{f_3}ny_0}{q}\right)
\Gamma\left(\frac{2^{f_2}mX}{q}\right)
\Gamma\left(\frac{3^{f_3}nY}{q}\right)
\mathcal{E}_{f_2,f_3}(m,n;q). 
$$
Throughout the sequel, we denote by $\hat{k}$ the largest divisor
coprime to $6$ of $k\in \mathbb{N}$. Now, we uniquely factorise $q$ and
$d$ in the form
\begin{align*}
& q=r_1r_2,\quad  \mbox{$d=d_1d_2$ where $d_1\mid r_1, d_2\mid r_2$},
\\
& (6r_1,r_2)=1=(d_1,d_2), \quad \rad(\hat{r}_1)^2\mid
(r_1/d_1), \quad  \mu^2(r_2/d_2)=1. 
\end{align*}
The existence and uniqueness of this factorisation is obvious in the
case when $q$ is a prime power and then follows for the general case
by multiplicativity. We may now write
 \begin{equation} \label{split1}
 \begin{split}
\mathcal{S}_{1} =~&  \frac{XY}{q^2}  \sum\limits_{\substack{(6r_1,r_2)=1\\ r_1r_2=q}} \ \sum\limits_{\substack{d_1\mid r_1,\  d_2\mid r_2\\ \mbox{\scriptsize rad}(\hat{r}_1)^2\mid e_1\\ \mu^2(e_2)=1}} \
\sum\limits_{\substack{m\not=0\\ (m,r_1)=d_1\\ (m,r_2)=d_2}}
\sum\limits_{n} \e\left(-\frac{2^{f_2}mx_0+3^{f_3}ny_0}{q}\right)
\\ & \times 
\Gamma\left(\frac{2^{f_2}mX}{q}\right) \Gamma\left(\frac{3^{f_3}nY}{q}\right) \mathcal{E}_{f_2,f_3}(m,n;q),
\end{split}
\end{equation} 
where
\begin{equation} \label{q1q2def}
e_i:=\frac{r_i}{d_i} \quad \mbox{for } i=1,2.
\end{equation}

Using \eqref{mult1} we factorise the exponential sum as
\begin{equation} \label{expsplit}
\mathcal{E}_{f_2,f_3}(m,n;q)=\mathcal{E}_{f_2,f_3}(m\overline{r_2},n\overline{r_2};r_1)\mathcal{E}_{f_2,f_3}(m\overline{r_1},n\overline{r_1};r_2).
\end{equation}
We shall use Lemma \ref{composi} to evaluate the first exponential sum
on the right-hand side.  To this end we recall the definition of 
 $r(p)$ from \eqref{rpdef}. Then we set 
\begin{equation} \label{e0def}
r_0:=\prod\limits_{\substack{p\\ 
v_p(e_1)\ge r(p)}} p^{v_{p}(r_1)}, \quad d_0:=(m,r_0),\quad
e_0:=\frac{r_0}{d_0},
\end{equation}
and 
\begin{equation}\label{m1n1def}
d_1':=(n,r_1),\quad  e_1':=\frac{r_1}{d_1'}, \quad m_1:=\frac{m}{d_1}, \quad n_1:=\frac{n}{d_1'}.
\end{equation}
Now, from \eqref{relexp} and Lemma \ref{composi}, we deduce that
if $\mathcal{E}_{f_2,f_3}(m\overline{r_2},n\overline{r_2};r_1)\neq0$
then 
\begin{equation} \label{d1primecondi}
d_1'=2^{j_2}3^{j_3}d_1\in \mathbb{N} \mbox{ for some } j_2,j_3\in \mathbb{Z} \mbox{ with } |j_2|\le 5 \mbox{ and } |j_3|\le 4,
\end{equation}
and, if this is the case, then
\begin{equation} \label{firstexplicit}
\mathcal{E}_{f_2,f_3}(m\overline{r_2},n\overline{r_2};r_1)= \sqrt{r_1d_1} \cdot \e\left(\frac{-2^{g_2}3^{g_3}\cdot a \cdot \overline{br_2}\cdot \overline{m_1}^2{n_1}^3}{e_0}\right) \left(\frac{n_1}{\hat{e}_1}\right)\cdot Q, 
\end{equation}
where $g_2,g_3\in \mathbb{Z}$ depend at most on $d_1$, $d_1'$, $r_1$, $f_2$, $f_3$ and satisfy
$$
|g_2|\le 20, \quad |g_3|\le 17, \quad g_p= 0 \mbox{ if }
p\mid e_0 \mbox{ for } p=2,3,   
$$
and $Q$ depends at most on $a$, $b$, $d_1$, $d_1'$, $r_1$, $r_2$,
$f_2$, $f_3$ and the residue classes of $m_1$ and $n_1$ modulo
$2^{5}3^4$ and satisfies the bound 
$Q=O(1)$. Moreover, from \eqref{bruch}, it follows that
\begin{equation} \label{e0e1rel}
e_1=2^{l_2}3^{l_3}e_0\quad \mbox{for some } l_2,l_3\in \mathbb{Z} \mbox{ with } 0\le l_2\le 5\mbox{ and } 0\le l_3\le 4.  
\end{equation}

Now we turn to the second exponential sum on the right-hand side of
\eqref{expsplit}. Since $r_2/d_2=e_2$ is square-free and $d_2\mid m$,
we have that $(r_2/\rad(r_2))\mid m$. Moreover,  
from Lemma \ref{compositeexplicitevaluation1}(ii) and the fact that 
$(r_2,6)=1$, we deduce 
$$
\mathcal{E}_{f_2,f_3}(m\overline{r_1},n\overline{r_1};r_2)\not= 0 \Rightarrow r_2/(n,r_2) \mbox{ is square-free}
\Rightarrow (r_2/\rad(r_2))\mid n.
$$
We set
\begin{equation} \label{sdef}
d_2':=\frac{r_2}{\rad(r_2)}
\end{equation}
and note that $d_2'\mid d_2$ since $e_2=r_2/d_2$ is supposed to be square-free. We further set
\begin{equation} \label{primeprime}
d_2'':=\frac{d_2}{d_2'}=\frac{d_2\rad(r_2)}{r_2}=\frac{\rad(r_2)}{e_2}.
\end{equation}
In particular $(d_2'',e_2)=1$ since $\rad(r_2)$ is square-free. Now it follows that
\begin{equation} \label{secondexplicit}
\mathcal{E}_{f_2,f_3}(m\overline{r_1},n\overline{r_1};r_2)=d_2' \mathcal{E}_{f_2,f_3}(m_2\overline{r_1},n_2\overline{r_1};\rad(r_2)),
\end{equation}
where 
\begin{equation} \label{m2n2def}
m_2:=\frac{m}{d_2'}, \quad n_2:=\frac{n}{d_2'}.
\end{equation}
We also note that
\begin{equation} \label{stardef}
d_2^*:=\frac{r_2d_2^2}{(d_2')^3}={d_2''}^2  \rad(r_2) = {d_2''}^3e_2 \in \mathbb{N},
\end{equation}
which we will need in the following.

From \eqref{split1} and the above considerations, we conclude that $m=d_1d_2u$ for some nonzero integer $u$ with $(u,e_1e_2)=1$, and $n=d_1'd_2'v$ for some integer $v$ with $(v,e_1')=1$. Note that if $n=0$ then $e_1'=1$ and so $(0,e_1')=1$. By \eqref{m1n1def}, \eqref{firstexplicit}, \eqref{sdef}, \eqref{stardef} and the properties of the Jacobi symbol, we have
\begin{align*}
\mathcal{E}_{f_2,f_3}(m\overline{r_2},n\overline{r_2};r_1) &=
\sqrt{r_1d_1} \cdot \e\left(\frac{-2^{g_2}3^{g_3} a \cdot
    \overline{br_2} \cdot \overline{d_2u}^2 (d_2'v)^3}{e_0}\right)
\left(\frac{d_2'v}{\hat{e}_1}\right) \cdot Q\\ 
&= \sqrt{r_1d_1} \cdot \e\left(\frac{-2^{g_2}3^{g_3} \cdot a \cdot \overline{bd_2^*} \cdot \overline{u}^2 v^3}{e_0}\right) \left(\frac{v}{e}\right) \cdot Q',
 \end{align*}
where $e$ is the square-free kernel of $\hat{e}_1$, the unique square-free number for which $\hat{e}_1/e$ is a perfect square,  and
$$
Q':=\left(\frac{d_2'}{e}\right)\cdot Q=O(1)
$$
and depends at most on $a$, $b$, $d_1$, $d_1'$,  $r_1$, $r_2$
and the residue classes of $u$ and $v$ modulo $2^53^4$. Note that
the dependence on $d_2'$ is really a dependence on $r_2$ by
\eqref{sdef}, and the dependence on $f_2,f_3$ is one on $r_1r_2=q$.  Moreover, we write
$$
2^{g_2}3^{g_3}\equiv A\overline{B} \bmod{e_0},
$$
where $(A,B)=1$ and $A$ and $B$ are integers of the form $2^{g_2'}3^{g_3'}$ and $2^{g_2''}3^{g_3''}$, respectively, and set
\begin{equation} \label{a1b1}
a_1:=\frac{Aa}{(Aa,Bbu^2)}, \quad b_1:=\frac{Bbu^2}{(Aa,Bbu^2)}.
\end{equation}
We note that $(a_1,b_1)=1$. Now the exponential term 
takes the shape
$$
\e\left(\frac{-2^{g_2}3^{g_3} \cdot a \cdot \overline{bd_2^*} \cdot \overline{u}^2 v^3}{e_0}\right)=
\e\left(\frac{-a_1 \cdot \overline{b_1d_2^*} \cdot v^3}{e_0}\right). 
$$

Using \eqref{q1q2def}, \eqref{m1n1def}, \eqref{primeprime},
\eqref{secondexplicit} and \eqref{m2n2def}, we have 
$$
\mathcal{E}_{f_2,f_3}(m\overline{r_1},n\overline{r_1};r_2) = d_2'
\mathcal{E}_{f_2,f_3}(\overline{r_1}d_1 d_2''
u,\overline{r_1}d_1'v;\rad(r_2))=d_2'\mathcal{E}_{f_2,f_3}(\overline{e_1}d_2''u,\overline{e_1'}v;\rad(r_2)). 
$$
We further note that $d_2''e_2=\rad(r_2)$. Since $(d_2'',e_2)=1$,  we
can use \eqref{mult1} again to factorise the last exponential sum as 
$$
\mathcal{E}_{f_2,f_3}(\overline{e_1}d_2''u,\overline{e_1'}v;\rad(r_2))=
\mathcal{E}_{f_2,f_3}(\overline{e_1}u,\overline{e_1'd_2''}v;e_2)
\mathcal{E}_{f_2,f_3}(0,\overline{e_1'e_2}v;d_2'').
$$
Combining everything, we obtain
\begin{equation} \label{split2}
\begin{split}
\mathcal{S}_{1} =~ & \frac{XY}{q^2} \sum\limits_{\substack{(6r_1,r_2)=1\\ r_1r_2=q}} \ \sum\limits_{\substack{d_1\mid r_1,\  d_2\mid r_2\\ \mbox{\scriptsize rad}(\hat{r}_1)^2\mid e_1\\ \mu^2(e_2)=1}} \ \sideset{}{'} \sum\limits_{d_1'\mid r_1} \ d_2'\sqrt{r_1d_1}   \\
&\times  \sum\limits_{\substack{u\not=0\\ (u,e_1e_2)=1}}
\e\left(-\frac{2^{f_2}d_1d_2ux_0}{q}\right)\Gamma\left(\frac{2^{f_2}d_1d_2uX}{q}\right)
S(d_1,d_1',d_2,r_1,r_2,u), 
\end{split}
\end{equation} 
where the $'$ attached to the third summation symbol on the right-hand
side encodes the condition that \eqref{d1primecondi} holds
and  $S(u)=S(d_1,d_1',d_2,r_1,r_2,u)$ is given by 
\begin{equation} \label{sumovervdef}
\begin{split}
  S(u)
  :=~& \sum\limits_{(v,e_1')=1}
  \e\left(-\frac{3^{f_3}d_1'd_2'vy_0}{q}\right)\Gamma\left(\frac{3^{f_3}d_1'd_2'vY}{q}\right)
  \e\left(\frac{-a_1 \cdot \overline{b_1d_2^*} \cdot
  v^3}{e_0}\right) \left(\frac{v}{e}\right)\\ 
  &\times    \mathcal{E}_{f_2,f_3}(\overline{e_1}u,\overline{e_1'd_2''}v;e_2)
  \mathcal{E}_{f_2,f_3}(0,\overline{e_1'e_2}v;d_2'')
  Q'_{d_1,d_1',r_1,r_2,u}( v^\flat),
\end{split}
\end{equation}
and we henceforth use $v^\flat$ to denote the residue class of 
$v$ modulo $2^{5}3^4$. 
We have further dropped the dependency of $Q'$ on $a$ and
$b$ since these are treated as fixed integers. 

\bigskip

Now our strategy is to utilise the cancellation in the sum over $v$
coming from the exponential term. 
The main obstacle is that in the generic case, the denominator
$e_0$ is very large compared to the length of the sum over $v$ (note
that the sum over $v$ can be freely truncated at
$q^{1+\varepsilon}/(d_1'd_2'Y)$ since $\Gamma$ has rapid decay). To
reduce the size of the denominator in the exponential term, we flip
the numerator and denominator by means of the identity 
\eqref{genflip}, which gives
$$
\e\left(\frac{-a_1\cdot \overline{b_1d_2^*}\cdot v^3}{e_0}\right)=\e\left(\frac{-a_1v^3}{b_1e_0d_2^*}\right) \cdot \e\left(\frac{a_1 \overline{e_0} v^3}{b_1d_2^*}\right),
$$
where $\overline{e_0}$ is the multiplicative inverse of $e_0$ modulo
$b_1d_2^*$. 
The first factor on the right-hand side will turn
out to be a slowly oscillating weight function. The cancellation in
the sum over $v$ will come from the second factor.
The sum in \eqref{sumovervdef} now takes the form
\begin{equation} \label{sumoverv}
\begin{split}
   S(u)  =~& 
\sum\limits_{(v,e_1')=1}
\Gamma\left(\frac{3^{f_3}d_1'd_2'vY}{q}\right) 
\e\left(-\frac{3^{f_3}d_1'd_2'vy_0}{q}\right) 
\e\left(\frac{-a_1v^3}{b_1e_0d_2^*}\right) \e\left(\frac{a_1
    \overline{e_0} v^3}{b_1d_2^*}\right) \\  
  & \times 
  \left(\frac{v}{e}\right) 
  \mathcal{E}_{f_2,f_3}(\overline{e_1}u,\overline{e_1'd_2''}v;e_2)
  \mathcal{E}_{f_2,f_3}(0,\overline{e_1'e_2}v;d_2'') 
Q'_{d_1,d_1',r_1,r_2,u}(v^\flat).
\end{split}
\end{equation} 
The idea is now to write this as a short sum of complete exponential
sums to modulus $2^{5}3^{4}b_1d_2^*e$.  This will be done in the next
section using the Poisson summation formula.

\section{Second application of Poisson summation } \label{s:second}

We begin by removing the coprimality condition $(v,e_1')=1$ in
\eqref{sumoverv} using M\"obius inversion. This gives
\begin{equation} \label{remcop}
S(u)=\sum\limits_{\substack{t\mid e_1'\\ (t,e)=1}} \mu(t)\left(\frac{t}{e}\right) S(u,t),
\end{equation}
where $S(u,t)=S(d_1,d_1',d_2,r_1,r_2,u,t)$ is given by 
\begin{equation} \label{sumovervt}
\begin{split}
  S(u,t) =~ &
  \sum\limits_{v} \Gamma\left(\frac{3^{f_3}d_1'd_2'tvY}{q}\right)
  \e\left(-\frac{3^{f_3}d_1'd_2'tvy_0}{q}\right)
  \e\left(\frac{-a_1t^3v^3}{b_1e_0d_2^*}\right)
\e\left(\frac{a_1 \overline{e_0} t^3v^3}{b_1d_2^*}\right)
\\
  &\times 
  \left(\frac{v}{e}\right) 
  \mathcal{E}_{f_2,f_3}(\overline{e_1}u,\overline{e_1'd_2''}tv;e_2)
 \mathcal{E}_{f_2,f_3}(0,\overline{e_1'e_2}tv;d_2'') 
  Q''(v^\flat),
\end{split}
\end{equation}
with 
$$
Q''(v^\flat)=Q''_{d_1,d_1',r_1,r_2,u,t}(v^\flat):=Q'_{d_1,d_1',
r_1,r_2,u}(t^\flat v^\flat).
$$
To simplify the notations, we 
define a function $\Psi : \mathbb{R} \rightarrow \mathbb{C}$ by
$$
\Psi\left(\frac{3^{f_3}d_1'd_2'tY}{q}\cdot v\right)
:=\Gamma\left(\frac{3^{f_3}d_1'd_2'tY}{q}\cdot v\right) 
\e\left(-\frac{3^{f_3}d_1'd_2'ty_0}{q} \cdot v\right) 
\e\left(\frac{-a_1t^3}{b_1e_0d_2^*}\cdot v^3\right). 
$$
Equivalently, if $C:=2^{l_2-3j_2}3^{l_3-3j_3-3f_3}$, then 
\begin{equation} \label{Phidef2}
\begin{split}
\Psi(z) &= \Gamma(z) \e\left(-\frac{3^{f_3}d_1'd_2'ty_0}{q} \cdot 
\frac{q}{3^{f_3}d_1'd_2'tY} \cdot z\right) \e\left(-\frac{a_1t^3}{b_1e_0d_2^*}  \cdot \left(\frac{q}{3^{f_3}d_1'd_2'tY}\right)^3 \cdot z^3\right)\\ 
&= \Gamma(z) \e\left(-\frac{y_0}{Y} \cdot z\right)  \e\left(-C\cdot \frac{a_1q^2}{b_1d_1^2d_2^2Y^3} \cdot z^3 \right),
\end{split}
\end{equation}
where in the last equation we have used \eqref{q1q2def},
\eqref{e0def}, \eqref{d1primecondi}, \eqref{e0e1rel}, \eqref{stardef}
and $q=r_1r_2$. We note that $\Psi$ depends on
$d_1,d_1',d_2,r_1,r_2,u,y_0$ and $Y$. We may now write 
\begin{align*}
S(u,t) 
=~&  \sum\limits_{v} \Psi\left(\frac{3^{f_3}d_1'd_2'tY}{q} \cdot
  v\right) 
 \e\left(\frac{a_1 \overline{e_0} t^3v^3}{b_1d_2^*}\right)  \left(\frac{v}{e}\right) \\
& \times \mathcal{E}_{f_2,f_3}(\overline{e_1}u,\overline{e_1'd_2''}tv;e_2)
\mathcal{E}_{f_2,f_3}(0,\overline{e_1'e_2}tv;d_2'') Q''(v^\flat).
\end{align*}

Let $D:=2^{5}3^4$. We shall  split the summation over $v$ into residue classes modulo $Db_1d_2^*e$. We note that
$\rad(r_2)\mid d_2^*$ by \eqref{stardef}. Hence we get
\begin{align*}
  S(u,t)=~& \sum\limits_{k \bmod{Db_1d_2^*e}} \e\left(\frac{a_1
      \overline{e_0} t^3k^3}{b_1d_2^*}\right) 
  \left(\frac{k}{e}\right) 
  \mathcal{E}_{f_2,f_3}(\overline{e_1}u,\overline{e_1'd_2''}tk;e_2)
  \mathcal{E}_{f_2,f_3}(0,\overline{e_1'e_2}tk;d_2'') Q''(k^\flat)\\  
&\times  \sum\limits_{v\equiv k
    \bmod{Db_1d_2^*e}} \Psi\left(\frac{3^{f_3}d_1'd_2'tY}{q}\cdot v \right),
  \end{align*}
where $k^\flat$ denotes the residue class of $k$ modulo $D$. We detect this residue class using additive characters, getting
\begin{equation} \label{furthersplit}
\begin{split}
  S(u,t) =~ & \frac{1}{D} \sum\limits_{x=1}^D \sum\limits_{y=1}^D
  \e\left(\frac{xy}{D}\right)Q''(y) \sum\limits_{k
    \bmod{Db_1d_2^*e}} \e\left(\frac{a_1 \overline{e_0}
      t^3k^3}{b_1d_2^*}\right)  \left(\frac{k}{e}\right) 
 \mathcal{E}_{f_2,f_3}(\overline{e_1}u,\overline{e_1'd_2''}tk;e_2)
 \\ 
  & \times 
 \mathcal{E}_{f_2,f_3}(0,\overline{e_1'e_2}tk;d_2'') 
  \e\left(-\frac{xk}{D}\right) \sum\limits_{v\equiv k
    \bmod{Db_1d_2^*e}} \Psi\left(\frac{3^{f_3}
d_1'd_2'tY}{q}\cdot v \right).
\end{split}
\end{equation}
Applying Poisson summation to the inner-most sum over $v$ after a
linear change of variables, we see that it is 
\begin{equation} \label{afterpoi}
\begin{split}
&= 
\frac{q}{3^{f_3}Db_1d_2^*ed_1'd_2'tY} 
\sum\limits_{h\in \mathbb{Z}}
\e\left(\frac{kh}{Db_1d_2^*e} \right)
\hat\Psi\left(\frac{q}{3^{f_3}Db_1d_2^*ed_1'd_2'tY} \cdot h\right)\\ 
&= \frac{r_1}{3^{f_3}Db_1ed_1'{d_2''}^2tY} \sum\limits_{h\in
  \mathbb{Z}} \e\left(\frac{kh}{Db_1d_2^*e} \right)
\hat\Psi\left(\frac{r_1}{3^{f_3}Db_1ed_1'{d_2''}^2tY}\cdot h\right),
\end{split}
\end{equation}
where in the last equation, we have used 
\eqref{sdef}, \eqref{stardef} and $q=r_1r_2$.

Combining  \eqref{furthersplit} and \eqref{afterpoi} we get
\begin{equation} \label{Srew}
 S(u,t)
= 
\frac{r_1}{3^{f_3}
D^2b_1ed_1'{d_2''}^2tY}  \sum\limits_{x=1}^D \sum\limits_{y=1}^D 
\e\left(\frac{xy}{D}\right)Q''(y)\sum\limits_{h\in \mathbb{Z}} 
\hat\Psi\left(\frac{r_1}{3^{f_3}Db_1ed_1'{d_2''}^2tY}\cdot h\right) F_x(h),
\end{equation}
where $F_x(h)$ is the complete exponential sum
\begin{align*}
F_x(h) := ~& \sum\limits_{k \bmod{Db_1d_2^*e}} \e\left(\frac{kh}{Db_1d_2^*e} \right) 
\e\left(\frac{a_1 \overline{e_0} t^3k^3}{b_1d_2^*}\right) 
\left(\frac{k}{e}\right) \e\left(-\frac{xk}{D}\right)  \\
&\times
\mathcal{E}_{f_2,f_3}(\overline{e_1}u,\overline{e_1'd_2''}tk;e_2) 
\mathcal{E}_{f_2,f_3}(0,\overline{e_1'e_2}tk;d_2''). 
\end{align*}
We now split the exponential sum $F_x(h)$ into parts and estimate
them. We observe that 
$$
Db_1d_2^*e=Db_1{d_2''}^3e_2e
$$ 
by \eqref{stardef} and a little thought reveals that 
$Db_1{d_2''}^3$, $e_2$, and $e$ are pairwise coprime. Therefore we
may write $k$ mod $Db_1d_2^*e$ in the form $k=\alpha \cdot e_2e +\beta
\cdot Db_1{d_2''}^3e_2+\gamma\cdot Db_1{d_2''}^3e$, where $\alpha$
runs over residue classes mod $Db_1{d_2''}^3$, $\beta$ runs over
residue classes mod $e$, and $\gamma$ runs over residue
classes mod $e_2$. It follows that 
\begin{equation} \label{factorisierung}
F_x(h)=G_1(h)G_2(h)G_3(h),
\end{equation}
where
\begin{align*}
G_1(h)
&:= \sum\limits_{\alpha
  \bmod{Db_1{d_2''}^3}}\e\left(\frac{Da_1\overline{e_0}t^3e^3e_2^2\alpha^3+h\alpha}{Db_1{d_2''}^3}\right) \e\left(-\frac{xe_2e\alpha }{D}\right)  \mathcal{E}_{f_2,f_3}(0,\overline{e_1'}te\alpha;d_2''),\\ 
G_2(h)&:=\left(\frac{Db_1\rad(r_2)}{e}\right) \sum\limits_{\beta
  \bmod{e}}  \e\left(\frac{h\beta}{e} \right)\cdot
\left(\frac{\beta}{e}\right),\\ 
G_3(h) &:= 
\sum\limits_{\gamma \bmod{e_2}}
\e\left(\frac{h\gamma}{e_2}\right)   \e\left(\frac{D^3a_1
    \overline{e_0}t^3e^3b_1^2{d_2''}^6\gamma^3}{e_2}\right) 
\mathcal{E}_{f_2,f_3}(\overline{e_1}u,D\overline{e_1'}tb_1{d_2''}^2e\gamma;e_2). 
\end{align*}
We note that only $G_1(h)$ depends on $x$. Our next task is to provide
satisfactory upper bounds for the modulus of these exponential sums.

\bigskip

Beginning with $G_2(h)$ we see that this is a Gauss sum with a
quadratic character. Since the modulus $e$ is
squarefree this character is primitive and it follows that
\begin{equation} \label{F2est}
|G_2(h)|\le e^{1/2}.
\end{equation}
Moreover, if $h=0$ then we have
\begin{equation} \label{F2est0}
G_2(0) = \begin{cases} 
0 & \mbox{ if } e>1,\\  
1 & \mbox{ if } e=1. \end{cases}
\end{equation}
Turning to $G_3(h)$, for which we use 
$(\overline{e_1}u,e_2)=1$ and Lemma \ref{generalexpestimate}, we
deduce that
\begin{equation} \label{F3est1}
|G_3(h)| \le \sum\limits_{\gamma \bmod{e_2}} |\mathcal{E}_{f_2,f_3}(\overline{e_1}u,D\overline{e_1'}tb_1{d_2''}^2e\gamma;e_2)| \ll e_2^{3/2}2^{\omega(e_2)}.
\end{equation}

The treatment of $G_1(h)$ is more taxing and 
we will need to introduce some more notation.
For a natural number $n$ 
let $\rho(n)$ be defined multiplicatively by 
\begin{equation} \label{C}
\rho(p^\alpha):=p^{f(\alpha)},
\end{equation}
where
\begin{equation} \label{Cfunc}
f(\alpha):=
\begin{cases}
1/2 & \mbox{if } \alpha=1, \\
1  & \mbox{if } \alpha=2 \mbox{ or } \alpha=3,\\
5/4  & \mbox{if } \alpha=4,\\
\alpha/4 & \mbox{if } \alpha\ge 5. 
\end{cases}
\end{equation}
We note that $\rho(n)\le \sqrt{n}$ for every $n\in\mathbb{N}$.
Furthermore, for any $\alpha\leq \beta$ one easily checks that 
$$
\alpha-f(\alpha)\leq \beta-f(\beta), \quad 
f(\alpha)\leq f(\beta), \quad
f(\alpha+\beta)\leq f(\alpha)+f(\beta).
$$
Hence it follows that $m/\rho(m)\le
n/\rho(n)$ and $\rho(m)\le \rho(n)$ if $m\mid n$, and furthermore 
$$
\rho(n_1n_2)\le \rho(n_1)\rho(n_2),
$$
for all $n_1,n_2$. 
The estimation of $G_1(h)$ is the object of the
following result.

\begin{lemma}\label{lem:star}
We have 
$$
G_1(h)\ll 
\begin{cases}
{d_2''}^{7/2}b_1^{1/2}(h,b_1)^{1/2}2^{\omega(b_1)}(4+\ve)^{\omega(d_2'')} &
\mbox{if $h\neq 0$,}\\ 
a^{1/2}b_1 2^{\omega(b_1)}{d_2''}^3/\rho(bu^2) & 
\mbox{if $h= 0$.} 
\end{cases}
$$
\end{lemma}

\begin{proof}
To deal with $G_1(h)$ we write $
\alpha=\alpha_1+Dd_2''\alpha_2,
$
where $\alpha_1$ runs over the residue classes mod $Dd_2''$, and
$\alpha_2$ runs over the residue classes mod $b_1{d_2''}^2$. This gives
\begin{equation} \label{F1splitting}
\begin{split}
G_1(h)=~& \sum\limits_{\alpha_1 \bmod{Dd_2''}}  
\e\left(\frac{D
    a_1\overline{e_0}t^3e^3e_2^2\alpha_1^3+h\alpha_1}{Db_1{d_2''}^3}\right)
\e\left(-\frac{xe_2e\alpha_1 }{D}\right)\\
&\times \mathcal{E}_{f_2,f_3} (0,\overline{e_1'}te\alpha_1;d_2'')
 G_1(h,\alpha_1), 
\end{split}
\end{equation}
where 
$$
G_1(h,\alpha_1):=\sum\limits_{\alpha_2 \bmod{ b_1{d_2''}^2}}  \e\left(\frac{P_{h,\alpha_1}(\alpha_2)}{b_1{d_2''}^2}\right)
$$
and
$$
P_{h,\alpha_1}(X):=D a_1\overline{e_0}t^3e^3e_2^2\cdot \frac{(\alpha_1+Dd_2''X)^3-\alpha_1^3}{Dd_2''}+hX.
$$
Using $(\overline{e_1'}t e,d_2'')=1$, we deduce from Lemma \ref{generalexpestimate} that
\begin{equation} \label{F1part1est}
|\mathcal{E}(0,\overline{e_1'}te\alpha_1;d_2'')| \ll
 (\alpha_1,d_2'')^{1/2} {d_2''}^{1/2}2^{\omega(d_2'')}.  
\end{equation}

Let us begin by considering the case $h\neq 0$, to deal with which 
we call upon the general upper bound 
for complete exponential sums presented in 
Lemma \ref{polyexpsumest}. We wish to apply this with 
$Q=b_1{d_2''}^2$ and
$g(X)=P_{h,\alpha_1}(X)$.
We note that $g'(X)=3D a_1\overline{e_0} t^3 e^3e_2^2 (\alpha_1 +Dd_2''
X)^2 +h$.  Thus we have 
$$
n=2, \quad \eta=1, \quad 
\Delta= 4hA, \quad 
A=3D^3a_1\overline{e_0}t^3e^3e_2^2d_2''^2,
$$
and 
$$
(\Delta,Q)=
{d_2''}^2(12D^3a_1\overline{e_0}t^3e^3e_2^2h,b_1).
$$
Recalling that $D=2^{5}3^4$ and that $t$ can be assumed to be square-free with 
$t/(t,6)$ coprime to $b_1$, so $(t^3,b_1)\mid 2^33^3$, we readily
conclude that $(\Delta,Q)\ll d_2''^2(h,b_1)$. 
Hence 
$$
G_1(h,\alpha_1)\ll
b_1^{1/2}(h,b_1)^{1/2}{d_2''}^{2}
2^{\omega(b_1d_2''^2)}.
$$
Combining this with \eqref{wellknown}, \eqref{F1splitting}
and \eqref{F1part1est} we obtain
the bound for $G_1(h)$ in the statement of the lemma. 

Suppose now that $h=0$ and let
$\delta:=D a_1\overline{e_0}t^3e^3e_2^2$. Then we have
$$
G_1(0,\alpha_1)=\sum\limits_{\alpha_2 \bmod{ b_1{d_2''}^2}}  \e\left(\frac{\delta D^2\alpha_2^3}{b_1}+\frac{3\delta D\alpha_1 \alpha_2^2}{b_1d_2''}+\frac{3\delta \alpha_1^2\alpha_2}{b_1{d_2''}^2}\right).
$$
We may write $\alpha_2=x_1+x_2\cdot b_1d_2''$, with $x_1\in \mathbb{Z}/b_1d_2''\mathbb{Z}$ and $x_2\in \mathbb{Z}/d_2''\mathbb{Z}$.  It follows that
$$
G_1(0,\alpha_1)=\sum\limits_{x_1 \bmod{ b_1{d_2''}}}  \e\left(\frac{\delta D^2 x_1^3}{b_1}+\frac{3\delta D\alpha_1 x_1^2}{b_1d_2''}+\frac{3\delta\alpha_1^2x_1}{b_1{d_2''}^2}\right) 
\sum\limits_{x_2 \bmod{d_2''}} \e\left(\frac{3\delta\alpha_1^2x_2}{d_2''}\right).
$$
Using the facts that $(3\delta,d_2'')=1$ and $d_2''$ is square-free,
we deduce that the innermost sum on the right-hand side of the above equation is $0$ unless $\alpha_1\equiv 0$ mod $d_2''$.
Hence we have
\begin{equation} \label{not0case}
G_1(0,\alpha_1)=0 \quad \mbox{if } \alpha_1\not\equiv 0 \mbox{ mod } d_2''
\end{equation}
and 
\begin{align*}
G_1(0,jd_2'') & =  {d_2''}^2\sum\limits_{x \bmod{b_1}}
\e\left(\frac{\delta D^2 x^3+3\delta D j x^2+3\delta
    j^2x}{b_1}\right)\\ &= 
{d_2''}^2\sum\limits_{x \bmod{b_1}} \e\left(\frac{\delta
    \left((Dx+j)^3 -j^3\right)}{Db_1}\right)\\ &= 
\frac{{d_2''}^2}{D}  \sum\limits_{x \bmod{Db_1}}
\e\left(\frac{\delta\left((Dx+j)^3-j^3\right)}{Db_1}\right).
\end{align*}
Let $E:=(\delta,Db_1)$, $\delta':=\delta/E$ and $b_1':=Db_1/E$. Then
it follows that 
\begin{equation} \label{0caseprime}
G_1(0,jd_2'')=\frac{E{d_2''}^2}{D}  \sum\limits_{x=1}^{b_1'} \e\left(\frac{\delta' \left((Dx+j)^3-j^3\right)}{b_1'}\right).
\end{equation}
As above we have $E=O(1)$.

In the following, we want to bound the sum on the right-hand side of
\eqref{0caseprime}. To this end we consider in general terms 
exponential sums of the form
$$
G(c,r;Q)=\sum\limits_{x=1}^{Q} \e\left(\frac{c\left((Drx+j)^3-j^3\right)/r}{Q}\right).
$$
A standard calculation shows that these exponential sums are
multiplicative in the sense that 
\begin{equation} \label{Gfacto}
G(c,r;Q)=G(c,rQ_2;Q_1) G(c,rQ_1;Q_2) 
\end{equation}
if $Q=Q_1Q_2$ with  $(Q_1,Q_2)=1$. 
Moreover, we note that if $(r,Q)=1$, then
\begin{equation} \label{gausssumme}
G(c,r;Q)=\sum\limits_{x=1}^{Q}
\e\left(\frac{c\overline{r}\left((Drx+j)^3-j^3\right)}{Q}\right). 
\end{equation}
If in addition $(D,Q)=1$, then $Drx+j$ runs through all residue
classes modulo $Q$ as $x$ runs through all of them, and therefore it
follows that
\begin{equation} \label{gausssumme1}
G(c,r;Q)=\e\left(-\frac{c\overline{r}j^3}{Q}\right)
\sum\limits_{x=1}^{Q}\e\left(\frac{c\overline{r}x^3}{Q}\right).
\end{equation}
Hence $G(c,r;Q)$ is essentially a cubic Gauss sum in this case.

We claim the bound 
\begin{equation} \label{cubicsumest}
|G(c,r;Q)| \le (Q,9D^6) \cdot \frac{Q}{\rho(Q)} \cdot 2^{\omega(Q)},
\end{equation}
if $(cr,Q)=1$, where $\rho(Q)$ is given multiplicatively by \eqref{C} and
\eqref{Cfunc}. By \eqref{Gfacto},
to prove this it suffices to
establish this bound for $Q$ a prime power. Suppose that 
$Q=p^{\alpha}$ with $p$ prime and $\al\geq 1$. Clearly 
\eqref{cubicsumest} follows from the Weil bound if
$\alpha=1$ and $p>3$. It follows from
Lemma~\ref{polyexpsumest} with $n=2, \eta=2$ and $\Delta=(3cD^3r^2)^2$
if $\alpha\ge 5$.   Next we note that 
$$
\sum\limits_{x=1}^{p^2} \e\left(\frac{c\overline{r}
    x^3}{p^2}\right)=p+E(c\overline{r},0;p^2), \quad
\quad \sum\limits_{x=1}^{p^3} \e\left(\frac{c\overline{r}
    x^3}{p^3}\right)=p^2+E(c\overline{r},0;p^3), 
$$
in the notation of \eqref{Eq0}. Moreover, by Lemma \ref{mncaseslemma}(ii), we have $E(c\overline{r},0;p^2)=0=E(c\overline{r},0;p^3)$ if $p>3$. Thus 
\eqref{cubicsumest} follows from \eqref{gausssumme1} if $\alpha=2,3$ and $p>3$.
When $\alpha=4$ and $p>3$ we deduce from 
Lemma \ref{mncaseslemma}(ii) that
$$
\sum\limits_{x=1}^{p^4} \e\left(\frac{c\overline{r}
    x^3}{p^4}\right)
=
\sum_{\substack{x=1 \\ p\mid x}}^{p^4} 
\e\left(\frac{c\overline{r}
    x^3}{p^4}\right)
+E(c\overline{r},0,p^4)=
p^2
\sum\limits_{x=1}^{p} \e\left(\frac{c\overline{r}
    x^3}{p}\right),
$$
which has modulus at most $2p^{5/2}$ by the Weil estimate if
$p>3$. Since \eqref{cubicsumest} follows trivially if $1\leq \alpha
\leq 4$ and $p=2,3$, this therefore completes its proof. 

Employing \eqref{cubicsumest}, the definition of $b_1$ in
\eqref{a1b1}  and the facts about the function $\rho(Q)$ stated above
before the lemma, we deduce that
$$
G(\delta',1;b_1')
\ll \frac{b_1'2^{\omega(b_1')}}{\rho(b_1')}
\ll
\frac{b_12^{\omega(b_1)}}{\rho(b_1)}
\ll
\frac{b_12^{\omega(b_1)}\rho(Aa)}{\rho(bu^2)} \ll 
\frac{b_1 2^{\omega(b_1)}a^{1/2}}{\rho(bu^2)}.
$$
Therefore 
$$
G_1(0,\alpha)=\frac{Ed_2''^2}{D}\cdot G(\delta',1;b_1')
\ll
\frac{a^{1/2} b_12^{\omega(b_1)}d_2''^2}{\rho(bu^2)}.
$$
Combining this with \eqref{F1splitting} and 
\eqref{not0case},  and using the trivial
estimate $\mathcal{E}(0,0;d_2'')\le d_2''$, we therefore get 
the second part of the lemma.
\end{proof}

Recall that $d_2''e_2=\rad(r_2)$ and $d_1'\gg d_1$. In particular 
$2^{\omega(e_2)}(4+\ve)^{\omega(d_2'')}\leq
(4+\ve)^{\omega(r_2)}$. 
 Then from 
\eqref{Srew}--\eqref{F3est1} and  
Lemma \ref{lem:star} we obtain 
\begin{equation} \label{Srew2}
S(u,t)
\ll A_0(u,t)+A_1(u,t),
\end{equation}
in \eqref{sumovervt}, 
where 
\begin{align}\label{A0def}
A_0(u,t)
&:=\frac{a^{1/2}r_1\cdot \rad(r_2)e_2^{1/2}2^{\omega(b_1)+\omega(e_2)}}{d_1\rho(bu^2)tY} \cdot |\hat\Psi(0)|,\\
\label{A1def}
A_1(u,t)&:=
\frac{r_1\cdot
  \rad(r_2)^{3/2}2^{\omega(b_1)}(4+\ve)^{\omega(r_2)} }{d_1e^{1/2}b_1^{1/2}tY}
 \sum\limits_{h\not=0} \left| 
\hat\Psi\left(\frac{r_1}{3^{f_2} 
Db_1ed_1'{d_2''}^2tY}\cdot h\right) \right|
 (h,b_1)^{1/2}.
\end{align}

\bigskip

Let 
\begin{equation} \label{alphabeta}
\alpha:=\frac{r_1}{3^{f_2} Db_1ed_1'{d_2''}^2tY}, \quad \beta:=\frac{Ca_1q^2}{b_1d_1^2d_2^2Y^3}, \quad 
\gamma:=\frac{y_0}{Y}.
\end{equation}
We now turn to the estimation of 
$|\hat\Psi(z)|$. By \eqref{Phidef2} we have
\begin{equation} \label{PhidefFourier}
\hat\Psi(z)=F\left(z+\gamma,\beta\right),
\end{equation}
where $F$ is given by \eqref{Fdef1}.
We note that $\alpha>0$ and $\beta>0$ 
since we assumed at the outset that $a,b>0$. Hence we are free to
apply the estimates \S \ref{s:airy}.  
Our goal is to bound $|\hat\Psi(0)|$ in \eqref{A0def} and the series
$$
\sum\limits_{h\not=0} |\hat\Psi(\alpha h)| (h,b_1)^{1/2} 
=
\sum\limits_{h\not=0} |F(\gamma+\alpha h,\beta)|  (h,b_1)^{1/2}
$$
in \eqref{A1def}.  We shall split the latter into three ranges.

By \eqref{Fbound2} and \eqref{wellknown}, we have
\begin{equation} \label{triviality} 
\sum\limits_{1\le |h|\le 2|\gamma|/\alpha} |F(\gamma+\alpha h,\beta)|  
(h,b_1)^{1/2} \ll \sum\limits_{1\le |h|\le 2|\gamma|/\alpha}  
(h,b_1)^{1/2} \ll \frac{|\gamma|}{\alpha} (1+\ve)^{\omega(b_1)}.
\end{equation}
Next let $\kappa$ be any positive real number and let $P:=2abqXY$. 
Using $|\gamma+\alpha h|
\gg \alpha |h|$ for $|h|> 2|\gamma|/\alpha$ and 
\eqref{Fbound3}, we combine \eqref{wellknown} with 
partial summation to obtain 
\begin{equation} \label{nontriviality}
\begin{split}
\sum\limits_{2|\gamma|/\alpha< |h|\le P^{\kappa}} |F(\gamma+\alpha h,\beta)|  
(h,b_1)^{1/2}  \ll ~&    
\sum\limits_{1\le |h|\le P^{\kappa}} \frac{(h,b_1)^{1/2}}{\alpha |h|}\\
&+ \sum\limits_{h\ge 1}  
\frac{\exp\left(-\pi \alpha h/(3\beta)\right)}{(\alpha \beta h)^{1/4}} (h,b_1)^{1/2}
\\ 
 \ll~&  
\left(\frac{\log
    P}{\alpha}+\frac{\beta^{1/2}}{\alpha}\right)(1+\ve)^{\omega(b_1)}. 
\end{split}
\end{equation}
Let us give a bit more explanation about how the second term in the
last line arises.
The term $\exp\left(-\pi \alpha h/(3\beta)\right)$ 
is negligible if $h$ is much larger than $\beta/\alpha$. Moreover
\eqref{wellknown} ensures that the term $(h,b_1)^{1/2}$ has average 
order $O((1+\ve)^{\omega(b_1)})$. Using a dyadic summation one readily
concludes that the sum in 
question is  bounded by the sum over $1\le h \ll \beta/\alpha$ of the
term $(1+\ve)^{\omega(b_1)}/(\alpha \beta h)^{1/4}$,  
which is $\ll (1+\ve)^{\omega(b_1)} \beta^{1/2}/\alpha$. 

Finally if $\kappa>0$ is large enough, then we infer from
\eqref{Fbound6} that 
\begin{equation} \label{largeh}
\sum\limits_{|h|>P^{\kappa}} |F(\gamma+\alpha h,\beta)|  (h,b_1)^{1/2} \ll P^{-1000}.
\end{equation}

Combining \eqref{triviality}, \eqref{nontriviality} and
\eqref{largeh},  we get
\begin{equation} \label{mirfallenkeinenamenmehrein}
\sum\limits_{h\not=0} |\hat\Psi(\alpha h)| (h,b_1)^{1/2}  \ll
\frac{\log P+|\gamma|+\beta^{1/2}}{\alpha} \cdot 
(1+\ve)^{\omega(b_1)}. 
\end{equation}
From \eqref{primeprime}, \eqref{alphabeta},
$d_1'\ll d_1$ and $q=r_1r_2$,
we therefore deduce that 
\begin{equation} \label{series3} 
\begin{split}
\sum\limits_{h\not=0} &|\hat\Psi(\alpha h)|  (h,b_1)^{1/2} \\
&\ll  \left(\frac{b_1ed_1d_2^2\cdot \rad(r_2)^2t\cdot Y_0\log
    P}{r_1r_2^2}+ \frac{(a_1b_1)^{1/2}ed_2\cdot
    \rad(r_2)^2t}{r_2Y^{1/2}}\right)(1+\ve)^{\omega(b_1)},
\end{split}
\end{equation} 
where 
\begin{equation}\label{Y0}
Y_0:=|y_0|+Y.
\end{equation}

We close this section by recording an upper bound for $\hat\Psi(0)$.
Using \eqref{Fbound2} and \eqref{Fbound3}, we have the bound
$$
\hat\Psi(0)\ll
\min\left\{1,\frac{1}{|\gamma|^{1/4}\beta^{1/4}}+\frac{1}{|\gamma|}\right\}
\ll \frac{1}{(|\gamma|+1)^{1/4}\beta^{1/4}}+\frac{1}{|\gamma|}.  
$$
Combining this with 
\eqref{a1b1} 
and \eqref{alphabeta}, we therefore derive the estimate
\begin{equation} \label{Phi0bound2} 
\hat\Psi(0)
\ll \frac{b^{1/4}|u|^{1/2}d_1^{1/2}d_2^{1/2}Y}{a^{1/4}q^{1/2}Y_0^{1/4}}+\frac{Y}{Y_0}, 
\end{equation}
where $Y_0$ is given in \eqref{Y0}.

\section{Conclusion of the  proof of Theorem \ref{cor:0}} \label{s:proof-upper}

We now derive our final  estimate for $\mathcal{S}_{1}$, before
combining it with the contents of \S \ref{s:thec} to complete the asymptotic
formula for $\mathcal{S}$.  This will in turn lead us to the statement
of Theorem \ref{cor:0}. 

Combining \eqref{remcop}, \eqref{Srew2}, \eqref{series3},
\eqref{Phi0bound2}, $q=r_1r_2$, $d_1\gg d_1'\geq d_1$ 
and $e_i=r_i/d_i$, we get
$$
S(u)\ll A_0(u)+A_1(u),
$$
where
\begin{align} \label{A0withoutt}
A_0(u) :=~& 
 \left(
 \frac{a^{1/4}b^{1/4}r_1^{1/2}\rad(r_2)
|u|^{1/2}}{d_1^{1/2}\rho(bu^2)Y_0^{1/4}}+\frac{a^{1/2}r_1\cdot  
 \rad(r_2)r_2^{1/2}}{d_1d_2^{1/2}\rho(bu^2)Y_0}\right)  
2^{\omega(b_1)+\omega(e_2)}
\sigma_{-1}(e_1),\\
\label{A1withoutt}
A_1(u)
:=~& \left(\frac{b_1^{1/2}e^{1/2}d_2^2\cdot
    \rad(r_2)^{7/2} \cdot Y_0 \log P}{r_2^2 Y}
+ \frac{a_1^{1/2}e^{1/2}d_2 r_1\cdot
    \rad(r_2)^{7/2}}{d_1r_2Y^{3/2}}\right)\\
&\times
(2+\ve)^{\omega(b_1)}2^{\omega(e_1)}(4+\ve)^{\omega(r_2)}.
\nonumber 
 \end{align}

To derive an estimate for $\mathcal{S}_1$, we have to sum the above
expressions over $u$ as dictated by \eqref{split2}.
Let 
$$
I_q:=
\sum_{\substack{u\in \ZZ_{\neq 0}\\
(u,e_1e_2)=1
}} \Gamma\left(
\frac{2^{f_2}d_1d_2uX}{q}
\right)|S(u)|.
$$
For any $M>0$ let us
consider the pair of sums
$$
\Sigma_i(M):=\sum_{\substack{0<|u|\leq M\\
(u,e_1e_2)=1
}} A_i(u),
$$
for $i=0,1$, with $A_0(u), A_1(u)$ given by  
\eqref{A0withoutt} and \eqref{A1withoutt}, respectively. 
Estimating these is 
the object of the following pair of results.

\begin{lemma}\label{lem:S1}
We have 
\begin{align*}
\Sigma_0(M)\ll 
\sigma_0(M):=~&(\log (2+ M))^2 
\left(
\frac{a^{1/2}\rad(r_2)r_1r_2^{1/2}}{\rho(b)d_1d_2^{1/2}Y_0}  
 + \frac{a^{1/4}b^{1/4}\rad(r_2)r_1^{1/2}M^{1/2}}{\rho(b)d_1^{1/2}Y_0^{1/4}}
\right) \\
&\times (\log q)^\ve (2+\ve)^{\omega(b)}2^{\omega(e_2)}.
\end{align*}
\end{lemma}

\begin{proof}
Recall that $d_2''=\rad(r_2)/e_2$ and 
$e_2=r_2/d_2$.  Applying the bound $\sigma_{-1}(e_1)\ll (\log q)^\ve$
we insert our expression for $A_0(u)$ to see that
$$
\Sigma_0(M)\leq (\log q)^\ve 2^{\omega(e_2)}
\left(
\frac{a^{1/2}\rad(r_2)r_1r_2^{1/2}}{d_1d_2^{1/2}Y_0}+ 
\frac{a^{1/4}
b^{1/4}M^{1/2}\rad(r_2)
r_1^{1/2}}{d_1^{1/2} Y_0^{1/4}} 
\right)\sum_{|u|\leq M} 
\frac{2^{\omega(b_1)}}{\rho(bu^2)},
$$
where $C$ is given multiplicatively by \eqref{C} and 
\eqref{Cfunc}.  
We claim that 
$$
\sum_{|u|\leq M} 
\frac{2^{\omega(b_1)}}{\rho(b u^2)}\ll \frac{(2+\ve)^{\omega(b)}}{\rho(b)}(\log (2+ M))^2,
$$
which once inserted into the above leads to the statement of the
lemma. 

To establish the claim we see that 
$$
\sum_{|u|\leq M} 
\frac{2^{\omega(b_1)}}{\rho(b u^2)} \ll
\frac{2^{\omega(b)}}{\rho(b)} \sum_{u=1}^M \frac{\rho(b)2^{\omega(u)}}{\rho(bu^2)}.
$$
Clearly
$$
\sum_{u=1}^M \frac{\rho(b)2^{\omega(u)}}{\rho(bu^2)}\leq
\prod_{\substack{p\leq M\\ p^\beta\| b}}
\left(
1+\sum_{k\geq 1} \frac{2p^{f(\beta)}}{p^{f(\beta+2k)}}
\right)
\prod_{\substack{p\leq M\\ p\nmid b}}
\left(
1+\sum_{k\geq 1} \frac{2}{p^{f(2k)}}
\right)=\Pi_1\Pi_2,
$$
say. Merten's theorem yields $\Pi_2\ll (\log (2+ M))^2$. Moreover 
we have
$$
\Pi_1\leq \prod_{p^\beta\| b}
\left(
1+\sum_{k\geq 1} \frac{2p^{f(\beta)}}{p^{f(\beta+2k)}}
\right)\ll  (1+\ve)^{\omega(b)}.
$$
Putting these together therefore establishes the claim.
\end{proof}

\begin{lemma}\label{lem:S2}
Let $P=2abqXY$.  We have 
\begin{align*}
\Sigma_1(M)\ll \sigma_1(M):=~& \log (2+ M)^{1+\ve} \cdot
\frac{e^{1/2}r_1\rad(r_2)^{7/2}}{r_2}\left(
  \frac{b^{1/2}d_2^2Y_0\log P }{qY}\cdot M^2 +
\frac{a^{1/2}d_2}{d_1Y^{3/2}}\cdot M\right)\\
&\times
(2+\ve)^{\omega(b)}2^{\omega(e_1)}(4+\ve)^{\omega(r_2)}.
\end{align*}
\end{lemma}

\begin{proof}
This follows from insertion of our expression for $A_1(u)$ into
$\Sigma_1(M)$, using $a_1\le a$ and $b_1\le bu^2$ and applying \eqref{familiarest}.
\end{proof}

Since $\Gamma$ has exponential decay, it is easily seen by a dyadic summation that Lemmas  \ref{lem:S1} and \ref{lem:S2} imply the bound
$
I_q
\ll \sigma_0(M)+\sigma_1(M),
$
with
$$
M:=\frac{q}{d_1d_2X}.
$$
It now follows that
$$
I_q\ll (\log P)^{2+\ve} \left(I_q^{(0)}
+I_q^{(1)}\right),
$$
where
\begin{align*}
I_q^{(0)}&:=
\frac{\rad(r_2)r_1^{1/2}q^{1/2}}{\rho(b) d_1d_2^{1/2}}
\left(
\frac{a^{1/2}}{Y_0}  
+\frac{a^{1/4}b^{1/4} }{X^{1/2}Y_0^{1/4}}
\right)  (2+\ve)^{\omega(b)}2^{\omega(e_2)},
\\
I_q^{(1)}&:=
\frac{e^{1/2}\rad(r_2)^{7/2}r_1 q}{d_1^2r_2X}\left(
  \frac{b^{1/2}Y_0}{XY} +
\frac{a^{1/2}}{Y^{3/2}}\right)
(2+\ve)^{\omega(b)}2^{\omega(e_1)}(4+\ve)^{\omega(r_2)}.
\end{align*}

Recalling  from 
\eqref{sdef} that $d_2'=r_2/\rad(r_2)$, we deduce from \eqref{split2} 
that our  work so far has shown that
$$
\mathcal{S}_{1} \ll\frac{XY}{q^2}\cdot (\log P)^{2+\ve}
\sum\limits_{\substack{(r_1,r_2)=1\\ r_1r_2=q}} \
\sum\limits_{\substack{d_1\mid r_1\\ \rad(\hat{r}_1)^2|e_1}} \sum\limits_{\substack{d_2 \mid r_2\\ \mu^2(e_2)=1}} \sideset{}{'}  
\sum\limits_{d_1'\mid r_1}
\frac{\sqrt{d_1r_1}r_2}{\rad(r_2)} 
\left(I_q^{(0)}+I_q^{(1)}\right).
$$
Let us write 
$\mathcal{S}_1^{(i)}$ for the overall contribution 
from the term involving $I_q^{(i)}$ for $i=0,1$.

Beginning with the case $i=0$ we see that
\begin{align*}
\mathcal{S}_1^{(0)}\ll~&
\frac{XY (2+\ve)^{\omega(b)}(\log P)^{2+\ve}}{\rho(b) q^{1/2}}
\left(
\frac{a^{1/2}}{Y_0}  
+\frac{a^{1/4}b^{1/4} }{X^{1/2}Y_0^{1/4}}\right) \\
&\times 
\sum\limits_{\substack{(r_1,r_2)=1\\ r_1r_2=q}} \
\sum\limits_{d_1\mid r_1} \sum\limits_{d_2 \mid r_2} \sideset{}{'}  
\sum\limits_{d_1'\mid r_1}
\frac{2^{\omega(e_2)}}{ d_1^{1/2}d_2^{1/2}}.
\end{align*}
Here we recall that the $'$ attached to the inner sum denotes the condition
described in \eqref{d1primecondi}. Thus for fixed $d_1$ the number of
available $d_1'$ is $O(1)$.  
Thus the  sum over $r_1,r_2,d_1,d_2,d_1'$ is clearly at most
$$
\sum_{\substack{(r_1,r_2)=1\\
r_1r_2=q}}
\sigma_{-1/2}(r_1)\sigma_{-1/2}(r_2)
2^{\omega(r_2)}\ll (1+\ve)^{\omega(q)}
\sum_{\substack{r_2\mid q\\
(q/r_2,r_2)=1}}
2^{\omega(r_2)}\leq    (3+\ve)^{\omega(q)},
$$
as can be seen by checking at prime powers. On observing that
$\rho(b)\geq b^{1/4}$ we therefore conclude that
\begin{equation}
  \label{T1term}
\mathcal{S}_1^{(0)}
\ll
\frac{(2+\ve)^{\omega(b)} (3+\ve)^{\omega(q)}(\log
  P)^{2+\ve}}{b^{1/4}  q^{1/2} }
\left(
\frac{a^{1/2}XY}{Y_0}  
+\frac{a^{1/4}b^{1/4}X^{1/2}Y }{Y_0^{1/4}}\right),
\end{equation}
with $Y_0=|y_0|+Y$.

Turning to the case $i=1$ we obtain
\begin{align*}
\mathcal{S}_1^{(1)}
\ll
q^{1/2} (2+\ve)^{\omega(b)} (\log P)^{2+\ve}\left(
  \frac{b^{1/2}Y_0}{X} +
\frac{a^{1/2}}{Y^{1/2}}\right)
S_q
\end{align*}
where
$$
S_q:=
\sum\limits_{\substack{(r_1,r_2)=1\\ r_1r_2=q}} \
\sum\limits_{\substack{d_1\mid r_1
\\
\rad(\hat{r}_1)^2\mid e_1}}
\sum\limits_{\substack{d_2 \mid r_2\\
    \mu^2(e_2)=1}} 
\frac{e^{1/2}\rad(r_2)^{5/2} }{d_1^{3/2}r_2^{3/2}}\cdot 2^{\omega(e_1)
}(4+\ve)^{\omega(r_2)}.
$$
Here we have carried out the summation over $d_1'$ trivially. 
The latter sum is a multiplicative function in $q$ and it will suffice to
analyse it at prime powers $q=p^\nu$ for $\nu\in \NN$. Let us write $r_1=p^{\alpha}$, $r_2=p^{\beta}$ and 
$e\leq e(e_1)$, the square-free kernel of $e_1$. Then the
outer summation is over $\alpha,\beta\geq 0$ for which
$\alpha+\beta=\nu$ and $\min\{\alpha,\beta\}=0$. Put
$S_{p^\nu}=T_1+T_2$ where $T_1$ is the contribution from
$(\alpha,\beta)=(0,\nu)$ and $T_2$ is the contribution from
$(\alpha,\beta)=(\nu,0)$. An easy calculation shows that
$$
T_{1}=
(4+\ve) p^{(5-3\nu)/2}
\sum_{\substack{0\leq \delta_2\leq \nu\\
\nu-\delta_2\leq 1}}1 =(8+\ve) p^{(5-3\nu)/2}
$$
and 
$$
T_2=\sum_{\substack{0\leq \delta_1\leq \nu\\
2\leq \nu-\delta_1 \mbox{ \scriptsize{if} $p>3$}}}
\frac{e(p^{\nu-\delta_1})^{1/2}}{p^{3\delta_1/2}}
\cdot 2^{\omega(p^{\nu-\delta_1})}
\begin{cases}
\le 4,  & \mbox{if $p\leq 3$},\\
=0, & \mbox{if $\nu=1$ and $p>3$,}\\
=2(T_{2,1}+T_{2,2}), & \mbox{if $\nu\geq 2$ and $p>3$,}
\end{cases}
$$
where
$$
T_{2,1}:=\sum_{\substack{0\leq \delta_1\leq \nu-2\\ 2\mid \nu-\delta_1}}
p^{-3\delta_1/2},\quad
T_{2,2}:=\sum_{\substack{0\leq \delta_1\leq \nu-2\\ 2\nmid \nu-\delta_1}}
p^{(1-3\delta_1)/2}.
$$
It is now clear that
$$
T_{2,1}+T_{2,2}\leq p^{\kappa}(1+O(1/p)), 
$$
where $\kappa=0$ if $\nu$ is even, and $\kappa=1/2$ if $\nu$ is odd. Putting this all together we conclude that
$$
S_{p^\nu}\leq \begin{cases} (8+\ve)p, & \mbox{if $\nu=1$,}\\ 2p^{\kappa}\left(1+O(1/\sqrt{p})\right), & \mbox{if $\nu>1$}. \end{cases}
$$
Hence it
follows that
$$
S_q\ll s^{1/2}s_1^{1/2}4^{\omega(s)}(2+\varepsilon)^{\omega(q)}, 
$$
in the notation of \eqref{eq:s-s1}.
Inserting this into our expression for $\mathcal{S}_1^{(1)}$ yields
\begin{equation}
  \label{T2term}
\mathcal{S}_1^{(1)}\ll
s^{1/2}s_1^{1/2} q^{1/2} (2+\ve)^{\omega(b)+\omega(q)}
4^{\omega(s)}
(\log P)^{2+\ve}\left(
  \frac{b^{1/2}Y_0}{X} +
\frac{a^{1/2}}{Y^{1/2}}\right).
\end{equation}
Combining this with \eqref{T1term} we deduce that
\begin{align*}
\mathcal{S}_1\ll~&
\frac{(2+\ve)^{\omega(b)} (3+\ve)^{\omega(q)}(\log
  P)^{2+\ve}}{b^{1/4}  q^{1/2} }
\left(
\frac{a^{1/2}XY}{Y_0}  
+\frac{a^{1/4}b^{1/4}X^{1/2}Y }{Y_0^{1/4}}\right)
\\
&+
s^{1/2}s_1^{1/2} q^{1/2} (2+\ve)^{\omega(b)+\omega(q)}
4^{\omega(s)}
(\log P)^{2+\ve}\left(
  \frac{b^{1/2}Y_0}{X} +
\frac{a^{1/2}}{Y^{1/2}}\right),
\end{align*}
with $Y_0=|y_0|+Y$ and $s,s_1$ given by \eqref{eq:s-s1}.
It will be important in what follows to note that 
\begin{equation} \label{Sestfinal0*}
\mathcal{S}_1=\mathcal{S}_1^{(0)}+\mathcal{S}_1^{(1)},
\end{equation}
where $\mathcal{S}_1^{(0)}$ satisfies \eqref{T1term}
and corresponds to the contribution from $h=0$ and likewise
$\mathcal{S}_1^{(1)}$ satisfies \eqref{T2term}
and corresponds to the contribution from $h\neq 0$.

Drawing together \eqref{SGammadef}, \eqref{S00ev}, \eqref{S10est} and
our estimate for $\mathcal{S}_1$, 
we are now ready to record our final estimate for 
$\mathcal{S}$.

\begin{thm}\label{thm:0}
Let $\ve>0$, let $X,Y\geq 2$ and let $P=2abqXY$.
Let $\rad(q),  s,s_1$ be
given by \eqref{eq:s-s1}
and let $Y_0=|y_0|+Y$.
Then we have 
\begin{align*}
\sideset{}{^*}\sum\limits_{\substack{x,y\\ ax^2+by^3 \equiv 0
    \bmod{q}}} \Gamma\left(\frac{x-x_0}{X} \right)
\Gamma\left(\frac{y-y_0}{Y}\right)
=
\frac{\phi(q)XY}{q^2}+O\left(
E_1+E_2+E_3\right),
\end{align*}
where  
\begin{align*}
E_1&=
\frac{(2+\ve)^{\omega(q)}\rad(q)^{1/2}X}{q},\\
E_2&=
\frac{(2+\ve)^{\omega(b)} (3+\ve)^{\omega(q)}}{b^{1/4}  q^{1/2} }
\left(
\frac{a^{1/2}XY}{Y_0}  
+\frac{a^{1/4}b^{1/4}X^{1/2}Y }{Y_0^{1/4}}\right) (\log
  P)^{2+\ve},
\\
E_3&=
s^{1/2}s_1^{1/2} q^{1/2} (2+\ve)^{\omega(b)+\omega(q)}
4^{\omega(s)}\left(
  \frac{b^{1/2}Y_0}{X} +
\frac{a^{1/2}}{Y^{1/2}}\right) (\log
  P)^{2+\ve}.
\end{align*}
\end{thm}

This result is uniform in everything involved. 
In the application we have in mind it will prove crucial to have the
precise dependence on all of the parameters worked out. Often however
it suffices to bound the arithmetic functions involving $a,b, q$
crudely. 
Assuming without loss of generality that
$a,b\leq q$, Theorem~\ref{cor:0} is now a simple
consequence of Theorem~\ref{thm:0} with $x_0=y_0=0$ and $X=\XX,
Y=\YY$.

\section{Asymptotic formula for $M(B,\XX,\YY;a,b;q)$}

Let $M(B,\XX,\YY;a,b;q)$ be defined in \eqref{eq:Mq}, for 
$B\geq 2$ and $\boldsymbol{X},\boldsymbol{Y}\geq 1$
and non-zero integers $a,b,q$ such that $q>0$.
Again we may assume without loss of generality that $a,b>0$ and
$(ab,q)=1$. 
It is now time to turn to an asymptotic formula for
$M(B,\XX,\YY;a,b;q)$, building on the proof of Theorem \ref{thm:0}.
We will proceed under
the hypothesis that 
\begin{equation} \label{sizecond}
\log (2abq\boldsymbol{X}\boldsymbol{Y})  \ll  \log B
\end{equation}
and
\begin{equation} \label{voraus*}
a\XX^2\ge qB, \quad 
b\YY^3\ge qB. 
\end{equation}
It will be convenient to set
\begin{equation} \label{sigtaudef}
\sigma:=\left(\frac{qB}{a}\right)^{1/2}, \quad 
\tau:=\left(\frac{qB}{b}\right)^{1/3},
\end{equation}
and to make a change of variables
\begin{equation} \label{st}
s:=\frac{x}{\sigma}\quad \mbox{and} \quad t:=\frac{y}{\tau}.
\end{equation}
Then the constraints $0<x\le \boldsymbol{X}$, $|y|\le \boldsymbol{Y}$
and $|ax^2+by^3|\le qB$ on $x$ and $y$ translate into the conditions 
$$
0<s\le \frac{\boldsymbol{X}}{\sigma},\quad |t| \le
\frac{\boldsymbol{Y}}{\tau},\quad |s^2+t^3|\le 1.
$$
Likewise the assumption \eqref{voraus*} becomes
\begin{equation} \label{voraus}
\frac{\boldsymbol{X}}{\sigma}\ge 1, \quad \frac{\boldsymbol{Y}}{\tau}\ge 1.
\end{equation}

For any region $\Omega\subset \mathbb{R}^2$ set
\begin{equation} \label{Somega}
\mathcal{S}(\Omega;a,b;q):=\# \left\{(x,y)\in \mathbb{Z}^2  \ : \  (xy,q)=1,\ ax^2+by^3\equiv 0 \bmod{ q},\ 
 (s,t)\in \Omega\right\}.
 \end{equation}
 Then 
 \begin{equation} \label{Srewrite}
 M(B,\boldsymbol{X},\boldsymbol{Y};a,b;q) = \mathcal{S}(\mathcal{R};a,b;q),
 \end{equation}
with
\begin{equation} \label{region}
\mathcal{R}=\left\{ (s,t)\in \mathbb{R}^2\ :\  0<s\le \frac{\boldsymbol{X}}{\sigma},\quad |t| \le \frac{\boldsymbol{Y}}{\tau},\quad |s^2+t^3|\le 1\right\}.
\end{equation}

We split the region $\mathcal{R}$ into two subregions
$$
\mathcal{R}_1:=\left\{ (s,t)\in \mathcal{R} \ :\  0<s\le L^{1/2}\right\} 
$$
and 
$$
\mathcal{R}_2:=\left\{ (s,t)\in \mathcal{R} \ :\  L^{1/2}<s \le
  \frac{\boldsymbol{X}}{\sigma}\right\}, 
$$
where throughout this paper, we set
$$
L:=\log\log \left(2000 B\right).
$$
Since $B\geq 1$ it follows that $L> 2$.  

Next, we want to bound the areas of regions $\mathcal{R}_1$ and
$\mathcal{R}_2$. To this end, we make the following observation. Let
us assume that $s>2$ and $(s,t)\in \mathcal{R}$ (in particular, this
is satisfied if $(s,t)\in \mathcal{R}_2$). Then we have $t<0$. Indeed,
if one assumes $s>2$ and $t>0$, we would have $2^2<s^2+t^3\le 1$,
which is impossible.  It follows that  
\begin{equation} 
\begin{split}
\label{clever}
1\ge \left| s^2+t^3 \right| = \left|s^2-|t|^3\right| &= \left| \left(|t|-s^{2/3}\right)\left(|t|^2+|t|s^{2/3}+s^{4/3}\right)\right|\\
&> \left| |t|-s^{2/3} \right| s^{4/3}\\
&=\left| t+s^{2/3} \right|
 s^{4/3}. 
\end{split}
\end{equation}
In this way we see that
$
\Area(\mathcal{R}_2)\ll L^{-1/6}.
$
Moreover, using \eqref{voraus}, we have 
\begin{equation} \label{lowerboundarea}
1\ll \Area(\mathcal{R}_1)\ll 1.
\end{equation} 
Hence
\begin{equation} \label{nochwas}
\Area(\mathcal{R}_1)=\Area(\mathcal{R})\left(1+O\left(L^{-1/6}\right)\right).
\end{equation}
Furthermore, we note that
$$
0<s\le L^{1/2}, \quad |t|\ll L^{1/3}
$$ 
for any $(s,t)\in \mathcal{R}_1$.

By our decomposition of $\mathcal{R}$, we have 
\begin{equation} \label{S123split}
\mathcal{S}(\mathcal{R};a,b;q)=
\mathcal{S}(\mathcal{R}_1;a,b;q)+\mathcal{S}(\mathcal{R}_2;a,b;q).
\end{equation}
In what follows
we describe how to bound the terms  $\mathcal{S}(\mathcal{R}_j;a,b;q)$
using sums involving smooth weights. 
To this end, for any smooth function $\Psi : \mathbb{R} \rightarrow
\mathbb{R}$ of rapid decay, we define 
\begin{equation} \label{SPhidef}
 \mathcal{S}(\Psi;s_0,t_0,\mathcal{X},\mathcal{Y};a,b;q)  := \sum\limits_{\substack{x,y\in \mathbb{Z}  \\  (xy,q)=1 \\ ax^2+by^3\equiv 0 \bmod{ q}}} \Psi\left(\frac{s-s_0}{\mathcal{X}}\right) \Psi\left(\frac{t-t_0}{\mathcal{Y}}\right), 
\end{equation}
where $s,t$ depend on $x,y$ via \eqref{st}.
We will choose smooth weights involving the Gaussian weight $\Gamma$
so that we can apply our work above.

\subsection{Treatment of region $\mathcal{R}_1$}

In what follows we set 
$$
\Delta:=\frac{1}{L}=\frac{1}{\log\log (2000B)}.
$$
A little thought reveals that region $\mathcal{R}_1$ can be approximated by regions $\mathcal{R}_1^{-}$ and $\mathcal{R}_1^{+}$ that each are disjoint unions of $O(L^4)$ half-open squares 
$\mathcal{Q}=[s,s+\Delta)\times [t,t+\Delta)$ of side length $\Delta$  such that 
$$
\mathcal{R}_1^{-} \subset \mathcal{R}_1 \subset \mathcal{R}_1^+
$$
and
$$
\left(1-O(\Delta^{1/2})\right)\Area(\mathcal{R}_1) < \Area(\mathcal{R}_1^-) < 
 \Area(\mathcal{R}_1)< \Area(\mathcal{R}_1^+)  < \left(1+O(\Delta^{1/2})\right)\Area(\mathcal{R}_1).
$$
Hence we have 
$$
\mathcal{S}(\mathcal{R}_1^-;a,b;q)\le \mathcal{S}(\mathcal{R}_1;a,b;q)\le \mathcal{S}(\mathcal{R}_1^+a,b;q).
$$
Let $\mathcal{M}^{\pm}$ be the set of squares whose union is $\mathcal{R}_1^{\pm}$.  It follows that
\begin{equation} \label{redtosq}
\mathcal{S}(\mathcal{R}_1^{\pm};a,b;q)=\sum\limits_{\mathcal{Q}\in \mathcal{M}^{\pm}} 
\mathcal{S}(\mathcal{Q};a,b;q).
\end{equation}

Next let $\chi$ be the characteristic function of the interval
$[-1/2,1/2)$. We approximate   $\chi$ by the smooth weight functions
$\Phi_{+},\Phi_{-} : \mathbb{R} \rightarrow \mathbb{R}$ constructed in
Lemma \ref{weightslemma}. 
There it is proved that $\Phi_{+}$, $\Phi_{-}$ have
rapid decay and satisfy the inequalities   
$
 \Phi_{-}(x)\le \chi(x)\le \Phi_{+}(x)$ for all $x\in \mathbb{R}$.
From the definitions \eqref{Somega} and \eqref{SPhidef} and we deduce that
\begin{equation} \label{SPhi-+}
\mathcal{S}(\Phi_{-};s_1,t_1,\Delta,\Delta;a,b;q) \le \mathcal{S}(\mathcal{Q};a,b;q) \le 
\mathcal{S}(\Phi_{+};s_1,t_1,\Delta,\Delta;a,b;q),
\end{equation}
where $(s_1,t_1)$ is the centre of the square $\mathcal{Q}$. 
Using \eqref{SPhidef} and the definitions of $\Phi_-$ and
$\Phi_+$ in \S \ref{s:gauss}, we further have
\begin{equation} 
\begin{split}
\label{redtogauss}
 \mathcal{S}(\Phi_{\pm};s_1,t_1,\Delta,\Delta;a,b;q) =~& 
L^2 \hspace{-0.2cm}
\int\limits_{-1/2\mp \Delta^{1/2}}^{1/2\pm \Delta^{1/2}}
\int\limits_{-1/2\mp \Delta^{1/2}}^{1/2\pm \Delta^{1/2}}
\hspace{-0.3cm}
\mathcal{S}(\Gamma;s_1+\Delta\mu,t_1+\Delta\nu,\Delta^2,\Delta^2;a,b;q)
\dif\mu\dif\nu \\  &\pm  
\varepsilon_1 L \int\limits_{-1/2\mp \Delta^{1/2}}^{1/2\pm
  \Delta^{1/2}} \mathcal{S}(\Gamma;s_1+\Delta\mu,t_1,\Delta^2,\Delta;a,b;q)
\dif \mu\\ & \pm
\varepsilon_1 L \int\limits_{-1/2\mp \Delta^{1/2}}^{1/2\pm
  \Delta^{1/2}} \mathcal{S}(\Gamma;s_1,t_1+\Delta\nu,\Delta,\Delta^2;a,b;q)
\dif\nu  \\ & + \varepsilon_1^2
\mathcal{S}(\Gamma;s_1,t_1,\Delta,\Delta;a,b;q), 
\end{split}
\end{equation} 
where
$$
\varepsilon_1=24 \exp(-L)=\frac{24}{\log (2000B)}. 
$$

In this way, the estimation of 
$\mathcal{S}(\mathcal{R}_1;a,b;q)$ has been reduced to the task of
estimating sums of the form
$\mathcal{S}(\Gamma;s_0,t_0,\mathcal{X},\mathcal{Y};a,b;q)$ in
\eqref{SPhidef}, which are precisely the sums considered in
Theorem~\ref{thm:0}, with  
\begin{equation} \label{x_0y_0XY}
X:=\sigma\mathcal{X},\quad x_0:=\sigma s_0,\quad
Y:=\tau\mathcal{Y},\quad y_0:=\tau t_0. 
\end{equation}
We then have 
$$
s_0=s_1+(i-1) \Delta\mu, \quad t_0=t_1+(j-1) \Delta\nu, \quad
\mathcal{X}=\Delta^i, \quad \mathcal{Y}=\Delta^j,
$$
for $i,j \in \{1,2\}$.
Since $t_1\ll L$, $\nu\ll 1$ and 
$y_0=t_0Y/\mathcal{Y}$ we may deduce that 
$$
Y\leq Y+|y_0|=Y_0\ll YL^3=Y (\log\log 2000B)^3.
$$
Using this and the fact that $(\log \log 2000B)^{-2} =\Delta^{2} \le
\mathcal{X}, \mathcal{Y} \le \Delta=(\log \log 2000B)^{-1}$,
together with \eqref{sigtaudef} and \eqref{x_0y_0XY}, we deduce from Theorem \ref{thm:0}
that 
\begin{align*}
\mathcal{S}(\Gamma;s_0,t_0,\mathcal{X},\mathcal{Y};a,b;q)
=~& 
\frac{\varphi(q)}{q}\cdot
\frac{B^{5/6}}{a^{1/2}b^{1/3}q^{1/6}} \cdot
\mathcal{X}\mathcal{Y}
+
O\left(\frac{(2+\ve)^{\omega(q)}B^{1/2}}{a^{1/2}}\right)\\
&+O\left(
\frac{(2+\ve)^{\omega(b)} (3+\ve)^{\omega(q)}}{b^{1/4}  } \cdot B^{1/2}
(\log B)^{2+\ve}\right)\\
&+ O\left((ss_1)^{1/2} (2+\ve)^{\omega(b)}
(2+\ve)^{\omega(q)}
4^{\omega(s)}
\cdot \frac{a^{1/2}b^{1/6}q^{1/3}(\log
  B)^{2+\ve}}{B^{1/6}}
\right),
\end{align*}
where we note that
$\rad(q)\leq q$. Let 
\begin{equation}
  \label{eq:stone}
f_\ve(a,b;q):=
\frac{(2+\ve)^{\omega(q)}}{a^{1/2} }+
\frac{(3+\ve)^{\omega(q)}}{
b^{1/4-\ve}}.
\end{equation}
Then on 
inserting this into \eqref{redtosq}, \eqref{SPhi-+} and
\eqref{redtogauss}, and recalling Lemma \ref{weightslemma} and \eqref{nochwas}, 
we conclude that
\begin{equation} \label{Sestfinal2komplett}
\begin{split}
\mathcal{S}(\mathcal{R}_1;a,b;q) 
=~& 
\frac{\varphi(q)}{q}\cdot \frac{B^{5/6}}{a^{1/2}b^{1/3}q^{1/6}}
\cdot
\Area(\mathcal{R})\left(1+O\left(L^{-1/6}\right)\right)\\
&+O\left(f_\ve(a,b;q)
B^{1/2}
(\log B)^{2+\ve}\right)\\
&+ O\left((ss_1)^{1/2} (2+\ve)^{\omega(b)+\omega(q)}
4^{\omega(s)}
\cdot \frac{a^{1/2}b^{1/6}q^{1/3}(\log
  B)^{2+\ve}}{B^{1/6}}
\right).
\end{split}
\end{equation}

\subsection{Treatment of region $\mathcal{R}_2$}\label{s:R2}

Next we deal with region $\mathcal{R}_2$. Heuristically,
$\mathcal{R}_2$ should yield an error contribution because the area of
$\mathcal{R}_2$ is small compared to that of
$\mathcal{R}_1$. Therefore it will suffice to cover $\mathcal{R}_2$ by a larger
region whose area is still small compared to that of $\mathcal{R}_1$. 

Using \eqref{clever}, we may cover the region $\mathcal{R}_2$ by a
union of $O(\log B)$ regions of the form 
\begin{equation} \label{Sregion}
\mathcal{R}(U) :=\left\{(s,t)\in \mathbb{R}^2 \ : \ \frac{4}{3}U<s\le \frac{5}{3}U,\ t<0,\ \left|t+s^{2/3}\right| < \frac{1}{U^{4/3}}\right\}, 
\end{equation}
where 
$$
\frac{3}{4}L^{1/2}\le U\le \frac{\boldsymbol{X}}{\sigma}.
$$
We observe that in certain cases the condition $U\le
\boldsymbol{X}/\sigma$ above can be strengthened. Since $s\asymp
|t|^{3/2}$ and $|t|\le \boldsymbol{Y}/\tau$ in region $\mathcal{R}_2$,
it follows that $s\ll \left(\boldsymbol{Y}/\tau\right)^{3/2}$, and
hence it suffices to suppose that 
\begin{equation} \label{Srange*}
\frac{3}{4}L^{1/2}\le U\ll \boldsymbol{U},
\end{equation}
where
\begin{equation} \label{Umathcaldef}
\boldsymbol{U}:= \min\left\{\frac{\boldsymbol{X}}{\sigma},\left(\frac{\boldsymbol{Y}}{\tau}\right)^{3/2}\right\}.
\end{equation}
In particular $\UU \geq 1$ by \eqref{voraus*}.

We may further cover $\mathcal{R}(U)$ by a union of (small) rectangles 
$$
\mathcal{Q}(s_0):=[s_0-\mathcal{X},s_0+\mathcal{X}]\times [t_0-\mathcal{Y},t_0+\mathcal{Y}]
$$ 
of side lengths $2\mathcal{X}$ and $2\mathcal{Y}$, centred at points $(s_0,t_0)$ with
$$
\frac{4}{3}U\le s_0 \le \frac{5}{3}U\quad \mbox{and} \quad
t_0=-s_0^{2/3}, 
$$
and satisfying the condition that 
\begin{equation} \label{r2covered}
\left\{ (s,t)\in \mathcal{R}(U)\ :\ s_0-\mathcal{X}\le s\le s_0+\mathcal{X}\right\} \subseteq \mathcal{Q}(s_0).
\end{equation}

Let $T$ be a parameter with 
\begin{equation} \label{Tcondi}
U^{-1}\le T\le U^{1-\varepsilon},
\end{equation}
to be chosen later appropriately. We set
\begin{equation} \label{sidelengths}
\mathcal{X} := T, \quad \mathcal{Y}:= 2TU^{-1/3}.
\end{equation} 
In the following, we suppose that $s_0\in [4U/3,5U/3]$ and show that the condition \eqref{r2covered} then holds automatically under the above choices of $\mathcal{X}$ and $\mathcal{Y}$. 

Considering the definition of $\mathcal{R}(U)$, and taking into account that the function $f(s)=-s^{2/3}$ is monotonically decreasing for $s>0$, we have
\begin{align*}
\{ (s,t)\in \mathcal{R}(U) & : s_0-\mathcal{X}\le s\le
  s_0+\mathcal{X}\} \\ 
& \subseteq
\left[s_0-\mathcal{X},s_0+\mathcal{X}\right]\times
\left[-(s_0+\mathcal{X})^{2/3}-\frac{1}{U^{4/3}},
  -(s_0-\mathcal{X})^{2/3}+\frac{1}{U^{4/3}}\right], 
\end{align*}
and hence it suffices to prove that
\begin{equation} \label{stineq1}
t_0-\mathcal{Y} \le -(s_0+\mathcal{X})^{2/3}-\frac{1}{U^{4/3}}  
\end{equation}
and 
\begin{equation} \label{stineq2}
t_0+\mathcal{Y} \ge -(s_0-\mathcal{X})^{2/3}+\frac{1}{U^{4/3}}.  
\end{equation}
Obviously we have
$$
-(s_0+\mathcal{X})^{2/3}=-s_0^{2/3}\left(1+\frac{\mathcal{X}}{s_0}\right)^{2/3}>-s_0^{2/3}\left(1+\frac{\mathcal{X}}{s_0}\right)\ge - s_0^{2/3}-\frac{T}{U^{1/3}}=t_0-\frac{T}{U^{1/3}}.
$$
Similarly,
$$
-(s_0-\mathcal{X})^{2/3}=-s_0^{2/3}\left(1-\frac{\mathcal{X}}{s_0}\right)^{2/3}\le -s_0^{2/3}\left(1-\frac{\mathcal{X}}{s_0}\right)\le -s_0^{2/3}+\frac{T}{U^{1/3}}=t_0+\frac{T}{U^{1/3}}. 
$$
Recalling that $\mathcal{Y}= 2TU^{-1/3}$ and $T\ge U^{-1}$, we
therefore obtain \eqref{stineq1} and \eqref{stineq2}, and so
\eqref{r2covered} holds.

Now instead of covering $\mathcal{R}(U)$ discretely by rectangles
$\mathcal{Q}(s_0)$ and summing up their contributions, we rather
integrate, which will be useful because we shall later take advantage
of cancellations occurring in this integration. To be precise, we
claim that
\begin{equation}\label{int1}
\mathcal{S}(\mathcal{R}(U);a,b;q) \le \frac{1}{T}\int\limits_{-\infty}^{\infty}
W\left(\frac{s_0}{U}\right) \mathcal{S}(\mathcal{Q}(s_0);a,b;q) \dif s_0,
\end{equation}
where $W$ is a smooth weight function with 
\begin{equation} \label{Wconditions}
\mbox{supp}(W)\subseteq [1,2] \quad \mbox{and} \quad W(u)=1 \mbox{ for
} \frac{4}{3}\le u\le \frac{5}{3}. 
\end{equation}
The inequality \eqref{int1} is seen as follows. Using \eqref{Somega},
we have 
\begin{align*} 
\frac{1}{T}\int\limits_{-\infty}^{\infty} &W\left(\frac{s_0}{U}\right)
\mathcal{S}(\mathcal{Q}(s_0);a,b;q) \dif s_0\\
&= 
\frac{1}{T}\int\limits_{-\infty}^{\infty} W\left(\frac{s_0}{U}\right)
 \# \left\{(x,y)\in \mathbb{Z}^2   :  
\begin{array}{l}
(xy,q)=1,\ ax^2+by^3\equiv 0 \bmod{ q},\\
(s,t)\in \mathcal{Q}(s_0)
\end{array}
\right\} \dif s_0\\
&\ge  \sum\limits_{x,y} \frac{1}{T}\int\limits_{\mathcal{M}(s,t)} W\left(\frac{s_0}{U}\right) \dif s_0 ,
\end{align*}
where the sum is over $(x,y)\in \ZZ^2$ such that 
$ax^2+by^3\equiv 0 \bmod{ q}, (xy,q)=1$ and  $(s,t)\in \mathcal{R}(U)$,
and where 
$$
\mathcal{M}(s,t):=\left\{ s_0 \in \left[\frac{4}{3}U, \frac{5}{3}U\right] \ : \ (s,t)\in \mathcal{Q}(s_0)\right\}. 
$$
By \eqref{r2covered} and the fact that $\mathcal{X}=T$, we have
$$
\mathcal{M}(s,t)=\left[4U/3, 5U/3\right]\cap [s-T,s+T]  
$$
for any $(s,t) \in \mathcal{R}(U)$. 
Moreover, by \eqref{Wconditions}, we have
$
W\left(s_0/U\right)=1$ whenever $s_0\in \mathcal{M}(s,t)$. 
Hence it follows that
$$
\sum_{x,y} \frac{1}{T}\int\limits_{\mathcal{M}(s,t)} W\left(\frac{s_0}{U}\right) \dif s_0=
\sum_{x,y} \frac{\meas(\mathcal{M}(s,t))}{T} \ge \sum_{x,y} 1 =
\mathcal{S}(\mathcal{R}(U);a,b;q).
$$
Combining these inequalities we therefore 
obtain the desired inequality \eqref{int1}.

From \eqref{SPhidef} and \eqref{int1}, we further deduce that
\begin{equation}\label{int}
\mathcal{S}(\mathcal{R}(U);a,b;q) \ll \frac{1}{T}\int\limits_{-\infty}^{\infty}W\left(\frac{s_0}{U}\right)\mathcal{S}(\Gamma;s_0,t_0,\mathcal{X},\mathcal{Y};a,b;q) \dif s_0,
\end{equation}
where $\Gamma$ is defined as in \eqref{gaussiandef}
and $\mathcal{S}(\Gamma;s_0,t_0,\mathcal{X},\mathcal{Y};a,b;q)$
is given in \eqref{SPhidef}.
Thus the route is open for a modification of the approach 
used to prove 
Theorem \ref{thm:0}. Using  \eqref{SRaftertrans} and \eqref{int}, we
obtain 
\begin{equation} \label{intincluded}
\mathcal{S}(\mathcal{R}(U);a,b;q)
\ll \mathcal{S}(U):=\frac{XY}{q^2T} \sum\limits_{m\in \mathbb{Z}}
\sum\limits_{n \in \mathbb{Z}} \Gamma\left(\frac{mX}{q}\right)
\Gamma\left(\frac{nY}{q}\right) \mathcal{I}(m,n;q)\mathcal{E}(m,n;q), 
\end{equation}
where $X,Y$ are given by \eqref{sigtaudef} and \eqref{x_0y_0XY},  and
furthermore, 
$$
\mathcal{I}(m,n;q):=\int\limits_{-\infty}^{\infty}
W\left(\frac{s_0}{U}\right) \e\left(-\frac{m\sigma s_0-n\tau
    s_0^{2/3}}{q}\right) \dif s_0. 
$$
As previously we split the sum on the right-hand side of
\eqref{intincluded} into  
\begin{equation} \label{SUsplitting}
\mathcal{S}(U)=\mathcal{S}_0(U)+\mathcal{S}_1(U)+\mathcal{S}_2(U),
\end{equation} 
where $\mathcal{S}_{0}(U)$ denotes the contribution of $m=0$ and
$n=0$, $\mathcal{S}_{1}(U)$ the contribution of $m\not=0$ and $n$
arbitrary, and $\mathcal{S}_{2}(U)$ the contribution of $m=0$ and
$n\not=0$. If $m=0$ and $n=0$, then we note that
$$
\mathcal{I}(0,0;q)= O(U).
$$
Since $\mathcal{E}(0,0;q)=\varphi(q)$ we obtain
\begin{equation} \label{conto1}
\mathcal{S}_0(U)\ll \frac{\varphi(q)}{q^2}\cdot XY\cdot \frac{U}{T}.
\end{equation}

Next if $m=0$ and $n\not=0$, then using the remark after Lemma 2.1 in 
\cite{Ivic}, we get
$$
\mathcal{I}(0,n;q)= O\left(\frac{U^{1/3}q}{|n|\tau}\right).
$$
Hence 
$$
\mathcal{S}_2(U) \ll \frac{XYU^{1/3}}{\tau q T} \sum\limits_{n \not=0}
\frac{1}{|n|} \cdot \Gamma\left(\frac{nY}{q}\right) 
 |\mathcal{E}(0,n;q)|. 
$$
Arguing as in the proof of \eqref{S10est} we easily obtain
\begin{equation} \label{conto22}
\mathcal{S}_2(U) \ll \frac{XY\rad(q)^{1/2}U^{1/3}}{\tau q T}\cdot
(\log B)(2+\ve)^{\omega(q)}.
\end{equation}

It remains to estimate
$$
\mathcal{S}_1(U)=\frac{XY}{q^2T} \sum\limits_{m\not=0} \sum\limits_{n
  \in \mathbb{Z}} \Gamma\left(\frac{mX}{q}\right)
\Gamma\left(\frac{nY}{q}\right) \mathcal{I}(m,n;q)\mathcal{E}(m,n;q). 
$$
Recall from \eqref{x_0y_0XY} that $x_0=\sigma s_0$. Pulling out the
integral and making a change of variables $s_0\rightarrow x_0$, we
obtain 
$$
\mathcal{S}_1(U)=\frac{XY}{\sigma q^2T} \int\limits_{-\infty}^{\infty}
\sum\limits_{m\not=0} \sum\limits_{n \in \mathbb{Z}}
\Gamma\left(\frac{mX}{q}\right) \Gamma\left(\frac{nY}{q}\right) 
W\left(\frac{x_0}{\sigma U}\right) \e\left(-\frac{mx_0+ny_0}{q}\right)
\mathcal{E}(m,n;q) \dif x_0, 
$$
where 
\begin{equation} \label{x0y0connect}
y_0=-\delta x_0^{2/3}
\end{equation}
and
\begin{equation} \label{sigmataudelta}
\delta:=\frac{\tau}{\sigma^{2/3}}=\left(\frac{a}{b}\right)^{1/3},
\end{equation}
by \eqref{sigtaudef}.

Denote the term $\mathcal{S}_1$ in \eqref{S11} by $\mathcal{S}_1(x_0,y_0)$. Then 
$$
\mathcal{S}_1(U)=\frac{1}{\sigma T} \int\limits_{-\infty}^{\infty}
W\left(\frac{x_0}{\sigma U}\right) \mathcal{S}_1(x_0,y_0) \dif x_0. 
$$
Before we estimate $\mathcal{S}_1(U)$, we recall the main steps in the
estimation of $\mathcal{S}_1=\mathcal{S}_1(x_0,y_0)$. Our starting
point was the identity \eqref{split2}. In \eqref{remcop} the sum
$S(u)$ therein was further split into
sums $S(u,t)$ for which we derived
the identity \eqref{Srew}. The exponential sum $F_x(h)$ is
independent of $x_0$ and $y_0$ and was estimated 
separately for the cases $h\neq 0$ and $h=0$. In the case $h\not=0$
we got an estimate in which only one factor depends on $h$, namely the
factor $(h,b_1)^{1/2}$. The function $\hat\Psi$ in \eqref{Srew} can be
written as an exponential integral and depends on $y_0$. Using
\eqref{PhidefFourier} its
estimation was the subject of \S \ref{s:airy}. Furthermore, the series
$$
\sum\limits_{h\not=0} |\hat\Psi(h\alpha)| 
 (h,b_1)^{1/2} 
$$
was estimated in \eqref{mirfallenkeinenamenmehrein}. 
For $\hat\Psi(0)$ we obtained the estimate
\eqref{Phi0bound2}. Putting everything together, we arrived at an
estimate for $\mathcal{S}_1$ which 
was recorded in \eqref{T1term}, \eqref{T2term} and \eqref{Sestfinal0*}. 
We recall that $\mathcal{S}_1^{(0)}$ and
$\mathcal{S}_1^{(1)}$ are the contributions of $h=0$ and $h\not=0$,
respectively. Hence the terms on the right-hand sides of
\eqref{mirfallenkeinenamenmehrein} and \eqref{Phi0bound2} correspond
to the terms on the right-hand sides of \eqref{T2term} and
\eqref{T1term}, respectively.   

Our strategy for the estimation of $\mathcal{S}_1(U)$ is as
follows. We denote the terms $\mathcal{S}_1^{(i)}$ in
\eqref{Sestfinal0*} by 
$\mathcal{S}_1^{(i)}(x_0,y_0)$, for $i=0,1$. Then 
\begin{equation} \label{S1Usplit}
\mathcal{S}_1(U)=
\mathcal{S}_1^{(0)}(U)
+\mathcal{S}_1^{(1)}(U),
\end{equation}
where, for $i=0,1$, 
$$
\mathcal{S}_1^{(i)}(U) := 
\frac{1}{\sigma T} \int\limits_{-\infty}^{\infty}
W\left(\frac{x_0}{\sigma U}\right) \mathcal{S}_1^{(i)}(x_0,y_0) \dif
x_0.  
$$

We estimate $\mathcal{S}_1^{(0)}(U)$ by integrating trivially the
right-hand 
side of \eqref{T1term}, giving 
\begin{align*}
\mathcal{S}_1^{(0)}(U) \ll~&
\frac{(\log B)^{2+\ve}}{\sigma T}\int_{\sigma U}^{2\sigma U} 
\frac{(2+\ve)^{\omega(b)} (3+\ve)^{\omega(q)}}{b^{1/4}  q^{1/2} }
\left(
\frac{a^{1/2}XY}{Y_0}  
+\frac{a^{1/4}b^{1/4}X^{1/2}Y }{Y_0^{1/4}}\right)
\d x_0.
\end{align*}
Taking
$Y_0=|y_0|+Y>|y_0|=\delta x_0^{2/3}$ by 
\eqref{x0y0connect} into account, we see that 
$$
\int_{\sigma U}^{2\sigma U} Y_0^{-\theta} \d x_0 
< 
\frac{1}{\delta^{\theta}} 
\int_{\sigma U}^{2\sigma U} x_0^{-2\theta/3} \d x_0 
\ll \frac{(\sigma U)^{1-2\theta/3}}{\delta^{\theta}},
$$
for any $\theta\in [0,3/2).$
Using
\eqref{sigmataudelta} we may therefore conclude that
\begin{equation}\label{T1Uestim}
\mathcal{S}_1^{(0)}(U) \ll
\frac{(2+\ve)^{\omega(b)} (3+\ve)^{\omega(q)} (\log B)^{2+\ve}}{
b^{1/4} q^{1/2}T }
\left(
\frac{a^{1/2}XY U^{1/3}}{\tau}   
+\frac{a^{1/4}b^{1/4}X^{1/2}Y U^{5/6}}{\tau^{1/4}}\right).
\end{equation}

\subsection{Estimation of $\mathcal{S}_1^{(1)}(U)$}

We now turn to the more delicate estimation of
$\mathcal{S}_1^{(1)}(U)$. For this we will first 
write $\mathcal{S}_1^{(1)}$ explicitly as an identity. 
Then we pull in the integral and
integrate non-trivially. Roughly speaking, we will proceed to 
compare termwise
the sizes of the 
integrals and integrands. This will tell us how the estimate
\eqref{T2term} needs to be adjusted to obtain a final
estimate for   $\mathcal{S}_1^{(1)}(U)$.  

Recalling \eqref{split2}, \eqref{remcop} and 
\eqref{Srew}, we may write    
\begin{align*} 
\mathcal{S}_1^{(1)} =~& \frac{XY}{q^2}
\sum\limits_{\substack{(6r_1,r_2)=1\\ r_1r_2=q}} \
\sum\limits_{\substack{d_1\mid r_1,\  d_2\mid r_2\\ \mbox{\scriptsize
      rad}(\hat{r}_1)^2\mid e_1\\ \mu^2(e_2)=1}}\ \sideset{}{'}
\sum\limits_{d_1'\mid r_1} d_2'\sqrt{r_1d_1}  
\sum\limits_{\substack{u\not=0\\ (u,e_1e_2)=1}}
\Gamma\left(\frac{2^{f_2}d_1d_2uX}{q}\right)  \\  
& \times
\sum\limits_{\substack{t\mid e_1'\\ (t,e)=1}}
\mu(t)\left(\frac{t}{e}\right)  
\e\left(-\frac{2^{f_2}d_1d_2ux_0}{q}\right) \frac{\alpha}{D}
\sum\limits_{x=1}^D \sum\limits_{y=1}^D  
\e\left(\frac{xy}{D}\right)Q''(y) \sum\limits_{h\not= 0}
\hat\Psi(\alpha h) F_x(h), 
\end{align*} 
where $\alpha$ is defined as in \eqref{alphabeta}.
It follows that
\begin{align*}
\mathcal{S}_1^{(1)}(U) \ll ~& 
\frac{XY}{q^2}
\sum\limits_{\substack{(6r_1,r_2)=1\\ r_1r_2=q}} \
\sum\limits_{\substack{d_1\mid r_1,\  d_2\mid r_2\\ \mbox{\scriptsize
      rad}(\hat{r}_1)^2\mid e_1\\ \mu^2(e_2)=1}}\  \sideset{}{'}
\sum\limits_{d_1'\mid r_1} d_2'\sqrt{r_1d_1}  
\sum\limits_{\substack{u\not=0\\ (u,e_1e_2)=1}}
\Gamma\left(\frac{2^{f_2}d_1d_2uX}{q}\right)  \\ & \times
\sum\limits_{\substack{t\mid e_1'\\ (t,e)=1}}
|\mu(t)|
\alpha\sum\limits_{x=1}^D 
\sum\limits_{h\not= 0}
|F_x(h)| \\   &
\times  \frac{1}{\sigma T}
\left|
\int\limits_{-\infty}^{\infty} W\left(\frac{x_0}{\sigma U}\right) 
\e\left(-\frac{2^{f_2}d_1d_2ux_0}{q}\right) \hat\Psi_{y_0}(\alpha h)
\dif x_0\right|, 
\end{align*} 
where the subscript $y_0$ indicates the dependency of 
the function
$\hat\Psi$ on $y_0$, which in turn depends on $x_0$.  
A key point in the estimation of $\mathcal{S}_1^{(1)}$ was our bound
\eqref{mirfallenkeinenamenmehrein} for 
$$
\sum\limits_{h\not= 0} \left| \hat\Psi_{y_0}(\alpha h) \right|  (h,b_1)^{1/2}.
$$
We recall that the factor $(h,b_1)^{1/2}$ that appears here 
came from our
estimate of $F_x(h)$ in \S \ref{s:second}.
For the estimation of
$\mathcal{S}_1^{(1)}(U)$, we require a bound for 
\begin{equation} \label{instead}
\mathcal{L}:=\frac{1}{\sigma T} \sum\limits_{h\not= 0} \left|
  \int\limits_{-\infty}^{\infty} W\left(\frac{x_0}{\sigma U}\right) 
\e\left(-\frac{2^{f_2}d_1d_2ux_0}{q}\right) \hat\Psi_{y_0}(\alpha h)
  \dif x_0\right|  (h,b_1)^{1/2} 
\end{equation}
instead. After deriving a bound for this sum, we shall compare the
terms appearing in it with the corresponding terms on the right-hand
side of \eqref{mirfallenkeinenamenmehrein}. This will allow us to
infer a bound for $\mathcal{S}_1^{(1)}(U)$ directly from our bound
\eqref{T2term} for $\mathcal{S}_1^{(1)}$. 

We recall that 
$$
\hat\Psi_{y_0}(\alpha h)=F(\gamma+\alpha
h,\beta)=F\left(\frac{y_0}{Y}+\alpha
  h,\beta\right)=F\left(-\frac{\delta x_0^{2/3}}{Y}+\alpha
  h,\beta\right), 
$$
where $\alpha$ and $\beta$ are defined in \eqref{alphabeta} and $\delta$ is defined as in \eqref{sigmataudelta}. Hence
the integrand in \eqref{instead} is
$$
\mathcal{I}_h(x_0):=
W\left(\frac{x_0}{\sigma U}\right)
\e\left(-\frac{2^{f_2}d_1d_2ux_0}{q}\right) F\left(-\frac{\delta
    x_0^{2/3}}{Y}+\alpha h,\beta\right).
$$

As in \eqref{largeh} the contribution of $|h|>B^{\kappa}$ to
\eqref{instead} is negligible if $\kappa$ is large enough. 
We split the remaining contribution to
\eqref{instead} as follows. 
\begin{equation}
\begin{split}
\label{remaining}
\sum\limits_{1\le |h|\le B^{\kappa}} \left|
  \int\limits_{-\infty}^{\infty} \mathcal{I}_h(x_0) \dif x_0 \right|
 (h,b_1)^{1/2} 
=~& \sum_{i=1,2}
\sum\limits_{1\le |h|\le B^{\kappa}} \left|
  \int\limits_{\mathcal{M}_i} \mathcal{I}_h(x_0) \dif x_0 \right|
 (h,b_1)^{1/2},
\end{split}
\end{equation}
where 
$$
\mathcal{M}_1:=\left\{ x_0\in [\sigma U,2\sigma U] \ : \
  \left|-\frac{\delta x_0^{2/3}}{Y}+\alpha h\right|\le 1\right\} \quad
\mbox{and} \quad 
\mathcal{M}_2:=[\sigma U,2\sigma U] \setminus \mathcal{M}_1.
$$

Let us begin with the treatment of the integral over $\mathcal{M}_1$. 
Using \eqref{Fbound2} we have
\begin{equation} \label{integrandtrivial}
\mathcal{I}_h(x_0)\ll 1.
\end{equation} 
It remains to estimate the Lebesgue measure of $\mathcal{M}_1$. To
this end, we note that  \eqref{x_0y_0XY}, \eqref{Tcondi}, \eqref{sidelengths}
and \eqref{sigmataudelta} imply 
\begin{equation} \label{wirbrauchendas}
\frac{\delta x_0^{2/3}}{Y} \asymp \frac{U}{T}\gg U^{\varepsilon} \quad
\mbox{if } x_0\in[\sigma U,2\sigma U]. 
\end{equation}
Hence we have
\begin{equation} \label{alphahsize}
\alpha |h| \asymp \frac{\delta x_0^{2/3}}{Y} \quad \mbox{if }
x_0\in[\sigma U,2\sigma U] \mbox{ and } \left|-\frac{\delta
    x_0^{2/3}}{Y}+\alpha h\right|\le 1. 
\end{equation}
Using this we claim that $\mathcal{M}_1$ is an interval of length
\begin{equation} \label{thelength}
\mbox{meas}(\mathcal{M}_1) \ll \frac{\sigma U}{\alpha |h|}.
\end{equation} 
To see this we  let $m_1$ and $m_2$ be the two endpoints of
the interval $\mathcal{M}_1$ in question. We have 
$m_1 \asymp m_2 \asymp \sigma U.$ 
Under this condition Taylor's theorem gives
$$
m_1^{2/3}-m_2^{2/3} \asymp (m_1-m_2) m_2^{-1/3}.
$$
Furthermore, by the definition of our interval $\mathcal{M}_1$, we have
$$
\frac{\delta m_1^{2/3}}{Y}-\frac{\delta m_2^{2/3}}{Y} \ll 1.
$$
Combining these two inequalities, we deduce that
$$
m_1-m_2 \ll \frac{Y m_2^{1/3}}{\delta}\ll
\frac{Y\sigma^{1/3}U^{1/3}}{\delta}\ll \frac{\sigma U}{\alpha |h|},
$$
by \eqref{alphahsize} and $x_0\in [\sigma U,2\sigma U]$. This
therefore gives \eqref{thelength}. Employing \eqref{wellknown}, \eqref{integrandtrivial},
\eqref{thelength} and partial summation, we get 
\begin{equation} \label{firstintest}
\sum\limits_{1\le |h|\le B^{\kappa}} \left|
  \int\limits_{\mathcal{M}_1} \mathcal{I}_h(x_0) \dif x_0 \right|
 (h,b_1)^{1/2}\ll \frac{\sigma U}{\alpha} 
\sum\limits_{1\le |h|\le B^{\kappa}} \frac{(h,b_1)^{1/2}}{|h|} \ll
\frac{\sigma U \log B}{\alpha} \cdot (1+\ve)^{\omega(b_1)}. 
\end{equation}

Now we turn to the second integral on the right-hand side of \eqref{remaining}. We write $\mathcal{M}_2=\mathcal{M}_2^-\cup \mathcal{M}_2^+$, where
\begin{align*}
\mathcal{M}_2^-
&:=\left\{ x_0\in [\sigma U,2\sigma U] \ : \ -\frac{\delta x_0^{2/3}}{Y}+\alpha h\le  -1\right\}\\
\mathcal{M}_2^+&:=\left\{ x_0\in [\sigma U,2\sigma U] \ : \ -\frac{\delta x_0^{2/3}}{Y}+\alpha h\ge  1\right\}.
\end{align*}
If $x_0\in \mathcal{M}_2^-$, then we have
\begin{align*}
F\left(-\frac{\delta x_0^{2/3}}{Y}+\alpha h,\beta\right) =~&
\frac{2^{1/2}\exp\left(-\pi \left|-\delta x_0^{2/3}/Y+\alpha
      h\right|/(3\beta)\right)}{\left(3\left|-\delta
      x_0^{2/3}/Y+\alpha h\right|\beta\right)^{1/4}} \\
&\times
\cos\left(2\pi\left(\frac{1}{8}-\frac{2\left|-\delta
        x_0^{2/3}/Y+\alpha h\right|^{3/2}}{3^{3/2}
      \beta^{1/2}}\right)\right) 
+O\left(\frac{1}{\left|-\delta x_0^{2/3}/Y+\alpha h\right|}\right)  
\end{align*}
by \eqref{Fbound3c}. It follows that
\begin{equation} \label{M-split}
\int\limits_{\mathcal{M}_2^-} \mathcal{I}_h(x_0) \dif x_0 \ll 
\frac{1}{\beta^{1/4}} \cdot \left| \int\limits_{\mathcal{M}_2^-}  \frac{Z(x_0)}{\left|-\delta x_0^{2/3}/Y+\alpha h\right|^{1/4}} \dif x_0 \right|+  
\int\limits_{\mathcal{M}_2^-} \frac{\dif x_0}{\left|-\delta x_0^{2/3}/Y+\alpha h\right|}, 
\end{equation}
where 
\begin{equation} \label{oscillator}
\begin{split}
Z(x_0):=~&
W\left(\frac{x_0}{\sigma U}\right)
\exp\left(-\frac{\pi \left|-\delta x_0^{2/3}/Y+\alpha
      h\right|}{3\beta}\right)
\e\left(-\frac{2^{f_2}d_1d_2ux_0}{q}\right)\\
&\times
\cos\left(2\pi\left(\frac{1}{8}-\frac{2\left|-\delta
        x_0^{2/3}/Y+\alpha h\right|^{3/2}}{3^{3/2}
      \beta^{1/2}}\right)\right). 
\end{split}
\end{equation}

We begin by estimating the second integral on the right-hand side of \eqref{M-split}. First assume that $\alpha h\ge \delta (\sigma U)^{2/3}/(2Y)$. Let $x_1$ be a positive real number such that
$$
\alpha h=\frac{\delta x_1^{2/3}}{Y}-1.
$$
If $x_1> 2\sigma U$, then $\mathcal{M}_2^-$ is empty. Thus we may assume that $\sigma U/3\le x_1\le 2\sigma U$ 
since $\delta x_1^{2/3}/Y\geq \alpha h \geq \delta (\sigma U)^{2/3}/(2Y)$. Hence we have
$$
\int\limits_{\mathcal{M}_2^-} \frac{\dif x_0}{\left|-\delta x_0^{2/3}/Y+\alpha h\right|} \le \int\limits_{x_1}^{2\sigma U} \frac{\dif x_0}{\left|-\delta x_0^{2/3}/Y+\alpha h\right|}.
$$
Now making a change of variables $z=\delta x_0^{2/3}/Y-\alpha h$, we
deduce that 
$$
\int\limits_{\mathcal{M}_2^-} \frac{\dif x_0}{\left|-\delta x_0^{2/3}/Y+\alpha h\right|} \le \frac{3}{2} \int\limits_{1}^{a} \frac{Yx_0^{1/3}}{\delta z} \dif z,
$$
where $a=\delta(2\sigma U)^{2/3}/Y-\alpha h$.  Similarly as in
\eqref{wirbrauchendas} and \eqref{alphahsize}, 
we have $\delta(2\sigma U)^{2/3}/Y \asymp U/T$
and 
$$
\alpha |h| \asymp \frac{\delta x_1^{2/3}}{Y} \asymp \frac{\delta (\sigma U)^{2/3}}{Y}.
$$
Hence $\log a\ll \log U$. 
Using this and the fact that $x_0\asymp \sigma U$, it follows that
\begin{equation} \label{afterchanging}
\int\limits_{\mathcal{M}_2^-} \frac{\dif x_0}{\left|-\delta
    x_0^{2/3}/Y+\alpha h\right|} \ll \min\left\{
\frac{Y(\sigma U)^{1/3}}{\delta}, \frac{\sigma U}{\alpha |h|} \right\}
\log U.
\end{equation}
If $\alpha h<\delta (\sigma U)^{2/3}/(2Y)$, then $\left|-\delta x_0^{2/3}/Y+\alpha h\right| \gg \delta (\sigma U)^{2/3}/Y$ and $\left|-\delta x_0^{2/3}/Y+\alpha h\right| \gg \alpha |h|$
for every $x_0\in [\sigma U,2\sigma U]$, and hence
\eqref{afterchanging} holds trivially.  

Next, we estimate the first integral on the right-hand side of \eqref{M-split}. Using integration by parts, we write this integral as 
\begin{equation}
\begin{split} \label{intparts}
\int\limits_{\mathcal{M}_2^-} \frac{Z(x_0)}{\left|-\delta x_0^{2/3}/Y+\alpha h\right|^{1/4}}  \dif x_0
=~& \frac{1}{\left(\delta (2\sigma U)^{2/3}/Y-\alpha h\right)^{1/4}}
\int\limits_{x_2}^{2\sigma U} Z(x_0) \dif x_0 \\
&+ 
\frac{\delta}{6Y} \int\limits_{x_2}^{2\sigma
  U}\frac{1}{z^{1/3}\left(\delta z^{2/3}/Y-\alpha h\right)^{5/4}}
\left(\int\limits_{x_2}^z Z(x_0) \dif x_0\right)\dif z,   
\end{split}
\end{equation}
where $x_2$ is the lower endpoint of the interval $\mathcal{M}_2^-$ and hence $\mathcal{M}_2^-=[x_2,2\sigma U]$.
We note that the second derivative of the function in the cosine on the right-hand side of 
\eqref{oscillator} satisfies
\begin{align*}
\frac{\dif^2}{\dif x_0^2} \left(\frac{1}{8}-\frac{2\left|-\delta x_0^{2/3}/Y+\alpha h\right|^{3/2}}{3^{3/2} \beta^{1/2}}\right)
&= \frac{\dif^2}{\dif x_0^2} \left(\frac{1}{8}-\frac{2\left(\delta x_0^{2/3}/Y-\alpha h\right)^{3/2}}{3^{3/2} \beta^{1/2}}\right)\\
&= -\frac{2\delta \alpha h}{3^{5/2}x_0^{4/3}Y\beta^{1/2}}\left(\frac{\delta x_0^{2/3}}{Y}-\alpha h\right)^{-1/2}.
\end{align*}
Now we write the cosine on the right-hand side of \eqref{oscillator} in the form $\cos(2\pi x)=(\e(x)+\e(-x))/2$, combine the resulting exponential terms with the term 
$$
\e\left(-\frac{2^{f_2}d_1d_2ux_0}{q}\right),
$$
use integration by parts to remove the smooth weight 
\begin{equation} \label{smoothweight}
W\left(\frac{x_0}{\sigma U}\right)
\exp\left(-\frac{\pi \left|-\delta x_0^{2/3}/Y+\alpha
      h\right|}{3\beta}\right)
\end{equation}
and employ Lemma 3.2 in \cite{IwKo} with
$k=2$. This leads to the conclusion that
$$
\int\limits_{x_2}^z Z(x_0) \dif x_0 \ll \frac{z^{2/3} Y^{1/2}
  \beta^{1/4}}{\delta^{1/2}\alpha^{1/2} |h|^{1/2}}\left(\frac{\delta
    z^{2/3}}{Y}-\alpha h\right)^{1/4}. 
$$
Combining this with \eqref{intparts} we get
$$
\int\limits_{\mathcal{M}_2^-} \frac{Z(x_0)}{\left|-\delta x_0^{2/3}/Y+\alpha h\right|^{1/4}}  \dif x_0 \ll \frac{\beta^{1/4}}{\alpha^{1/2} |h|^{1/2}} \left(\frac{(\sigma U)^{2/3} Y^{1/2}}{\delta^{1/2}}+
\frac{(\sigma U)^{1/3} \delta^{1/2}}{Y^{1/2}}\int\limits_{x_2}^{2\sigma U}
\frac{\dif z}{\delta z^{2/3}/Y-\alpha h}\right). 
$$
The integral on the right-hand side of this 
equals that on the left-hand side of \eqref{afterchanging}. Hence
we obtain 
\begin{equation} \label{intpartsest*}
\int\limits_{\mathcal{M}_2^-} \frac{Z(x_0)}{\left|-\delta
    x_0^{2/3}/Y+\alpha h\right|^{1/4}}  \dif x_0 \ll \frac{(\sigma
  U)^{2/3} Y^{1/2} \beta^{1/4}}{\delta^{1/2}\alpha^{1/2}
  |h|^{1/2}}\cdot \log U. 
\end{equation}

To proceed we will make use of the fact that we can assume that
\begin{equation} \label{strangeassumption}
\frac{\delta (\sigma U)^{2/3} }{Y}\ll \beta (\log B)^{1+\varepsilon_0}
\end{equation}
for any $\varepsilon_0>0$. This is seen as follows. The contribution of the terms with 
$$
|u|\geq \frac{q}{d_1d_2X} \cdot (\log B)^{(1+\varepsilon_0)/2}
$$
to $\mathcal{S}_1^{(1)}$ is negligible, since for such $u$ we 
have 
\begin{equation} \label{similarasin}
\Gamma\left(\frac{2^{f_2}d_1d_2uX}{q}\right)\ll  \exp\left(-c_1(\log
  B)^{1+\varepsilon_0}\right) 
\end{equation}
for some constant $c_1$, which is $\ll B^{-c_2}$ for any fixed
$c_2>0$. Hence, we can assume that   
$$
|u|< \frac{q}{d_1d_2X} \cdot (\log B)^{(1+\varepsilon_0)/2}.
$$
Then using
\eqref{a1b1}, 
\eqref{alphabeta}, 
\eqref{sigtaudef}, \eqref{x_0y_0XY}, \eqref{sidelengths} and 
\eqref{sigmataudelta}, we deduce that
$$
\beta (\log B)^{1+\varepsilon_0} \gg \frac{aq^2}{bu^2d_1^2d_2^2Y^3} \cdot (\log B)^{1+\varepsilon_0} \gg \frac{aX^2}{bY^3} \gg \frac{a\sigma^2U}{b\tau^3T} =
\frac{U}{T}\gg  \frac{\tau U^{2/3}}{Y} = \frac{\delta (\sigma
  U)^{2/3}}{Y}. 
$$

Combining \eqref{intpartsest*} with 
\eqref{M-split} and \eqref{afterchanging}, we get
\begin{equation} \label{M-est}
\int_{\mathcal{M}_2^-}
\mathcal{I}_h(x_0)\d x_0\ll 
\left(
\frac{\sigma U}{\alpha |h|}+
\frac{(\sigma U)^{2/3}Y^{1/2}}{\delta^{1/2}\alpha^{1/2}|h|^{1/2}}
\right) \log U,
\end{equation}
with the convention that the second term can be removed unless $
\alpha |h|\leq \beta (\log B)^{1+\ve}$. This convention is justified by the observation that if $\alpha |h|> \beta (\log B)^{1+\ve}$ then 
the integral in \eqref{intpartsest*} becomes negligible. To see this, we 
bound the weight function in \eqref{smoothweight}  similarly as in \eqref{similarasin} upon noting that the term $\alpha h$
dominates the term $\delta x_0^{2/3}/Y$ if $\alpha |h|> \beta (\log B)^{1+\ve}$ and \eqref{strangeassumption} holds with $\varepsilon_0<\varepsilon$.

The same estimate as in  \eqref{M-est} can be established for the
corresponding integral over $\mathcal{M}_2^+$ in a similar way. In
fact $\mathcal{M}_2^+$ 
is handled more easily than $\mathcal{M}_2^{-}$ since the role of
\eqref{Fbound3c} is replaced by the simpler \eqref{Fbound1}. 
Now, using $\mathcal{M}_2=\mathcal{M}_2^+\cup \mathcal{M}_2^-$, we have
$$
\int\limits_{\mathcal{M}_2} \mathcal{I}_h(x_0) \dif x_0 \ll 
\left(
\frac{\sigma U}{\alpha |h|}+
\frac{(\sigma U)^{2/3}Y^{1/2}}{\delta^{1/2}\alpha^{1/2}|h|^{1/2}}
\right) \log U,
$$
where 
the second term can be removed unless $
\alpha |h|\leq \beta (\log B)^{1+\ve}$.
Hence 
\begin{align*}
\sum\limits_{1\le |h|\le B^{\kappa}} &\left| \int\limits_{\mathcal{M}_2} \mathcal{I}_h(x_0) \dif x_0 \right|  (h,b_1)^{1/2}
\\
&
\ll \left(
\frac{\sigma U}{\alpha}
\sum\limits_{1\le |h|\ll B^\kappa}
\hspace{-0.2cm}
\frac{(h,b_1)^{1/2}}{|h|} + 
\frac{(\sigma U)^{2/3} Y^{1/2}}{\delta^{1/2}\alpha^{1/2}}
\hspace{-0.2cm}
\sum\limits_{1\le |h|\le 
\beta (\log B)^{1+\ve}/\alpha} 
\hspace{-0.2cm}
\frac{(h,b_1)^{1/2}}{|h|^{1/2}}\right) \log U\\ 
&\ll \left(
\frac{\sigma U }{\alpha} 
+ \frac{(\sigma U)^{2/3} \beta^{1/2}Y^{1/2}}{
\delta^{1/2}\alpha}\right)
(1+\ve)^{\omega(b_1)} (\log B)^{1+\ve}\log U.
\end{align*}
Combining this with \eqref{remaining} and \eqref{firstintest}
we get
\begin{equation}\label{insteadbound}
\mathcal{L}\ll
\left(
\frac{U }{\alpha T}
+ \frac{U^{2/3} \beta^{1/2}Y^{1/2}}{
\alpha \delta^{1/2}\sigma^{1/3} T} \right) 
(1+\ve)^{\omega(b_1)} (\log B)^{1+\ve}\log U,
\end{equation}
for the expression in \eqref{instead}.

We compare \eqref{insteadbound} with 
\eqref{mirfallenkeinenamenmehrein} for $\log P\ll \log B$ and $\gamma=0$ (and hence $Y_0=Y$). 
The first term on the right-hand side of \eqref{insteadbound} equals the term 
\begin{equation} \label{thefirstterm}
\frac{\log B}{\alpha} \cdot (1+\ve)^{\omega(b_1)} 
\end{equation}
on the right-hand side of \eqref{mirfallenkeinenamenmehrein} times a factor of size
\begin{equation} \label{enlargement1}
\ll \frac{U}{T} \cdot (\log B)^{\varepsilon}\log U.
\end{equation} 
The second term on the right-hand side of \eqref{insteadbound} equals the term
\begin{equation} \label{thesecondterm}
\frac{\beta^{1/2}}{\alpha}\cdot (1+\ve)^{\omega(b_1)} 
\end{equation}
on the right-hand side of \eqref{mirfallenkeinenamenmehrein} times a factor of size
\begin{equation} \label{enlargement2}
\ll \frac{U^{2/3} Y^{1/2}}{\delta^{1/2}\sigma^{1/3}T}\cdot (\log
B)^{1+ \varepsilon}\log U
=\frac{U^{2/3} Y^{1/2}}{\tau^{1/2}T}\cdot (\log
B)^{1+ \varepsilon}\log U.
\end{equation} 
The term \eqref{thefirstterm} in \eqref{mirfallenkeinenamenmehrein}
gave rise to the first term on the right-hand side of \eqref{T2term},
and the term \eqref{thesecondterm} in
\eqref{mirfallenkeinenamenmehrein} gave rise to the second 
term on the
right-hand side of \eqref{T2term}. Hence $\mathcal{S}_1^{(1)}(U)$ is
majorised by the sum of the same terms, multiplied by the factors in
\eqref{enlargement1} and \eqref{enlargement2}, respectively. In this
way  we
obtain 
\begin{equation} 
\begin{split}
\label{T2Uesti}
\mathcal{S}_1^{(1)}(U) \ll~& 
s^{1/2}s_1^{1/2} q^{1/2} (2+\ve)^{\omega(b)+\omega(q)}
4^{\omega(s)}
(\log B)^{3+\ve}
\left(
  \frac{b^{1/2}YU}{XT}
 +
\frac{a^{1/2}U^{2/3} }{\tau^{1/2}T}
\right) \log U.
\end{split}
\end{equation}

\subsection{Conclusion}\label{s:final}

Combining \eqref{SUsplitting}, \eqref{conto1}, \eqref{conto22},
\eqref{S1Usplit}, \eqref{T1Uestim} and \eqref{T2Uesti}, we get 
\begin{align*}
\mathcal{S}(U) \ll~& 
\frac{\phi(q)}{q^2}\cdot XY\cdot \frac{U}{T}
+
\frac{XY\rad(q)^{1/2}(2+\ve)^{\omega(q)}U^{1/3}\log B}{\tau q T}\\
&+
\frac{(2+\ve)^{\omega(b)} (3+\ve)^{\omega(q)} (\log B)^{2+\ve}}{
b^{1/4}  q^{1/2}T }
\left(
\frac{a^{1/2}XY U^{1/3}}{\tau}   
+\frac{a^{1/4}b^{1/4}X^{1/2}Y U^{5/6}}{\tau^{1/4}}\right)\\
&+
s^{1/2}s_1^{1/2} q^{1/2} (2+\ve)^{\omega(b)+\omega(q)}
4^{\omega(s)}
(\log B)^{3+\ve}
\left(
  \frac{b^{1/2}YU}{XT}
 +
\frac{a^{1/2}U^{2/3} }{\tau^{1/2}T}
\right)\log U.
\end{align*}
In view of \eqref{sigtaudef}, \eqref{x_0y_0XY}, \eqref{Tcondi} and \eqref{sidelengths}, we deduce that
\begin{equation} \label{combin1}
\begin{split}
\mathcal{S}(U) \ll~& 
\frac{\varphi(q)}{q}\cdot
\frac{B^{5/6}TU^{2/3}}{a^{1/2}b^{1/3}q^{1/6}}
+f_\ve(a,b;q)B^{1/2}T^{1/2}U^{1/2}(\log B)^{2+\ve}
\\
&+
\frac{a^{1/2}b^{1/6} s^{1/2}s_1^{1/2} q^{1/3} (2+\ve)^{\omega(q)+\omega(b)}
4^{\omega(s)}
(\log B)^{3+\ve} U^{2/3}\log U}{B^{1/6}T},
\end{split}
\end{equation}
where
$f_\ve(a,b;q)$ is given by \eqref{eq:stone}.

Now we set 
$$
Z:=\frac{a^{1/2}(bqss_1)^{1/4}}{B^{1/2}}.
$$
We wish to balance the first and the last term on the right-hand side
 of \eqref{combin1}.
To this end, we choose
$$
T:=\max\left\{U^{-1},\min\{U^{1-\ve},Z\}\right\}.
$$
This is clearly in accordance with \eqref{Tcondi} and leads to the
conclusion in \eqref{combin1}  that
\begin{align*}
\mathcal{S}(U) \ll~& 
\frac{\varphi(q)}{q}\cdot
\frac{B^{5/6}U^{-1/3}}{a^{1/2}b^{1/3}q^{1/6}}
+f_\ve(a,b;q) B^{1/2}(\log B)^{2+\ve}\left(
1+\frac{a^{1/4}(b q s s_1)^{1/8}U^{1/2}}{B^{1/4}}\right)
\\
&+
\frac{a^{1/2}b^{1/6} (ss_1)^{1/2} q^{1/3}(2+\ve)^{\omega(b)}
  (2+\ve)^{\omega(q)} 
4^{\omega(s)}
(\log B)^{3+\ve}\log U}{B^{1/6}U^{1/3-\ve}}\\
&+
\frac{(ss_1)^{1/4} q^{1/12} (2+\ve)^{\omega(b)+\omega(q)}
4^{\omega(s)}
B^{1/3}(\log B)^{4+\ve} U^{2/3}}{b^{1/12}}.
\end{align*}
Recalling \S \ref{s:R2}, and in particular
\eqref{Sregion}, \eqref{intincluded}
and the need to sum over the $O(\log B)$ dyadic intervals for $U$
satisfying   \eqref{Srange*},  we deduce that
\begin{equation} \label{combin2ende}
\begin{split}
\mathcal{S}(\mathcal{R}_2;a,b;q)
\ll~&
\frac{\varphi(q)}{q}\cdot
\frac{B^{5/6}L^{-1/6}}{a^{1/2}b^{1/3}q^{1/6}}+
f_\ve(a,b;q)  B^{1/2}(\log B)^{3+\ve}\\
&+
f_\ve(a,b;q) a^{1/4}(b q s s_1)^{1/8} B^{1/4}(\log
B)^{2+\ve}\UU^{1/2} 
\\
&+
a^{1/2}b^{1/6} (ss_1)^{1/2} q^{1/3} (2+\ve)^{\omega(b)}
(2+\ve)^{\omega(q)}
4^{\omega(s)}
B^{-1/6}(\log B)^{3+\ve}\\
&+
\frac{(ss_1)^{1/4} q^{1/12}(2+\ve)^{\omega(b)+\omega(q)}
4^{\omega(s)}
B^{1/3}(\log B)^{4+\ve} \UU^{2/3}}{b^{1/12}}.
\end{split}
\end{equation}
On combining \eqref{S123split}, \eqref{Sestfinal2komplett} and
\eqref{combin2ende}, and taking \eqref{lowerboundarea} into account,
we get a final estimate for 
$\mathcal{S}(\mathcal{R};a,b;q)$. 
In view of 
\eqref{Srewrite} and 
\eqref{Umathcaldef}
we may now record our final asymptotic
formula for $M(B,\boldsymbol{X},\boldsymbol{Y};a,b;q)$ in the
following result. 

\begin{thm}\label{t:final}
Let $\ve>0$, let $B\geq e^e$, $\XX,\YY\geq 1$ and assume that 
$(ab,q)=1$ are such that \eqref{sizecond} holds. 
Let $\rad(q),  s,s_1$ be
given by \eqref{eq:s-s1} and recall the definition 
\eqref{sigtaudef} of $\sigma, \tau$.
Suppose that
$$
\frac{\boldsymbol{X}}{\sigma}\ge 1, \quad
\frac{\boldsymbol{Y}}{\tau}\ge 1, 
$$
and let
$\mathcal{R}$ be given by \eqref{region}.
Then we have
$$
M(B,\boldsymbol{X},\boldsymbol{Y};a,b;q) = \frac{\varphi(q)}{q}
\cdot\frac{B^{5/6} \Area(\mathcal{R})}{a^{1/2}b^{1/3}q^{1/6}}
\left(1+O\left(\frac{1}{(\log\log B)^{1/6}}\right)\right) +O\left(
\sum_{i=1}^4 F_i\right), 
$$
where 
if $\UU=\min\left\{\boldsymbol{X}/\sigma,  
\left(\boldsymbol{Y}/\tau\right)^{3/2}\right\}$
and $f_\ve(a,b;q)$ is given by \eqref{eq:stone}, then 
\begin{align*}
F_1&:=
f_\ve(a,b;q)  B^{1/2}(\log B)^{3+\ve},\\
F_2&:=
f_\ve(a,b;q) a^{1/4}(b q s s_1)^{1/8} B^{1/4}(\log
B)^{2+\ve}\UU^{1/2},\\
F_3&:=
a^{1/2}b^{1/6} (ss_1)^{1/2} q^{1/3} (2+\ve)^{\omega(b)}(2+\ve)^{\omega(q)}
4^{\omega(s)}
B^{-1/6}(\log B)^{3+\ve},\\
F_4&:=
b^{-1/12}(ss_1)^{1/4} q^{1/12} (2+\ve)^{\omega(b)+\omega(q)}
4^{\omega(s)}
B^{1/3}(\log B)^{4+\ve} \UU^{2/3}.
\end{align*}
\end{thm}

\section{Proof of Theorem \ref{main}}

The remaining sections of this paper are concerned with the proof of
Theorem~\ref{main}. 
Let $U=X\setminus\{x_0=x_1=0\}$ be the distinguished open subset of our degree $2$
del Pezzo surface 
 $X$ in \eqref{eq:X}.  Then 
we see that the counting function is
\begin{equation}\label{eq:vic}
N_U(B)
=\frac{1}{2}\#\left\{
\x\in\ZZ^4: 
\begin{array}{l}
x_0^2=x_1x_2^3+x_1^3x_3, ~x_1\neq 0,\\
(x_1,x_2,x_3)=1,\\
\mbox{$|x_i|\leq B$ for $1\leq i\leq 3$},~|x_0|\leq B^2
\end{array}
\right\},
\end{equation}
on taking into account that $-\x$ and $\x$ represent the same point in
$\PP(2,1,1,1)$.

\subsection{Passage to the universal torsor}\label{s:passage}

We begin by establishing an explicit bijection
between the rational points which have to be counted on $U$ 
and the integral points on the
universal torsor above $\widetilde{X}$ which are subject to a number
of coprimality conditions.   We begin by simplifying the expression
for $N_U(B)$ in \eqref{eq:vic}. 
The contribution from solutions for which $x_0=0$ is bounded by the
number of points of height at most $B$ on the plane cubic
$x_2^3+x_1^2x_3=0.$ But the solutions of this equation are
parametrised by $(x_1,x_2,x_3)=(s^3,-s^2t, t^3)$ and so such points
contribute $O(B^{2/3})$ to $N_U(B)$. 
Turning to solutions with $x_0x_1\neq 0$ we note that 
$x_1(x_2^3+x_1^2x_3)>0$ 
in any solution to be counted. Thus we may assume without loss of
generality that $x_0,x_1>0$ on multiplying the final count by
$4$. We record the outcome of these manipulations in the following
result. 

\begin{lemma}\label{lem:prelim}
We have
$$
N_U(B)
=2\#\left\{
\x\in\ZZ^4: 
\begin{array}{l}
x_0^2=x_1x_2^3+x_1^3x_3, ~x_0,x_1>0, \\
(x_1,x_2,x_3)=1,\\
\mbox{$|x_i|\leq B$ for $1\leq i\leq 3$},~|x_0|\leq B^2
\end{array}
\right\}+O(B^{2/3}).
$$
\end{lemma}

Our goal is to 
pass from Lemma \ref{lem:prelim}
to a counting function that 
involves counting suitably restricted integer points on the
associated universal torsor.
In the present setting the universal torsor is an open subset
of the hypersurface \eqref{count} 
in $\AA^{11}$. 
Let $\bet=(\eta_1,\ldots, \eta_8)$ and $\bal=(\al_1,\al_2,\al_3)$. 
We let $\mathcal{N}(B)$ denote the number of integral solutions
$(\bet,\bal)\in\ZZ^{11}$
of the equation \eqref{count}
satisfying
\begin{equation} \label{etasgr0}
\eta_1,...,\eta_8,\alpha_2\ge 1,
\end{equation} 
together with the height conditions
\begin{equation} \label{heights1} 
\begin{cases}
\left| \eta_1^2\eta_2^4\eta_3^6\eta_4^5\eta_5^4\eta_6^3\eta_7^3\eta_8^2\right| \le B, \quad
\left| \eta_1^2\eta_2^3\eta_3^4\eta_4^3\eta_5^2\eta_6 \eta_7^2 \alpha_1\right| \le B, \\
\left|\alpha_3\right| \le B, \quad
\left|
  \eta_1^3\eta_2^6\eta_3^9\eta_4^7\eta_5^5\eta_6^3\eta_7^5\eta_8 \alpha_2\right| \le B^2, 
\end{cases}
\end{equation}
and the coprimality conditions
\begin{equation} \label{coprimeconds1}
(\alpha_1,\eta_2\cdots \eta_8)=(\alpha_2,\eta_1\cdots \eta_6 \eta_8)=1,
\end{equation} 
\begin{equation} \label{alpha3}
(\alpha_3,\eta_1\cdots \eta_7)= 1,
\end{equation}
and
\begin{equation} \label{coprimeconds2} 
\begin{cases}
(\eta_8,\eta_1\cdots \eta_5\eta_7)=(\eta_7,\eta_1\eta_2\eta_4 \eta_5
\eta_6)=(\eta_6,\eta_1\cdots
\eta_4)=1, \\
(\eta_5,\eta_1\eta_2\eta_3)=(\eta_4,\eta_1\eta_2)=(\eta_3,\eta_1)
= 1.  
\end{cases}
\end{equation}
We will establish the following result.

\begin{lemma}\label{lem:torsor}
We have $N_U(B)=2\mathcal{N}(B)+O(B^{2/3})$. 
\end{lemma}

\begin{proof}
Let $a,b\in\NN$. During the course of our argument we will make repeated use of the
following two facts.  Firstly $a\mid b^2$ if and only if 
$$
a=ha'^2, \quad b=ha'b', \quad (a',b')=1,
$$
and secondly, 
$a\mid b^3$ if and only if 
$$
a=hk^2a'^3, \quad b=hka'b', \quad (a',h)=(a'k,b')=1.
$$
The second fact is readily deduced from the first after drawing out
the greatest common divisor of $a$ and $b$. To establish the first
fact we set $h'=(a,b)$ and $a=h'a'$, $b=h'b'$ for coprime
$(a',b')=1$. But then $a'\mid h'$ and the claim follows on
writing $h'=a'h$ for $h\in \NN$.

In what follows we set $\al_3=x_3$.  For any $\x$ counted by $N_U(B)$
we have $x_1\mid x_0^2$. Hence we may write
$$
x_0=hx_0'\eta_8, \quad x_1=h\eta_8^2, \quad (x_0',\eta_8)=(h\eta_8^2,x_2,\al_3)=1,
$$
for $h,x_0',\eta_8\in\NN$. Once the substitution is made we obtain the
new equation
$hx_0'^2=x_2^3+h^2\eta_8^4\al_3$. It now follows that $h\mid x_2^3$
and so we can write
$$
h=jk^2\eta_6^3, \quad x_2=jk\eta_6\al_1, 
$$
with $j,k,\eta_6\in\NN$ such that $\al_1\neq 0$ and 
$
(\eta_6,j)=(k\eta_6,\al_1)=1,
$
in conjunction with the previous coprimality conditions.  Substituting
back into the equation yields
$x_0'^2=j^2k\al_1^3+jk^2\eta_6^3\eta_8^4\al_3$.  
Writing $\ell=(j,k,x_0')$ and noting that $\ell^3\mid x_0'^2$, so
that $\ell\mid (x_0'/\ell)^2$, we may therefore make the change of
variables
$$
j=m\eta_3^2j', \quad
k=m\eta_3^2k', \quad
x_0'=m^2\eta_3^3x_0'', 
$$
for $m,\eta_3,j',k',x_0''\in\NN$ satisfying 
$(j',k',m\eta_3x_0'')=(\eta_3,x_0'')=1$, in addition to the
previous conditions. 
This substitution now leads to the new equation
$mx_0''^2=j'^2k'\al_1^3+j'k'^2\eta_6^3\eta_8^4\al_3$.  
We are now led to the change of variables
$$
j'=\eta_2j'', \quad
k'=\eta_4k'', \quad
m=\eta_2\eta_4\eta_7, 
$$
for $\eta_2,\eta_4,\eta_7,j'',k''\in\NN$ with $(j''k'',\eta_7)=1$,
together with all the previous coprimality conditions. This yields the
equation
$\eta_7x_0''^2=\eta_2j''^2k''\al_1^3+\eta_4j''k''^2\eta_6^3 
\eta_8^4\al_3$.   
Now $j''k''\mid \eta_7x_0''^2$ and so
$j''k''\mid x_0''^2$. 
Moreover we have $(j'',k'')=1$ since clearly 
$(j'',k'',x_0'')=1$.  We are now led to 
make our final change of variables
$$
j''=u\eta_1^2, \quad
k''=v\eta_5^2, \quad
x_0''=uv\eta_1\eta_5\al_2,
$$
for $u,v,\eta_1,\eta_5,\al_2\in\NN$
such that
$(u\eta_1,v\eta_5)=(uv\eta_1\eta_5,\al_2)=1$. Substituting
this back in gives the equation
$uv\eta_7\al_2^2=u\eta_1^2\eta_2\al_1^3+v\eta_4\eta_5^2\eta_6^3
\eta_8^4\al_3$.   But then $u\mid \eta_4\eta_5^2\eta_6^3
\eta_8^4\al_3$ and $v\mid \eta_1^2 \eta_2\al_1^3$. Consideration of the
coprimality conditions leads to the conclusion that $u=v=1$, and the
equation simply becomes the universal torsor recorded in \eqref{count}. 

Tracing through our argument one sees that we have made the
transformation
\begin{align*}
x_0&=
\eta_1^3\eta_2^6\eta_3^9\eta_4^7\eta_5^5\eta_6^3\eta_7^5\eta_8\alpha_2,\\
x_1&=\eta_1^2\eta_2^4\eta_3^6\eta_4^5\eta_5^4\eta_6^3\eta_7^3\eta_8^2,\\
x_2&=\eta_1^2\eta_2^3\eta_3^4\eta_4^3\eta_5^2\eta_6\eta_7^2 \alpha_1,\\
x_3&=\alpha_3,
\end{align*}
with $(\bet,\bal)\in\ZZ^{11}$ satisfying 
\eqref{etasgr0}. This transformation 
leads directly to the height conditions in \eqref{heights1}.
Furthermore, a straightforward calculation shows that once taken in
the light of \eqref{count} the final coprimality conditions are given
by  \eqref{coprimeconds1}, \eqref{alpha3} and \eqref{coprimeconds2}.

Finally it is trivial to check that through the above transformation any 
point solution $(\bet,\bal)\in\ZZ^{11}$ of \eqref{count} satisfying 
\eqref{etasgr0}--\eqref{coprimeconds2}
produces a vector $\x\in\ZZ^4$ that will be counted  by
$N_U(B)$. This completes the proof of the lemma. 
\end{proof}

\subsection{Counting integral points on the universal torsor}

We aim to establish an asymptotic formula for the quantity
$\mathcal{N}(B)$ defined in \S \ref{s:passage}.
Accordingly we will assume that $B$ is large in all that follows. 
To simplify the situation, we define
\begin{equation} \label{X0X1X2}
\begin{split}
X_0&:=\left(\frac{\eta_1^2\eta_2^4\eta_3^6\eta_4^5\eta_5^4\eta_6^3\eta_7^3\eta_8^2}{B}
\right)^{1/3},\quad 
X_1:=\left(\frac{B\eta_4\eta_5^2\eta_6^3\eta_8^4}{\eta_1^2\eta_2}\right)^{1/3},\\
X_2&:=\left(B\eta_1\eta_2^2\eta_3^3\eta_4^4\eta_5^5\eta_6^6\eta_8^7\right)^{1/3}. 
\end{split}
\end{equation}
Then one sees that the height conditions \eqref{etasgr0} and
\eqref{heights1} can be reformulated as
$$
0<X_0 \le 1, \quad \left|\alpha_1\cdot \frac{X_0^2}{X_1} \right| \le
1, \quad 0<\alpha_2\cdot \frac{X_0^5}{X_2}\le 1, \quad 
|\alpha_3|\leq B,
$$
with $\eta_1,\ldots,\eta_8\geq 1$. 

Once taken together with \eqref{count} it is easy to see that 
the coprimality conditions
\eqref{coprimeconds1}--\eqref{coprimeconds2} are equivalent to the
same conditions, but with \eqref{alpha3} replaced by
$$
(\alpha_3,\eta_3\cdots \eta_6)= 1.
$$
Let us write $S(B,\XX,\YY;a,b;q)$ for the set underpinning the
cardinality introduced in \eqref{eq:Mq}.
Then it follows from
M\"obius inversion that
$$
\mathcal{N}(B)=\sum_{\substack{\bet\in \NN^8\\
{\mbox{\scriptsize{\eqref{coprimeconds2} holds}}}\\ X_0\leq 1
}}\sum_{k_3\mid \eta_3\cdots \eta_6}\mu(k_3)N_{k_3},
$$
where
$$
N_{k_3}
:=\#\left\{(x,y)\in
S\left(\frac{B}{k_3},\frac{X_2}{X_0^5},\frac{X_1}{X_0^2};\eta_7,\eta_1^2\eta_2;
k_3\eta_4\eta_5^2\eta_6^3\eta_8^4\right):\begin{array}{l}
(x,\eta_1\eta_2\eta_3)=1,\\
(y,\eta_2\eta_3\eta_7)=1\end{array}
\right\}.
$$
Note that $x=\al_2$ and $y=\al_1$ in this correspondence. The
coprimality condition  
$(k_3,xy)=1$ is automatic, since $k_3\mid \eta_3\cdots \eta_6$ and 
$(xy,\eta_3\cdots\eta_6)=1$. Furthermore, we clearly have 
$(\eta_1^2\eta_2\eta_7,
k_3\eta_4\eta_5^2\eta_6^3\eta_8^4)=1$. In particular
$N_{k_3}=0$ unless
$(k_3,\eta_1\eta_2\eta_7)=1$. Since $k_3\mid \eta_3\cdots \eta_6$
and $(\eta_1,\eta_3\cdots \eta_6)=1$ we see that $k_3$ is
automatically coprime to $\eta_1$, whence
$$
\mathcal{N}(B)=\sum_{\substack{\bet\in \NN^8\\
{\mbox{\scriptsize{\eqref{coprimeconds2} holds}}}\\ X_0\leq 1
}}
\sum_{\substack{k_3\mid \eta_3\cdots \eta_6\\
(k_3,\eta_2\eta_7)=1}} 
\mu(k_3) N_{k_3}.
$$
A further application of M\"obius inversion
therefore leads to the expression
\begin{equation}
\begin{split}
  \label{eq:main}
  \mathcal{N}(B)
=~&\sum_{\substack{\bet\in \NN^8\\
{\mbox{\scriptsize{\eqref{coprimeconds2} holds}}}\\ X_0\leq 1
}}
\sum_{\substack{k_3\mid \eta_3\cdots \eta_6\\
(k_3,\eta_2\eta_7)=1}} \mu(k_3)
\sum_{\substack{k_1\mid \eta_2\eta_3\eta_7\\
k_2\mid \eta_1\eta_2\eta_3\\
(k_1k_2,k_3\eta_4\eta_5\eta_6\eta_8)=1
}}
\mu(k_1)\mu(k_2) N_{k_1,k_2,k_3},
\end{split}
\end{equation}
where
$$
N_{k_1,k_2,k_3}:=
M\left(\frac{B}{k_3},\frac{X_2}{k_2X_0^5},
\frac{X_1}{k_1X_0^2};k_2^2\eta_7,k_1^3\eta_1^2\eta_2; 
k_3\eta_4\eta_5^2\eta_6^3\eta_8^4\right),
$$
in the notation of \eqref{eq:Mq}.

Thus we are led to apply Theorem \ref{t:final} with $B$ replaced by
$B/k_3$,  
$$
  a=k_2^2\eta_7 \in \NN, \quad b=k_1^3\eta_1^2\eta_2 \in \NN,\quad q=
k_3\eta_4\eta_5^2\eta_6^3\eta_8^4 \in \NN,
$$ 
and
\begin{equation} \label{nochnochwas}
\XX=\frac{X_2}{k_2X_0^5}\geq 1, \quad \YY=\frac{X_1}{k_1X_0^2}\geq 1
\end{equation}
One readily checks that $(ab,q)=1$ and the inequalities in \eqref{nochnochwas} are satisfied by our height conditions. Moreover,  
consultation with \eqref{sigtaudef} and \eqref{X0X1X2} reveals that
\begin{align*}
\frac{\boldsymbol{X}}{\sigma}
&=
\frac{X_2}{k_2X_0^5} 
\left(\frac{k_2^2\eta_7 }{\eta_4\eta_5^2\eta_6^3\eta_8^4 B}\right)^{1/2}
=X_0^{-9/2}\\
\frac{\boldsymbol{Y}}{\tau}
&=\frac{X_1}{k_1X_0^2} 
\left(\frac{k_1^3\eta_1^2\eta_2}{\eta_4\eta_5^2\eta_6^3\eta_8^4
    B}\right)^{1/3}=X_0^{-2}.
\end{align*}
It follows that 
$$
\frac{\XX}{\sigma}\geq \left(\frac{\YY}{\tau}\right)^{3/2}= X_0^{-3}\geq 1,
$$
since $X_0\leq 1$. Hence $\UU=X_0^{-3}$.
Thus all the hypotheses of Theorem \ref{t:final} are
met and we are therefore free to apply this result. 
To do so we note from \eqref{eq:s-s1} that in the present setting we have 
$$
s\mid \frac{k_3\eta_4}{(k_3,\eta_6)}, \quad
s_1\mid  \frac{k_3\eta_4\eta_6}{(k_3,\eta_6)},
$$
whence in particular
$$
ss_1\leq \frac{k_3^2\eta_4^2 \eta_6}{(k_3,\eta_6)^2}.
$$
Furthermore \eqref{eq:stone} yields
$$
f_\ve(a,b;q)\ll (\eta_2\cdots \eta_7)^\ve \left(
\frac{(2+\ve)^{\omega(\eta_8)}}{k_2}+ 
\frac{(3+\ve)^{\omega(\eta_8)}}{
\eta_1^{1/2-\ve}}\right).
$$

Let $\phi^*(q)=\phi(q)/q$ be the function introduced in \S
\ref{s:arith} and define
\begin{equation}
  \label{eq:gamma}
\gamma(\bet):=
\sum_{\substack{k_3\mid \eta_3\cdots \eta_6\\
(k_3,\eta_2\eta_7)=1}} 
\frac{\mu(k_3)}{k_3}
\sum_{\substack{k_1\mid \eta_2\eta_3\eta_7\\
k_2\mid \eta_1\eta_2\eta_3\\
(k_1k_2,k_3\eta_4\eta_5\eta_6\eta_8)=1 
}}
\frac{\mu(k_1)}{k_1}\cdot \frac{\mu(k_2)}{k_2}
\cdot \phi^{*}(k_3\eta_4\eta_5^2\eta_6^3\eta_8^4)
\end{equation}
if \eqref{coprimeconds2} holds and $\gamma(\bet)=0$ otherwise.
For $1\leq i\leq 4$ let $\mathcal{F}_i(B)$ denote the overall 
contribution from the term $F_i$ in Theorem~\ref{t:final} 
once inserted into \eqref{eq:main}. Then we have 
\begin{equation}
  \label{eq:final'}
  \mathcal{N}(B)=\mathcal{M}(B)\left(1+O\left(\frac{1}{(\log\log
      B)^{1/6}}\right)
\right)+O\left(\sum_{i=1}^4
\mathcal{F}_i(B)
\right), 
\end{equation}
where 
\begin{equation}
  \label{eq:MM}
\mathcal{M}(B):=
\sum_{\substack{\bet\in \NN^8\\
X_0\leq 1
}}\frac{
\Area(\mathcal{R})
B^{5/6}\gamma(\bet)}{\eta_1^{2/3}\eta_2^{1/3}\eta_4^{1/6}\eta_5^{1/3}\eta_6^{1/2}\eta_7^{1/2}\eta_8^{2/3}}
\end{equation}
and $\mathcal{R}$ is given by \eqref{region}.

\bigskip

Let us now turn to the estimation of 
$$
\mathcal{F}_i(B)\ll 
\sum_{\substack{\bet\in \NN^8\\
X_0\leq 1
}}
\sum_{\substack{
k_1\mid \eta_2\eta_3\eta_7\\
k_2\mid \eta_1\eta_2\eta_3\\
k_3\mid \eta_3\cdots \eta_6}}
|\mu(k_1)\mu(k_2)\mu(k_3)| F_i,
$$
for $1\leq i\leq
4$. Note that $F_i$ depends on $k_1,k_2,k_3$. During the course of these arguments we will make use of
the estimates from \S \ref{s:arith}. 

Beginning with the case $i=1$, we 
clearly have 
\begin{align*}
\mathcal{F}_1(B)
&\ll B^{1/2}(\log B)^{3+\ve}
\sum_{\substack{\bet\in \NN^8\\
X_0\leq 1
}}
\sum_{\substack{
k_2\mid \eta_1\eta_2\eta_3}}
|\mu(k_2)|(\eta_2\cdots \eta_7)^\ve \left(
\frac{(2+\ve)^{\omega(\eta_8)}}{k_2}+ 
\frac{(3+\ve)^{\omega(\eta_8)}}{
\eta_1^{1/2-\ve}}\right)\\
&\ll B^{1/2}(\log B)^{3+\ve}
\sum_{\substack{\bet\in \NN^8\\
X_0\leq 1
}}
(\eta_2\cdots \eta_7)^\ve \left(
(2+\ve)^{\omega(\eta_8)}
+\frac{(3+\ve)^{\omega(\eta_8)}}{
\eta_1^{1/2-\ve}}\right).
\end{align*}
On summing first over $\eta_8$ and then over $\eta_1$ we therefore
obtain $\mathcal{F}_1(B)\ll B (\log B)^{5+\ve}$, 
which is satisfactory. 

For the case $i=2$ 
we first note that 
\begin{align*}
F_2\ll~&
(\eta_2\cdots \eta_7)^\ve \left(
\frac{(2+\ve)^{\omega(\eta_8)}}{k_2}+ 
\frac{(3+\ve)^{\omega(\eta_8)}}{
\eta_1^{1/2-\ve}}\right)  
(k_2^2\eta_7)^{1/4} 
(k_1^3\eta_1^2\eta_2)^{1/8}\\
&\times 
(k_3\eta_4\eta_5^2\eta_6^3\eta_8^4)^{1/8}
 (k_3^2\eta_4^2 \eta_6)^{1/8}  \left(\frac{B}{k_3}\right)^{1/4}
(\log B)^{2+\ve}X_0^{-3/2}\\
=~&  k_1^{3/8}k_2^{1/2}k_3^{1/8}
(\eta_2\cdots \eta_7)^\ve \left(
\frac{(2+\ve)^{\omega(\eta_8)}}{k_2}+ 
\frac{(3+\ve)^{\omega(\eta_8)}}{
\eta_1^{1/2-\ve}}\right) \\
&\times 
\frac{B^{3/4}(\log B)^{2+\ve}}{
\eta_1^{3/4}\eta_2^{15/8}\eta_3^{3}\eta_4^{17/8}\eta_5^{7/4}\eta_6\eta_7^{5/4}\eta_8^{1/2}
}.
\end{align*}
Recall that $k_1,k_2,k_3$ do not depend on $\eta_8$. Summing 
over $\eta_8\leq
B^{1/2}/(\eta_1\eta_2^2\eta_3^3\eta_4^{5/2}\eta_5^2\eta_6^{3/2}\eta_7^{3/2})$,
we therefore find that 
\begin{align*}
\sum_{\eta_8} F_2\ll  B(\log B)^{4+\ve} 
\cdot \frac{ k_1^{3/8}k_2^{1/2}k_3^{1/8}
(\eta_2\cdots \eta_7)^\ve}{
\eta_1^{5/4}\eta_2^{23/8}\eta_3^{9/2}\eta_4^{27/8}\eta_5^{11/4}\eta_6^{7/4}
\eta_7^{2}}
 \left(
\frac{1}{k_2}+ \frac{1}{\eta_1^{1/2-\ve}}\right).
\end{align*}
It is now trivial to confirm that 
 $\mathcal{F}_2(B)\ll B (\log B)^{4+\ve}$, 
which is also satisfactory.

For the case $i=3$
we begin by noting that 
\begin{align*}
F_3
\ll~&
(\eta_2\cdots \eta_7)^\ve
 (k_2^2\eta_7)^{1/2}
 (k_1^3\eta_1^2\eta_2)^{1/6}  \left(
\frac{k_3^2\eta_4^2 \eta_6}{(k_3,\eta_6)^2}\right)^{1/2} \\
&\times
(k_3\eta_4\eta_5^2\eta_6^3\eta_8^4)^{1/3} 
(2+\ve)^{\omega(\eta_1)+\omega(\eta_8)}
 (B/k_3)^{-1/6}(\log B)^{3+\ve}\\
=~& \frac{k_1^{1/2}k_2 k_3^{3/2}}{(k_3,\eta_6)}\cdot 
(\eta_2\cdots \eta_7)^\ve  
(2+\ve)^{\omega(\eta_1)+\omega(\eta_8)}\\
&\times 
\frac{\eta_1^{1/3}\eta_2^{1/6}\eta_4^{4/3}\eta_5^{2/3}\eta_6^{3/2}\eta_7^{1/2}
\eta_8^{4/3} (\log B)^{3+\ve}}{B^{1/6}}.
\end{align*}
Summing first over $\eta_8$ we conclude that 
\begin{align*}
\sum_{\eta_8} F_3
&\ll 
B(\log B)^{4+\ve}\cdot 
\frac{k_1^{1/2}k_2 k_3^{3/2}}{(k_3,\eta_6)}\cdot 
\frac{(\eta_2\cdots \eta_7)^\ve (2+\ve)^{\omega(\eta_1)}}{
\eta_1^{2}\eta_2^{9/2}\eta_3^7\eta_4^{9/2}\eta_5^{4} \eta_6^{2}
\eta_7^{3}}\\
&\ll 
B(\log B)^{4+\ve}\cdot 
\frac{k_2(\eta_2\cdots \eta_7)^\ve (2+\ve)^{\omega(\eta_1)}}{
\eta_1^2 \eta_2^{4}\eta_3^{5}\eta_4^{3}\eta_5^{5/2} \eta_6^{3/2} \eta_7^{5/2}},
\end{align*}
since 
$k_1\leq \eta_2\eta_3\eta_7$ and 
$k_3^{3/2}/(k_3,\eta_6)\leq (\eta_3\eta_4\eta_5)^{3/2}\eta_6^{1/2}$.
It therefore follows from \eqref{eq:ape} that 
\begin{align*}
\mathcal{F}_3(B)
&\ll 
B(\log B)^{4+\ve}
\sum_{\eta_1,\ldots,\eta_7\leq B}
\frac{(\eta_2\cdots \eta_7)^\ve (2+\ve)^{\omega(\eta_1)}}{
\eta_1^2 \eta_2^{4}\eta_3^{5}\eta_4^{3}\eta_5^{5/2} \eta_6^{3/2}
\eta_7^{5/2}} 
\sum_{k_2\mid \eta_1\eta_2\eta_3} |\mu(k_2)|k_2\\
&\ll 
B(\log B)^{4+\ve}
\sum_{\eta_1\leq B}
\frac{(2+\ve)^{\omega(\eta_1)}\sigma_1(\eta_1)}{
\eta_1^2}\\
&\ll 
B(\log B)^{6+\ve},
\end{align*}
which is  satisfactory.

Finally for $i=4$ we note that 
\begin{align*}
F_4\ll ~&
(\eta_2\cdots \eta_7)^\ve
 (k_1^3\eta_1^2\eta_2)^{-1/12} 
\left(
\frac{k_3^2\eta_4^2 \eta_6}{(k_3,\eta_6)^2}\right)^{1/4}
(k_3\eta_4\eta_5^2\eta_6^3\eta_8^4)^{1/12}\\
&\times (2+\ve)^{\omega(\eta_1)}(2+\ve)^{\omega(\eta_8)}
 \left(\frac{B}{k_3}\right)^{1/3}(\log B)^{4+\ve} X_0^{-2}\\
\ll ~& \frac{(\eta_2\cdots \eta_7)^\ve 
(2+\ve)^{\omega(\eta_8)}(2+\ve)^{\omega(\eta_1)} B(\log B)^{4+\ve}}{
\eta_1^{3/2}\eta_2^{11/4}\eta_3^{15/4}\eta_4^{5/2}\eta_5^{9/4}\eta_6^{3/2}
\eta_7^{2}\eta_8},
\end{align*}
on taking  $k_3/(k_3,\eta_6)\leq \eta_3\eta_4\eta_5$. 
On summing first over $\eta_8$ one is readily led to the conclusion
that $\mathcal{F}_4(B)\ll B (\log B)^{6+\ve}$, which is
satisfactory. 

Recalling \eqref{eq:final'}, we may  summarise our investigation so
far in the following result.

\begin{lemma}\label{lem:treated_error_1}
We have 
$$
  \mathcal{N}(B)=\mathcal{M}(B)\left(1+O\left(\frac{1}{(\log\log
      B)^{1/6}}\right)
\right)+O\left(B (\log B)^{6+\ve}
\right), 
$$
as $B\rightarrow \infty$, 
where $\mathcal{M}(B)$ is given by \eqref{eq:MM} and \eqref{eq:gamma}.
\end{lemma}

It remains to estimate $\mathcal{M}(B)$ as $B\rightarrow \infty$. 
Let us begin by considering the factor $\Area(\mathcal{R})$. 
For $u\in (0,1)$ we write 
$$
v(u)
:=\int\int_{\{|s^2+t^3|\le 1, ~0 < s<u^{-9/2},  
~|t|\le u^{-2}\}} \d s \d t.
$$
Then we have $\Area(\mathcal{R})=v(X_0)$.
Writing 
$n=\eta_1^2\eta_2^4\eta_3^6\eta_4^5\eta_5^4\eta_6^3\eta_7^3\eta_8^2$, so
that $X_0=(n/B)^{1/3}$ in \eqref{X0X1X2}, we deduce that
\begin{equation}
  \label{eq:jim}
  \mathcal{M}(B)=
B^{5/6}\sum_{n\leq B} v\left((n/B)^{1/3}\right)\Delta(n),
\end{equation}
where 
$$
\Delta(n):=\sum_{\substack{
\bet\in \NN^8\\
n=\eta_1^2\eta_2^4\eta_3^6\eta_4^5\eta_5^4\eta_6^3\eta_7^3\eta_8^2}}
\frac{\gamma(\bet)n^{1/6}}{\eta_1\eta_2\eta_3\eta_4\eta_5\eta_6\eta_7\eta_8}
$$
and $\gamma$ is given by  \eqref{eq:gamma}.
Let us bring $\gamma$ into a simpler shape under the assumption that
\eqref{coprimeconds2} holds. 
Recall the definition of $\phi^{*}$ from \S \ref{s:arith} and its
properties recorded there.  Carrying out the inner summation in \eqref{eq:gamma}, we get
$$
\gamma(\bet)=\sum\limits_{\substack{k_3|\eta_3\cdots \eta_6\\ (k_3,\eta_2\eta_7)=1}} \frac{\mu(k_3)}{k_3}\cdot \frac{\phi^*(\eta_2\eta_3\eta_7)\phi^*(\eta_1\eta_2\eta_3)
\phi^*(k_3\eta_4\eta_5\eta_6\eta_8)}{\phi^*(\eta_2\eta_3\eta_7,k_3\eta_4\eta_5\eta_6\eta_8) \phi^*(\eta_1\eta_2\eta_3,k_3\eta_4\eta_5\eta_6\eta_8)}.
$$
Since $\phi^*(ab)\phi^*(a,b)=\phi^*(a)\phi^*(b)$ we see that the
summand is 
$$
\frac{\mu(k_3)}{k_3}\cdot \frac{\phi^*(\eta_1\cdots \eta_6\eta_8)\phi^*(\eta_2\cdots \eta_8)}{\phi^*(k_3\eta_4\eta_5\eta_6\eta_8)}.
$$
Applying \eqref{eq:wot}, we obtain 
$$
\gamma(\bet)=
\frac{
\phi^{*}(\eta_2\cdots \eta_8)
\phi^{*}(\eta_1\cdots \eta_6\eta_8)
\phi^{*}(\eta_4\eta_5\eta_6(\eta_3,\eta_2 \eta_7))}
{\phi^{*}(\eta_3\cdots \eta_6\eta_8)}
\prod_{\substack{p\mid \eta_3\\ p\nmid \eta_2\eta_4\cdots \eta_8}}
\left(1-\frac{2}{p}\right)
$$
if \eqref{coprimeconds2} holds and $\gamma(\bet)=0$ otherwise.
Using this it is now straightforward to calculate
the corresponding Dirichlet series 
\[F(s+1/6) = \sum_{n=1}^\infty \frac{\Delta(n)}{n^{s}} =
\sum_{\bet\in \NN^8} \frac{\gamma(\bet)}
{\eta_1^{2s+1}\eta_2^{4s+1}\eta_3^{6s+1}
  \eta_4^{5s+1}\eta_5^{4s+1}\eta_6^{3s+1}\eta_7^{3s+1}\eta_8^{2s+1}},
\] 
which is absolutely convergent for $\Re e (s)>0$ and features
non-negative Dirichlet coefficients.
By multiplicativity we have an Euler product $F(s)= \prod_p
F_{p}(s)$, and one finds that
$F_{p}(s+1/6)=1+(1-1/p)H_p(s)$, with 
\begin{align*}
H_p(s) =
&\frac{p^{-(2s+1)}(1-p^{-(3s+2)})}{(1-p^{-(2s+1)})
  (1-p^{-(3s+1)})} 
+
\frac{p^{-(3s+1)}(1-p^{-(6s+2)})}{(1-p^{-(3s+1)})
  (1-p^{-(6s+1)})}\\
&+
\frac{p^{-(3s+1)}(1-p^{-1})}{(1-p^{-(3s+1)})
  (1-p^{-(4s+1)})}
+
\frac{p^{-(4s+1)}(1-p^{-1})}{(1-p^{-(4s+1)})
  (1-p^{-(5s+1)})}\\
&+
\frac{p^{-(5s+1)}(1-p^{-1})}{(1-p^{-(5s+1)})
  (1-p^{-(6s+1)})}
+
\frac{p^{-(6s+1)}(1-2p^{-1}+p^{-(4s+2)})}{(1-p^{-(6s+1)})
  (1-p^{-(4s+1)})}\\
&+
\frac{p^{-(4s+1)}(1-p^{-1})}{(1-p^{-(4s+1)})
  (1-p^{-(2s+1)})}
+
\frac{p^{-(2s+1)}}{1-p^{-(2s+1)}}.
\end{align*}
Let $\mathbf{k}=(2,4,6,5,4,3,3,2)$.  
Then there exists $\delta_1>0$ and a function $G(s)$ which is
holomorphic in the region $\Re e(s) \ge -\delta_1$ for which 
$$
F(s+1/6)=G(s)\prod_{j=1}^8\zeta(k_js+1).
$$
Thus $F(s)$ has a meromorphic 
continuation to  the half-plane $\Re e(s) \geq 1/6-\delta_1$, 
with a pole of order $8$ at $s=1/6$. 
Moreover it will be useful to note that
\begin{equation}
  \label{eq:G0}
G(0)=\prod_p \left(1-\frac 1 p\right)^8 \left(1+\frac 8 p +\frac
  1{p^2}\right).
\end{equation}
To estimate $\mathcal{M}(B)$ we now have everything in place to apply a
standard Tauberian theorem, such as that 
recorded in work of Chambert-Loir and Tschinkel
\cite[Appendice~A]{perron}, for example. On noting that  
$\lim_{s\rightarrow 1/6} (s-1/6)^8F(s) = G(0)/(k_1\ldots k_8)$, 
we therefore conclude that 
$$
\sum_{n\leq t} \Delta(n)
= \frac{6 G(0)t^{1/6}P(\log t)}{7! \cdot 17280}
  +O(t^{1/6-\delta_2}),
$$
for $t\ge 1$, some $\delta_2>0$ and a monic polynomial $P$ of degree $7$.

Write $c=6G(0)/(
7! \cdot 17280)$ and $E=B^{1/6}(\log B)^6$ for short. Now we may
deduce from an application of partial summation and integration by parts that
\begin{align*}
\sum_{n\leq B}v\left((n/B)^{1/3}\right)\Delta(n)
&=
c\int_1^B v\left((t/B)^{1/3}\right) 
\frac{\d}{\d t} \left(
t^{1/6}(\log t)^7
\right)\d t +O(E)\\
&=\frac{1}{2} \cdot c B^{1/6}(\log B)^7 
\int_0^1 \frac{v(u)}{\sqrt{u}} \d u +O(E).
\end{align*}
Thus, on returning to our expression for $\mathcal{M}(B)$ in
\eqref{eq:jim}, we 
deduce that
\begin{equation}
  \label{eq:ben}
    \mathcal{M}(B)=
\frac{1}{2}\cdot \frac{\omega_\infty G(0)}{87091200}\cdot B (\log B)^7
+O\left(B(\log B)^6 \right),  
\end{equation}
where 
\begin{equation}
  \label{eq:om-inf}
  \omega_\infty=
6\int\int\int_{\{|s^2+t^3|\le 1, ~0 < s<u^{-9/2},  
~0<u<1, ~|t|\le u^{-2}\}} \frac{\d u\d s \d t}{\sqrt{u}}.
\end{equation}

\subsection{Final deduction of Theorem~\ref{main}}

Corralling Lemmas~\ref{lem:torsor} and \ref{lem:treated_error_1} 
with \eqref{eq:ben} we have therefore shown
that 
$$
N_U(B)= 
\frac{\omega_\infty G(0)}{87091200}\cdot
B(\log B)^7\left(1+O\left(\frac{1}{(\log\log B)^{1/6}}\right)\right),
$$
where $G(0)$ is given by \eqref{eq:G0} and
$\omega_\infty$ by \eqref{eq:om-inf}. 
It remains to show that the leading constant achieved in this estimate
agrees with Peyre's prediction \cite{p}.  

Let $\widetilde X$ 
denote a minimal desingularisation of $X$.
According to Peyre we should have 
$c_{X} =\alpha(\widetilde X)\omega_H(\widetilde X)$, 
where
$\alpha(\widetilde X)$ is a constant related to the geometry of 
$X$ and 
$\omega_H(\widetilde X)$
is related to the densities of rational points on $X$ over $\RR$ and 
$\QQ_p$ for all primes $p$. 
We claim that 
$\alpha(\widetilde X)= {1}/{87091200}$
and 
$$
\omega_H(\widetilde X) = \omega_\infty
\prod_p\Big(1-\frac{1}{p}\Big)^8\Big(1+\frac 8 p
+ \frac 1 {p^2}\Big),
$$
with $\omega_\infty$ given by \eqref{eq:om-inf}.
This will clearly suffice to complete the proof of Theorem~\ref{main}. 

Beginning with the local densities
it follows from \cite{p} that 
\begin{align*}
\omega_H(\widetilde X)
&= \lim_{s \to 1}\left((s-1)^{\rank
      \Pic(\widetilde X)} L(s,\Pic(\widetilde X))\right)  
\omega_\infty \prod_p \frac{\omega_p}{L_p(1,\Pic(\widetilde X))}\\
&= \omega_\infty \prod_p \Big(1 - \frac{1}{p}\Big)^8 \omega_p,
\end{align*}
since $L(s, \Pic(\widetilde X)) = \zeta(s)^8$, where $\omega_\infty$
and $\omega_p$ are the real and $p$-adic densities of points on $X$,
respectively. Applying a calculation of Loughran 
\cite[Lemma 2.3]{loughran} we obtain
$$
\omega_p= 1+\frac{8}{p}+\frac{1}{p^2},
$$
as claimed. 
To compute $\omega_\infty$ we parametrise the points by writing $x_3$ as a function of
$x_0,x_1,x_2$ in $f(\x)=x_0^2+x_1x_2^3+x_1^3x_3$.
Since $\x = -\x$ in $\PP(2,1,1,1)$, we may assume $x_0
\ge 0$. On observing that
$\frac{\partial f}{\partial x_3} (\x)= x_1^3$, 
the
Leray form
$\omega_L(\widetilde{X})$ is given by $x_1^{-3} \d x_0 \d x_1 \d x_2$.
Hence 
\[
\omega_\infty=
2\int\int\int_{\{|x_1^{-3}(x_0^2+x_1x_2^3)|\le 1, ~0 \le x_0, x_1
\le 1, ~|x_2|\le 1\}} x_1^{-3} \d x_0 \d x_1 \d x_2.
\]
But then the change of variables
$x_0 = sx_1^{3/2}$, $x_2 = tx_1^{2/3}$ and $x_1 = u^{3}$, easily yields
the value of $\omega_\infty$ given in \eqref{eq:om-inf}.

Turning to the calculation of 
$\alpha(\widetilde X)$ we follow the approach of Derenthal, Joyce and
Teitler \cite{d-j-t}. Thus \cite[Theorem~1.3]{d-j-t} shows that 
$$
\alpha(\widetilde X)=\frac{\alpha(Y)}{\# W(R_{\widetilde{X}})},
$$
where $\alpha(Y)$ is the ``$\alpha$-constant'' associated to the 
split non-singular del Pezzo surface $Y$ of degree $2$ and $W(R_{\widetilde{X}})$
is the Weyl group of the root system $R_{\widetilde{X}}$ whose simple roots
are the $(-2)$-curves on $\widetilde{X}$.  For us the root system is
$\mathbf{E}_7$ by \cite[Remark~5.6]{d-j-t}, whence an application of
\cite[Table~2]{d-j-t} and \cite[Theorem~4]{derenthal} shows that 
$$
\alpha(\widetilde X)=\frac{1}{30}\cdot 
\frac{1}{2^{10}\cdot 3^4\cdot 5\cdot 7}=\frac{1}{87091200},
$$
as required.

\end{document}